\documentclass[a4paper,10pt, reqno]{amsart}

\usepackage{amsmath,amssymb,latexsym, amsfonts}
\usepackage{verbatim}
\usepackage{times}
\usepackage{mathbbol}
\usepackage{hyperref}
\usepackage{mathabx}
\usepackage{bbm}

\usepackage{tikz}

\usepackage[all,ps]{xy}
\usepackage{color}

\usepackage{stmaryrd}

\usepackage{capt-of}

\newcommand{\compactlist}[1]{\setlength{\itemsep}{0pt} \setlength{\parskip}{0pt} \setlength{\leftskip}{-0.#1em}}

\numberwithin{equation}{section}

\theoremstyle{plain}

\newtheorem{theorem}{Theorem}[section]

\newtheorem{prop}[theorem]{Proposition}

\newtheorem{lem}[theorem]{Lemma}
\newtheorem{lemdfn}[theorem]{Lemma and Definition}
\newtheorem{dfnlem}[theorem]{Definition and Lemma}

\newtheorem{cor}[theorem]{Corollary}

\theoremstyle{definition}
\newtheorem{definition}[theorem]{Definition}
\newtheorem{dfn}[theorem]{Definition}
\newtheorem{example}[theorem]{Example}

\newtheorem{rem}[theorem]{Remark}

\newcommand{\ahha}{{\scriptscriptstyle{A}}}
\newcommand{\behhe}{{\scriptscriptstyle{B}}}

\newcommand{\emme}{{\scriptscriptstyle{M}}}
\newcommand{\enne}{{\scriptscriptstyle{N}}}

\newcommand{\uhhu}{{\scriptscriptstyle{U}}}

\newcommand{\vauu}{{\scriptscriptstyle{V}}}

\newcommand{\ga}{\alpha} 
\newcommand{\gb}{\beta}  

\newcommand{\gd}{\delta} 
\newcommand{\gD}{\Delta} 
\newcommand{\gve}{\varepsilon} 
  
\newcommand{\gvf}{\varphi}  

\newcommand{\gl}{\lambda}

\newcommand{\gs}{\sigma} 
\newcommand{\gS}{\Sigma}

\newcommand{\cC}{{\mathcal C}}

\newcommand{\cE}{{\mathcal E}}

\newcommand{\Aut}{\operatorname{Aut}}

\newcommand{\Hom}{\operatorname{Hom}}

\newcommand{\Tor}{{\rm Tor}}
\newcommand{\Ext}{{\rm Ext}}

\newcommand{\Cotor}{\operatorname{Cotor}}

\newcommand{\id}{{\rm id}}

\newcommand{\due}[3]{{}_{{#2 }} {#1}_{{ #3}}\,}    % Zweifachindex

\newcommand{\pl}{\partial}

\newcommand{\rmref}[1]{{\rm (}\ref{#1}{\rm )}}

\newcommand{{\Hl}}{{H^{\ell}}} 
\newcommand{{\mHop}}{{m_{H^{\rm op}}}} 
\newcommand{{\Hop}}{{H^{\rm op}}} 
\newcommand{{\mUop}}{{m_{U^{\rm op}}}} 
\newcommand{{\mUopp}}{{m_{\scriptscriptstyle{U^{\rm op}}}}} 
\newcommand{{\Uop}}{{U^{\rm op}}}
\newcommand{{\mVop}}{{m_{V^{\rm op}}}} 
\newcommand{{\Vop}}{{V^{\rm op}}}  
\newcommand{{\Ae}}{{A^{\rm e}}}
\newcommand{{\Be}}{{B^{\rm e}}}
\newcommand{{\Ue}}{{U^{\rm e}}}
\newcommand{{\He}}{{H^{\rm e}}}
\newcommand{{\Aop}}{{A^{\rm op}}}
\newcommand{{\Aope}}{({A^{\rm op}})^{\rm e}}
\newcommand{{\Aopl}}{{A^{\rm op}_\pl}}

\newcommand{{\Bop}}{{B^{\rm op}}}
\newcommand{{\Bopp}}{{\scriptscriptstyle{{B^{\rm op}}}}}
\newcommand{{\Bope}}{({B^{\rm op}})^{\rm e}}
\newcommand{{\Bpl}}{{B_\pl}}

\newcommand{{\op}}{{{\rm op}}}
\newcommand{{\coop}}{{{\rm coop}}}
\newcommand{{\sop}}{{*^{\rm op}}}
\newcommand{{\co}}{{{\rm co}}}

\newcommand{\kmod}{k\mbox{-}\mathbf{Mod}}                     %
\newcommand{\amoda}{A^{\rm e}\mbox{-}\mathbf{Mod}}                  %
\newcommand{\bmodb}{\mathbf{Mod}\mbox{-}\Be}                     %

\newcommand{\umod}{U\mbox{-}\mathbf{Mod}}                     %  Modul-Kategorien
\newcommand{\modu}{U^\mathrm{op}\mbox{-}\mathbf{Mod}}         %
\newcommand{\yd}{{}^\uhhu_\uhhu\mathbf{YD}}                     %  Modul-Kategorien
\newcommand{\ydr}{\mathbf{YD}{}^\vauu_\vauu} 
\newcommand{\yddual}{\mathbf{YD}^{\scriptscriptstyle{U^*}}_{\scriptscriptstyle{U^*}}}

\newcommand{\comodu}{\mathbf{Comod}\mbox{-}U}         
\newcommand{\ucomod}{U\mbox{-}\mathbf{Comod}}         
\newcommand{\comodv}{\mathbf{Comod}\mbox{-}V}                  
\newcommand{\vcomod}{V\mbox{-}\mathbf{Comod}}

\newcommand{\lact}{\smalltriangleright}                  
\newcommand{\ract}{\smalltriangleleft}
\newcommand{\blact}{\blacktriangleright}  
\newcommand{\bract}{\blacktriangleleft}

\newcommand{{\gog}}{{G \rightrightarrows G_0}}

\newcommand{{\rra}}{\rightrightarrows}

\newcommand{{\lra}}{\ \longrightarrow \ }
\newcommand{{\lla}}{\ \longleftarrow \ }
\newcommand{{\lma}}{\ \longmapsto \ }

\newcommand{{\bull}}{{\scriptscriptstyle{\bullet}}}
\newcommand{{\qqquad}}{{\quad\quad\quad}}
\newcommand{\Aopp}{{\scriptscriptstyle{\Aop}}}
\newcommand{\Aee}{{\scriptscriptstyle{\Ae}}}

\begin{document}

\title{When $\Ext$ is a {B}atalin-{V}ilkovisky algebra} 

\author{Niels Kowalzig}

\begin{abstract}
We show under what conditions the complex computing general $\Ext$-groups carries the structure of a cyclic operad such that $\Ext$ becomes a Batalin-Vilkovisky algebra. This is achieved by transferring cyclic cohomology theories 
for the dual of a (left) Hopf algebroid to the complex in question, which asks for the notion of contramodules introduced along with comodules by Eilenberg-Moore half a century ago. Another crucial ingredient is an explicit formula for the inverse of the Hopf-Galois map on the dual, by which we illustrate recent categorical results and answer a long-standing open question. As an application, we prove that the Hochschild cohomology of an associative algebra $A$ is Batalin-Vilkovisky if $A$ itself is a contramodule over its enveloping algebra $A \otimes \Aop$. This is, for example, the case for symmetric algebras and Frobenius algebras with semisimple Nakayama automorphism. We also recover the construction for Hopf algebras.
\end{abstract}

\address{Dipartimento di Matematica, Universit\`a degli Studi di Roma La
Sapienza, P.le Aldo Moro 5, 00185 Roma, Italia}

\email{niels.kowalzig@uniroma1.it}



\maketitle

\tableofcontents

%\vspace*{-1.0cm}

\section{Introduction}

The notion of higher structures on cohomology groups, more precisely, of Gerstenhaber algebras (consisting of a graded commutative product together with a graded Lie bracket that determines graded inner derivations of the product) and the stronger notion of Batalin-Vilkovisky algebras (a Gerstenhaber algebra with a degree $-1$ differential $B$ that fails to be a graded derivation of the product exactly by the graded Lie bracket), has attracted quite some attention recently; see, for example, 
\cite{Alm:AUAISOBVAWMZVC, BraLaz:HBVAIPG, EuSche:CYFA, GalTonVal:HBVA, Gin:CYA, KauWarZun:TOOOGBBVOAME, Lam:VDBDABVSOCYA, LamZhoZim:THCROAFAWSSNAIABVA, LiuZho:BVSOTCROAGA, Lod:ACOPIHC, Men:BVAACCOHA, Men:CMCMIALAM, KowKra:BVSOEAT,  Kow:BVASOCAPB, Kow:GABVSOMOO, Tra:TBVAOHCIBIIP, Vol:BVDOHCOFA, Wan:SHCAGAS} and references therein. 
A particular focus naturally lies on Hochschild theory: whereas it is a classical result \cite{Ger:TCSOAAR} 
that Hochschild cohomology $H^\bull(A,A)$ of an associative algebra $A$ (over a commutative ring $k$) always carries a Gerstenhaber algebra structure, this structure is not necessarily that of a Batalin-Vilkovisky (BV) algebra: a counterexample of an algebra the Hochschild cohomology of which is not BV can be easily constructed 
by considering a free algebra in two generators \cite{Krae:PC}. However, it is known for the following classes of algebras that $H^\bull(A,A)$ does indeed admit the structure of a BV algebra:
\begin{itemize}
\item
symmetric algebras \cite{Tra:TBVAOHCIBIIP, Men:BVAACCOHA};
\item
Frobenius algebras with semisimple Nakayama automorphism \cite{LamZhoZim:THCROAFAWSSNAIABVA, Vol:BVDOHCOFA};
\item
Calabi-Yau algebras \cite{Gin:CYA};
\item
twisted Calabi-Yau algebras \cite{KowKra:BVSOEAT},
\end{itemize}
and probably more. 
These results were obtained by various different approaches, for example, those in \cite{Gin:CYA, KowKra:BVSOEAT, LamZhoZim:THCROAFAWSSNAIABVA} by passing through a sort of Poincar\'e duality \cite{VdB:ARBHHACFGR, Lam:VDBDABVSOCYA, KowKra:DAPIACT} and using the notion of noncommutative differential calculus \cite{TamTsy:NCDCHBVAAFC}, or more precisely, the notion of a BV {\em module} structure on the respective Hochschild {\em homology} $H_\bull(A,A)$, see also \cite{Kow:GABVSOMOO} for a generalised treatment. 

It would be desirable to have a more direct approach ({\it i.e.}, one that does not use Poincar\'e duality) and in particular one method that covers all cases and yields a sufficient criterion to decide whether $H^\bull(A,A)$ is a BV algebra.

\subsection{Aims and objectives}
\label{aims}
The aim of this paper is threefold. 
First, this paper is a continuation of preceding work in \cite{KowKra:BVSOEAT, Kow:BVASOCAPB, Kow:GABVSOMOO} in which we investigated Gerstenhaber and BV algebra (as well as module) structures
on derived functors over quite general rings, or more precisely, on $\Ext^\bull_U(A,M)$, $\Tor^\bull(N,A)$, and $\Cotor^\bull_U(A,M)$ for a bialgebroid $U$ over a in general noncommutative base algebra $A$ and certain coefficients $M, N$. Here, the question remained open in which circumstances the canonical Gerstenhaber structure on $\Ext^\bull_U(A,M)$ given in \cite{KowKra:BVSOEAT, Kow:BVASOCAPB} (for certain coefficients $M$) is indeed a BV algebra.

The question about the existence of these higher structures is related to the structure of a cocyclic module (in the sense of Connes \cite{Con:CCEFE}) on the complexes computing the respective (co)homology groups and, in particular, to the existence of a (co)cyclic operator such that the respective complexes become a {\em cyclic operad} \cite{GetKap:COACH}, which implies a BV algebra structure on cohomology \cite{Men:BVAACCOHA}: see \S\ref{prelim} for all necessary definitions and results we are going to use. 

Hence, put differently, the first question we want to answer is: under which conditions is there a cyclic structure on the complex computing $\Ext^\bull_U(A,M)$? 

Observe that this question already appeared very early on a basic level \cite{Con:NCDG}: as discussed shortly in Remark \ref{spina}, it is a priori not clear how to define a cocyclic operator on the Hochschild complex $C^\bull(A,A)$ (which computes $\Ext_\Ae(A,A)$ if $A$ is $k$-projective) for an arbitrary associative $k$-algebra $A$ with coefficients in the algebra itself.
%, which is why usually cyclic cohomology is defined via the complex $C^\bull(A, A^*)$, where $A^* := \Hom_k(A,k)$. 

As Hochschild theory is obtained from bialgebroid theory by considering the bialgebroid $(\Ae,A)$, an answer in the bialgebroid setting will give an answer to this problem as well, which is the second goal of this article: for which associative $k$-algebras $A$ one can find a cocyclic operator on the Hochschild complex $C^\bull(A,A)$ making it into a cyclic operad and hence its Hochschild cohomology $H^\bull(A,A)$ into a BV algebra? 

%Observe that this is {\em not} always the case and a counterexample of an algebra $A$ the Hochschild cohomology of which is not BV %can be easily constructed 
%by considering a free algebra in two generators \cite{Krae:PC}. 

On the other extreme, 
as for the cyclic cohomology of a Hopf algebra $H$ over a commutative ring $k$ (see, {\it e.g.}, \cite{Men:BVAACCOHA, Men:CMCMIALAM} for an overview), the cocyclic operator $\tau$ for a cochain of, say, degree one $f \in \Hom_k(H,k)$ would be simply $\tau (f) := f \circ S$, where $S$ is the antipode of $H$. This unfortunately cannot be so easily generalised to a (left) Hopf algebroid $(U,A)$ as usually there is no antipode in a proper sense, and even if there were, this would not be of much help: here, the cochain space in degree one is $\Hom_\Aopp(U,A)$ and a possible antipode would turn the (various) $A$-module structures around, that is, $f \circ S$ would not land in $\Hom_\Aopp(U,A)$ again.

A simple idea of how to possibly obtain a cocyclic structure on the complex computing $\Ext^\bull_U(A,M)$ goes by passing through the duals: if a left bialgebroid $(U,A)$ is finitely generated $A$-projective (in one of the possible four senses), then one knows \cite{KadSzl:BAODTEAD} that the (right) dual $U^* := \Hom_\Aopp(U,A)$ is a {\em right} bialgebroid. 

On the other hand, in \cite{Kow:BVASOCAPB} we showed that the $\Cotor$-groups over a (left) Hopf algebroid carry the structure of a BV algebra, and via the 
$k$-module isomorphism $\Cotor^\bull_{U^*}(M,A) \simeq \Ext^\bull_U(A,M)$, where on the left hand side $M$ is considered as a right $U^*$-comodule and on the right hand side as a left $U$-module, there should be one on the $\Ext$-groups as well, arising from a cocyclic structure that, once obtained, possibly makes sense even if one drops the finiteness assumption, which is needed if one wants to include the Hochschild theory as $\Ae$ usually is not finitely generated over $A$. 

Here, however, arise two difficulties, which lead us to the third goal in this paper: for $\Cotor^\bull_{U^*}(M,A)$ to be a BV algebra, the right bialgebroid $U^*$ needs to carry a (right) Hopf algebroid structure, but until very recently it was not known whether this is the case. The question was asked in \cite{Boe:HA} and probably earlier, some progress in this direction was achieved in \cite{CheGavKow:DFOLHA}, but only in \cite{Schau:TDATDOAHAAHA} an affirmative answer was given by an elegant abstract categorical reasoning, which unfortunately lacked an explicit formula for the {\em translation map} characterising the Hopf structure as a substitute for the antipode, see \S\ref{gamberetti} for all technical details. Hence, the third question we aim to answer in this article is: what is the explicit Hopf structure on the dual $U^*$ of a left Hopf algebroid $U$? 

On top, there arise even more technical complications as the coefficient module $M$ in $\Cotor^\bull_{U^*}(M,A)$ needs to be a {\em (stable) anti Yetter-Drinfel'd module} over $U^*$, which means a left $U^*$-module and right $U^*$-comodule with action and coaction compatible in a certain way. Now, right $U^*$-comodules correspond to left $U$-modules, whereas left $U^*$-modules rather correspond to {\em right $U$-contramodules} which were introduced by \cite{EilMoo:FORHA} half a century ago, but later somehow forgotten. Hence, another question we wish to clarify is how the anti Yetter-Drinfel'd compatibility of a left $U^*$-module right $U^*$-comodule transforms into a compatibility between a left $U$-module structure and a right $U$-contramodule structure on the same underlying $A$-module.

\subsection{Main results}
\label{main}
With respect to the three question just asked, let us list the answers we found. Again, we refer to the main text for all details as well as notation.

\subsubsection{Duals of (left) Hopf algebroids}
In Theorem \ref{Hopfdual}, we give an explicit expression for the translation map on the (right) dual of a left Hopf algebroid:

\begin{theorem}
\label{A}
Let $(U,A)$ be a left Hopf algebroid with translation map $u \mapsto u_+ \otimes_\Aopp u_-$ for $u \in U$, and let $U$ be finitely generated projective as a right $A$-module via the target map. Then the right dual $(U^*,A)$ carries the structure of a right Hopf algebroid over a right bialgebroid. 
More precisely, the map $\tilde\beta^{-1}: U^* \to {U^*} \otimes_\Aopp U^*$ given by
\begin{equation*}
%\label{viviverde1}
({\tilde\gb}^{-1}(\phi))(u,v) = (u \rightslice \phi)(v) 
:= \gve(\phi(u_-v) \blact u_+)
%= \gve(u_+ t^\ell\phi(u_-v))
\end{equation*}
yields a translation map 
%which defines an inverse to the Hopf-Galois map 
%\begin{equation*}
%\label{favini}
%\gb: {U^*} \otimes_\Aopp U^* \to U^* \otimes_\ahha {U^*}, \quad \phi \otimes_\Aopp \psi \mapsto \phi \psi^{(1)} \otimes_\ahha \psi^{(2)}
%\end{equation*}
on $U^*$.
Explicitly, if $\{e_i\}_{1 \leq i \leq n} \in U, \ \{ e^i\}_{1 \leq i \leq n} \in U^*$ is a dual basis, the translation 
map reads as
\begin{equation*}
%\label{viviverde2}
\phi^- \otimes_\Aopp \phi^+ := \textstyle\sum_i e^i \otimes_\Aopp (e_i \rightslice \phi).
\end{equation*}
\end{theorem}

If $U$ happens to be not only a left Hopf algebroid but also a right Hopf algebroid (still on the underlying structure of a left bialgebroid) in the sense mentioned, for example, in \cite[\S2.3]{CheGavKow:DFOLHA} or \S\ref{gamberetti}, 
then applying the above map for $v=1$ leads to an isomorphism between the right dual $U^*$ and the left dual $U_*$, and one can then speak of {\em the only} dual, which is by the above again a left and a right Hopf algebroid (over the underlying structure of a right bialgebroid), see Remark \ref{nocheins}. This should lead to a possibly easier statement compared to the approach in \cite[\S5]{BoeSzl:HAWBAAIAD} about the dual(s) of a {\em full} Hopf algebroid.

\subsubsection{Contramodules}

A (right) contramodule over a left bialgebroid $(U,A)$ is a right $A$-module $M$ together with a right $A$-module map 
$$
\gamma: \Hom_\Aopp(U,M) \to M
$$ 
that fulfils a sort of associativity and unitality property, see Definition \ref{schoenwaers} for details. As mentioned in \S\ref{aims}, we are interested in the question of how anti Yetter-Drinfel'd modules over the dual $U^*$ of a finitely generated $A$-projective left Hopf algebroid $U$ transform into a module and contramodule over $U$ with compatibility between the action and the contraaction, leading to the notion of stable anti Yetter-Drinfel'd {\em contra}modules for left Hopf algebroids in Definition \ref{chelabertaschen1}. The main statement in Lemma \ref{sofocle2} is then that there is an equivalence
$$
{}_{\scriptscriptstyle{U^*}}\mathbf{aYD}^{\scriptscriptstyle{U^*}} \simeq {}_\uhhu \mathbf{aYD}^{\scriptscriptstyle{\rm contra-}\uhhu}
$$
between the categories of (stable) aYD modules over $U^*$ and (stable) aYD contramodules over $U$. None of the above categories appears to be monoidal but both of them are module categories over the category of Yetter-Drinfel'd modules (at least in the finite case), see Proposition \ref{jetztnkaffee}.

\subsubsection{Cyclic and BV structures on $\Ext$}

In \S\ref{atacvantaggi}, we explicitly describe the structure of a cocyclic $k$-module on the complex we are interested in, which allows us to prove in
Theorem \ref{corsoumbertoI} and its Corollary \ref{pourquoilamuit} the following central result:

\begin{theorem}
\label{B}
Let $U$ be a left Hopf algebroid and let $M$ be a stable anti Yetter-Drinfel'd contramodule over $U$. 
Then the complex 
$$
C^\bull(U,M) := \Hom_\Aopp(U^{\otimes_\Aopp \raise1pt\hbox{$\bull$}}, M) 
$$ 
can be made into a cocyclic $k$-module
with cocyclic operator
$$
(\tau f)(u^1, \ldots, u^n) = \gamma\big(
u^1_+ f(u^2_+, \ldots, u^n_+, u^n_- \cdots u^1_-(-))\big).
$$
In particular, choosing the base algebra itself as coefficient module, $C(U,A)$ becomes a cyclic operad with multiplication and therefore the cohomology groups $H^\bull(U,A)$ (resp.\ $\Ext^\bull_U(A,A)$ if $U_\ract$ is projective) form
a Batalin-Vilkovisky algebra.
\end{theorem}

For a Hopf algebra $H$ over a commutative ring $k$, the contraaction $\gamma$ that appears in the above theorem is simply evaluation on the unit and one thereby recovers the BV algebra structure on $\Ext^\bull_H(k,k)$ that was given by Menichi \cite{Men:CMCMIALAM}.

We also give a version for more general coefficients in which $A$ in the second statement in the above theorem is replaced by a braided commutative Yetter-Drinfel'd algebra (see the second part of Theorem \ref{corsoumbertoI}), where we, however, have to assume a certain finiteness condition again.

\subsubsection{BV algebra structures on Hochschild cohomology of associative algebras}
The aforementioned Theorem \ref{B} can then be applied to the case of the Hopf algebroid $(U,A) = (\Ae,A)$ which controls Hochschild theory, and therefore yields statements on when the Hochschild cohomology of an associative algebra carries the structure of a BV algebra.
In \S\ref{scottex}, we give a sufficient condition for when this is the case:

\begin{theorem}
\label{B1}
Let $A$ be an associative $k$-algebra which is a contramodule over $\Ae$ with contraaction $\gamma$. Then
$$
(\tau f)(a_1, \ldots, a_n) = \gamma\big(a_1 f(a_2, \ldots, a_n, -) \big)
$$
defines a cocyclic operator on the Hochschild complex 
$
C^\bull(A,A) := \Hom_k(A^{\otimes \bull},A)
$ 
such that the respective endomorphism operad $C(A,A)$ becomes a para-cyclic operad with multiplication, which is cyclic if $A$ is stable over $\Ae$. Hence, its Hochschild cohomology groups $H^\bull(A,A)$ (resp.\  $\Ext^\bull_\Aee(A,A)$ if $A$ is $k$-projective) form
a Batalin-Vilkovisky algebra.
\end{theorem}

As discussed in \S\ref{castelnuovo1} and \S\ref{castelnuovo3}, examples of when such a contraaction exists are given by symmetric algebras, or, more generally, by Frobenius algebras with semisimple (diagonalisable) Nakayama automorphism, recovering the aforementioned results of Menichi \cite{Men:BVAACCOHA}, Tradler \cite{Tra:TBVAOHCIBIIP}, 
Lambre-Zhou-Zimmermann \cite{LamZhoZim:THCROAFAWSSNAIABVA} as well as \cite{Vol:BVDOHCOFA}.

How to find a contraaction on (twisted) Calabi-Yau algebras so as to also recover the results in \cite{Gin:CYA} and \cite{KowKra:BVSOEAT} appears to be more intricate and will be deferred to a separate publication, in which we also plan to find new examples and to include further aspects like Koszul duals as in \cite{CheYanZho:BVAATNPDOKCYA}. 
The most important question to clarify would be whether one can classify contraactions $\gamma$ for a given
algebra $A$, and in particular whether the existence of such a map is
implied by or even equivalent to already known concepts on associative algebras, or whether this leads to a new class of algebras instead.

\bigskip

\thanks{ {\bf Acknowledgements.}   \!
With great pleasure, we would like to thank 
 Tomasz Brzezi\'nski, Domenico Fiorenza, 
 Laiachi El Kaoutit, 
Ulrich Kr\"ahmer, Boris Tsygan, and the referee
 for stimulating discussions and precious comments.

\section{Preliminaries}
\label{prelim}

In this preliminary section, we gather most of the basic (algebraic) ideas we are going to use, and also fix some notation.
Let $k$ be a commutative ground ring (most of the time of characteristic zero), and as always let an unadorned tensor product be meant over $k$.

\subsection{Cyclic operads and Batalin-Vilkovisky algebras}

The main point in this subsection is given by Theorem \ref{holl} below, which establishes a relationship between Gerstenhaber algebras and operads with multiplication resp.\ Batalin-Vilkovisky algebras and cyclic operads with multiplication, which is the fundamental result underlying our entire considerations. Let us define the required ingredients first.

\begin{dfn}
\label{fameancora}
\
\begin{enumerate}
\compactlist{99}
\item
A {\em Gerstenhaber algebra} 
over $k$ is given by a triple $\big(V, \smallsmile, \{\cdot,\cdot\}\big)$, where 
$V=\bigoplus_{p \in \mathbb{N}} V^p$ is a graded commutative $k$-algebra with multiplication
$ \ga \smallsmile \gb=(-1)^{pq}\gb \smallsmile \ga \in V^{p+q}$ for $\ga \in V^p,\gb \in V^q$, 
along with a graded Lie bracket
$$
        \{\cdot,\cdot\} : V^{p+1} \otimes_k V^{q+1} \rightarrow V^{p+q+1}
$$ 
on the desuspension 
$
        V[1]:=\bigoplus_{p \in \mathbb{N}} V^{p+1},
$
for which all operators $\{\gamma,\cdot\}$ satisfy the graded Leibniz rule
$$
        \{\gamma,\ga \smallsmile \gb\}=
        \{\gamma,\ga\} \smallsmile \gb + (-1)^{pq} \ga \smallsmile
        \{\gamma,\gb\},
$$
for $\gamma \in V^{p+1}$ and $\ga \in V^q$. 
\item
A {\em Batalin-Vilkovisky} algebra is a Gerstenhaber algebra $V$ with a $k$-linear differential
$
B: V^n \to V^{n-1} 
$
of degree $-1$
such that for all $\ga \in V^p$, $\gb \in V$
$$
\{\ga, \gb\} = (-1)^{p}\big(B(\ga \smallsmile \gb) - B(\ga) \smallsmile \gb - (-1)^p \ga \smallsmile B(\gb) \big)
$$
holds.
\end{enumerate}
\end{dfn}
A Batalin-Vilkovisky algebra may also be called {\em exact} Gerstenhaber algebra and $B$ is sometimes said to {\em generate} the Gerstenhaber bracket.  

Turning to the next ingredient we are going to use, note that in all what follows the term ``(cyclic) operad'' always refers to a non-$\gS$ (cyclic) operad in the category of $k$-modules in the sense given right below. See, for example, \cite{LodVal:AO, Mar:MFO} for more information on operads, and \cite{GetKap:COACH, MarShnSta:OIATAP, Men:BVAACCOHA} 
for cyclic ones.

\pagebreak

\begin{definition}
\
\label{moleskine}
\begin{enumerate}
\compactlist{99}
\item
A (non-$\gS$) {\em operad} in the category of $k$-modules is a sequence $\{O(n)\}_{n \geq 0}$ of $k$-modules with an identity element 
$\mathbb{1} \in O(1)$ together with $k$-bilinear operations 
$$
\circ_i: O(p) \otimes O(q) \to O(p+q-1)
$$
subject to
\begin{eqnarray}
\label{danton}
\nonumber
\gvf \circ_i \psi &=& 0 \qquad \qquad \qquad \qquad \qquad \! \mbox{if} \ p < i \quad \mbox{or} \quad p = 0, \\
(\varphi \circ_i \psi) \circ_j \chi &=& 
\begin{cases}
(\varphi \circ_j \chi) \circ_{i+r-1} \psi \qquad \mbox{if} \  \, j < i, \\
\varphi \circ_i (\psi \circ_{j-i +1} \chi) \qquad \hspace*{1pt} \mbox{if} \ \, i \leq j < q + i, \\
(\varphi \circ_{j-q+1} \chi) \circ_{i} \psi \qquad \mbox{if} \ \, j \geq q + i,
\end{cases} \\
\nonumber
\gvf \circ_i \mathbb{1} &=& \mathbb{1} \circ_i \gvf \ \ = \ \ \gvf \! \qquad \quad \qquad \mbox{for} \ i \leq p,   
\end{eqnarray}
for any $\varphi \in O(p), \ \psi \in O(q)$, and $\chi \in O(r)$. 
The operad is called {\em with multiplication} if there exists an {\em operad multiplication} $\mu \in O(2)$ and a {\em unit} $e \in O(0)$ such that 
% \begin{equation}
% \label{distinguished, I said 1}
% \begin{array}{rclrcl}
$ 
\mu \circ_1 \mu 
% &=& 
 = \mu \circ_2 \mu
$
and  
%\\
$
\mu \circ_1 e 
% &=&
= 
\mu \circ_2 e
= \mathbb{1}
$ 
%\end{array}
%\end{equation}
holds.
\item
A {\em cyclic} operad is a (non-$\gS$) operad $O$ equipped with $k$-linear maps 
$$
\tau_n: O(n) \to O(n)
$$ 
subject to
\begin{equation}
\label{superfluorescent1}
\begin{array}{rcll}
\tau (\gvf \circ_1 \psi) &=& \tau \psi \circ_q \tau \gvf, & \mbox{if} \ 1 \leq p, q, \\
\tau (\gvf \circ_i \psi) &=& \tau \gvf \circ_{i-1} \psi, & \mbox{if}  \ 0 \leq q \ \mbox{and} \  2 \leq i \leq p, \\
\tau \mathbb{1} &=& \mathbb{1} & \\
\tau^{n+1} &=& \id_{O(n)}, & 
\end{array}
\end{equation}
for every $\gvf \in O(p)$ and $\psi \in O(q)$. 
In case the last equation is not fulfilled, one also speaks (in analogy to cyclic homology) of a {\em para-}cyclic operad.
A {\em cyclic operad with multiplication} is both a cyclic operad and an operad with multiplication $\mu$ such that
%\begin{equation}
%\label{superfluorescent2}
$
\tau \mu = \mu.
$
%\end{equation}
\end{enumerate}
\end{definition}

The first part of the following well-known useful result is due to \cite{Ger:TCSOAAR, GerSch:ABQGAAD, McCSmi:ASODHCC} and possibly others, whereas its enhanced second part appeared in \cite[Thm.~1.4]{Men:BVAACCOHA}:

\begin{theorem}
%[\cite{Men:BVAACCOHA}]
\label{holl}
\
\begin{enumerate}
\compactlist{99}
\item
Any operad with multiplication defines a cosimplicial $k$-module the cohomology of which carries the structure of a Gerstenhaber algebra.
\item
Any cyclic operad with multiplication 
defines a cocyclic $k$-module the (simplicial) cohomology of which carries the structure of a Batalin-Vilkovisky algebra. 
\end{enumerate}
\end{theorem}

\subsection{Left and right Hopf algebroids}
The third fundamental concept on which our results are based, is the surprisingly powerful notion of bialgebroids and Hopf algebroids.
\subsubsection{$\Ae$-rings}
Let $A$ and $U$ be
(unital associative) $k$-algebras. Assume that there is 
a fixed $k$-algebra map
$
%\begin{equation*}
%\label{eta}
        \eta : \Ae := A \otimes \Aop \rightarrow 
        U.
$
%\end{equation*}
This map induces forgetful functors 
$$
\umod \rightarrow \amoda,
\quad
\modu \rightarrow \amoda
$$
from the category of left resp.\ right $U$-modules to the category of $A$-bimodules, that is, every $N \in \umod$ resp.\ $M \in \modu$ becomes an $A$-bimodule via
\begin{equation}
\label{pergolesi}
        a \lact n \ract b:=\eta(a \otimes b)n,\quad
        a \blact m \bract b:=m \eta(b \otimes a),\quad
        a,b \in A,n \in N,m \in M.
\end{equation}
This in particular applies to $U$ itself, that is, 
left and right multiplication in $U$ 
define $A$-bimodule structures of both these types on 
$U$ itself, and this defines two morphisms
$
% \begin{equation*}
%\label{basmati}
        s: A \rightarrow U, 
\ 
s(a):=\eta (a \otimes 1)$,
and
$
t: \Aop \rightarrow U, 
\
        t(b):=\eta (1 \otimes b)
$ 
for $a, b \in A$,
% \end{equation*}
the {\em source} resp.\ {\em target} map of the pair $(U,A)$.

\subsubsection{Left and right bialgebroids} 
Recall then from \cite{Tak:GOAOAA} that a {\em left bialgebroid} is a sextuple $(U,A,s^\ell,t^\ell,\Delta_\ell, \gve)$, abbreviated $(U,A)$, which adds to the data of the $\Ae$-ring $(U,A, s^\ell, t^\ell)$ introduced above two $A$-bimodule maps with respect to the $A$-bimodule structure $\due U \lact \ract$, the {\em left coproduct}
$
%\begin{equation}
%\label{tuscania1}
\Delta_\ell: U \to \due U {} \ract  \otimes_\ahha \due U \lact {}, \ u \mapsto u_{(1)} \otimes_\ahha u_{(2)}, 
%\end{equation}
$
where we use the common Sweedler subscript notation (with summation understood),
and the  {\em left counit} $\gve: U \to A$ such that $(U, \gD_\ell, \gve)$ becomes a counital $A$-coring, which means that one has a couple of compatibility conditions that need more technical attention than those for bialgebras as the base ring $A$ is in general noncommutative; see, {\em e.g.}, \cite[Def.~3.3]{Boe:HA}. For example, the counitality axioms read as $\gve(u_{(1)}) \lact u_{(2)} = u = u_{(1)} \ract \gve(u_{(2)})$. As in the bialgebra case, one also wants the comultiplication to be a morphism of the multiplication in the sense of $\gD(uv) = \gD(u)\gD(v)$, and in order to give this equation a well-defined sense, the coproduct has to corestrict to a map $U \to U \times_\ahha U$, where $U \times_\ahha U$ is the {\em Sweedler-Takeuchi product}, that is, the $\Ae$-submodule of  $U \otimes_\ahha U$
defined by
%
%
%whose elements $\sum_i u_i \otimes_\ahha v_i$ fulfil 
%\begin{equation}
%\label{tellmemore}
% \textstyle
%\sum_i a \blact u_i \otimes_\ahha v_i
%		  = \sum_i u_i \otimes_\ahha v_i \bract a, \
%		  \forall a \in A,
% \end{equation}
%
\begin{equation}
\label{tellmemore}
U \times_\ahha   U   := 
     \big\{ {\textstyle \sum_i} u_i \otimes u'_i \in U_\ract  \otimes_{\scriptscriptstyle A}  {}_\lact U \mid {\textstyle \sum_i} (a \blact u_i) \otimes u'_i = {\textstyle \sum_i} u_i \otimes (u'_i \bract a), \ \forall a \in A \big\},
\end{equation}
which in contrast to $U \otimes_\ahha U$ becomes an $\Ae$-ring by factorwise multiplication. Also, 
\begin{equation}
\label{alsodoch}
\gve(a \blact u) = \gve(u \bract a), \qquad \gve(uv) = \gve(u \bract \gve(v))
\end{equation}
for the counit, for all $u, v \in U$ and $a \in A$.

\medskip

A {\em right bialgebroid} \cite[\S2]{KadSzl:BAODTEAD}  
is again a sextuple $(V,B, s^r, t^r, \gD_r, \pl)$ formed by a $\Be$-ring $(V,B, s^r, t^r)$ together with two $B$-bimodule maps but this time with respect to the $B$-bimodule structure $\due V \blact \bract$, the  
{\em right coproduct} 
$\Delta_r: V \to \due V {} \bract  \otimes_\behhe \due V \blact {}, \ v \mapsto v^{(1)} \otimes_\behhe v^{(2)}$, where we use the less common Sweedler {\em super\/}script notation (with summation understood),
and the {\em right counit} $\pl: V \to B$, subject to certain compatibility conditions which are opposite to those of a left bialgebroid. Indeed, the {\em opposite} $( U^\op, A, t^\ell, s^\ell, \Delta_\ell, \epsilon )$ of a left bialgebroid
$( U, A, s^\ell, t^\ell, \Delta_\ell, \epsilon )  $  is a right  bialgebroid, and from this one can easily deduce the explicit axioms for a right bialgebroid.

Both notions of left and right bialgebroid generalise bialgebras (the case of which is recovered by taking $A=k$ resp.\ $B=k$); see, for example, \cite[\S2]{Kow:HAATCT} for an overview.

\subsubsection{Left and right Hopf algebroids}
\label{gamberetti}
Following Schauenburg's definition  \cite{Schau:DADOQGHA}, we say that a left bialgebroid $(U,A)$ 
is a {\em left Hopf algebroid} if 
the  {\em Hopf-Galois map}
\begin{equation}
\label{ilariabianchi1}
\ga_\ell: \due U \blact {} \otimes_{\Aopp} U_\ract \to U_\ract  \otimes_\ahha  \due U \lact,
\quad 
u \otimes_\Aopp u'  \mapsto  u_{(1)} \otimes_\ahha u_{(2)}  u'
\end{equation}
is bijective. In this case, one can define a so-called {\em translation map} for which we introduce the Sweedler notation 
$
u_+ \otimes_\Aopp u_-  :=  \ga^{-1}_\ell(u \otimes_\ahha 1).
$
%We abbreviate $\tilde\ga^{-1}_\ell := \ga^{-1}_\ell(- \otimes_\ahha 1)$ and write this simply as a map 
%$$
%U \to  \due U \blact {} \otimes_{\Aopp} U_\ract, \ u \mapsto u_+ \otimes_\Aopp u_-.
%$$
In case $A=k$ is central in $U$, the map $\ga_\ell$ is invertible if and only if $U$ is a Hopf algebra, and one has
 $  u_+ \otimes u_-  :=  u_{(1)} \otimes S(u_{(2)})  $, where $S$ is the antipode of the Hopf algebra.

%, and if $(U,A)$ is (part of) a {\em full} Hopf algebroid, that is, both a left and a right bialgebroid, then  $  u_+ \otimes u_-  :=  u^{(1)} \otimes S(u^{(2)})  $, where Sweedler superscripts refer to the {\em right} bialgebroid structure which is then part o

Of course, \eqref{ilariabianchi1} is not the only possible Hopf-Galois map that can be defined for a left bialgebroid: 
the map 
$
\ga_r: U_{\!\bract}  \otimes_\ahha \! \due U \lact {} \to U_{\!\ract}  \otimes_\ahha  \due U \lact {}, \ u \otimes_\ahha v  \mapsto  u_{(1)}  v \otimes_\ahha u_{(2)}
$ 
is another one, and if this map is invertible, the left bialgebroid $(U,A)$ is called a {\em right} Hopf algebroid.
%% ; one obtains a translation map
%% %
%% % U \to U_{\!\ract}  \otimes_\ahha  \due U \lact {}, \ u \mapsto 
%% $
%% u_{[+]} \otimes_\ahha u_{[-]} := \ga_r^{-1}(1 \otimes_\ahha u). 
%% $
%% Here, for $A = k$ one obtains $u_{[+]} \otimes u_{[-]} = u_{(2)} \otimes S^{-1}(u_{(1)})$ if and only if $U$ is a Hopf algebra with invertible antipode.

We refer to, for example,  \cite[Prop.\ 4.2]{BoeSzl:HAWBAAIAD} 
and \cite[\S2.3]{CheGavKow:DFOLHA} and references therein for further details and an explanation of the terminology. 
%The case of {\em full} Hopf algebroids \cite{BoeSzl:HAWBAAIAD}  will be briefly mentioned in \S\ref{t'n'c}.

Of course, the notions of left and right Hopf algebroid also exist if the underlying bialgebroid is a right instead of a left one. Since we are going to deal with the dual of a left bialgebroid (which is a right bialgebroid), we will need (one of) these concepts as well:
we say that the right bialgebroid $(V,B)$ is a {\em right Hopf algebroid} if the Hopf-Galois map 
\begin{equation}
\label{ilariabianchi2}
\gb: \due V \blact {} \otimes_{\Bopp} V_\ract \to V_\bract  \otimes_\behhe  \due U \blact {},
\quad 
v' \otimes_\Bopp v  \mapsto v' v^{(1)} \otimes_\behhe v^{(2)} 
\end{equation}
is bijective. Similarly as before, we define the corresponding translation map by
\begin{equation}
\label{ilariabianchi3}
v^- \otimes_\Bopp v^+  :=  \gb^{-1}(1 \otimes_\behhe v),
\end{equation}
and we abbreviate $\tilde\gb^{-1} := \gb^{-1}(1 \otimes_\behhe -)$.
Again, in case $B=k$ is central in $V$, the map $\gb$ is invertible if and only if $V$ is a Hopf algebra, and one has
 $  v^- \otimes v^+  :=  S(v_{(1)}) \otimes v_{(2)}$. 

\begin{rem}
\label{evenmoreconfusion}
This latter definition might now lead to slight confusion in terminology as saying ``left/right Hopf algebroid'' does not specify whether the underlying bialgebroid is left or right, whereas ``left/right Hopf algebroid for a left/right bialgebroid'' appears way too clumsy. However, in the following we will {\em always} mean by ``left Hopf algebroid'' as having an underlying left bialgebroid structure, and by ``right Hopf algebroid'' an underlying right one.
%, {\em except for \S\ref{t'n'c}}, where left and right Hopf algebroid structures are considered over the same underlying left bialgebroid structure.
We also want to stress that interchanging left and right here is more than a pure exercise in chirality yoga: as mentioned before, this determines the monoidality of the respective category of modules.
\end{rem}

The following lemma collects some useful properties for the so-defined translation maps for left and right bialgebroids (see \cite[Prop.~3.7]{Schau:DADOQGHA} for the first part and \cite[Lem.~2.14]{BoeSte:CCOBAAVC} for the second):

\begin{lem}
\quad
\begin{enumerate}
\compactlist{99}
\item Let $(U,A)$ be a left Hopf algebroid over an underlying left bialgebroid. Then the following relations hold:
\begin{eqnarray}
\label{Sch1}
u_+ \otimes_\Aopp  u_- & \in
& U \times_\Aopp U,  \\
\label{Sch2}
u_{+(1)} \otimes_\ahha u_{+(2)} u_- &=& u \otimes_\ahha 1 \quad \in U_{\!\ract} \! \otimes_\ahha \! {}_\lact U,  \\
\label{Sch3}
u_{(1)+} \otimes_\Aopp u_{(1)-} u_{(2)}  &=& u \otimes_\Aopp  1 \quad \in  {}_\blact U \! \otimes_\Aopp \! U_\ract,  \\
\label{Sch4}
u_{+(1)} \otimes_\ahha u_{+(2)} \otimes_\Aopp  u_{-} &=& u_{(1)} \otimes_\ahha u_{(2)+} \otimes_\Aopp u_{(2)-},  \\
\label{Sch5}
u_+ \otimes_\Aopp  u_{-(1)} \otimes_\ahha u_{-(2)} &=&
u_{++} \otimes_\Aopp u_- \otimes_\ahha u_{+-},  \\
\label{Sch6}
(uv)_+ \otimes_\Aopp  (uv)_- &=& u_+v_+ \otimes_\Aopp v_-u_-,
\\
\label{Sch7}
u_+u_- &=& s^\ell (\varepsilon (u)),  \\
\label{Sch8}
\varepsilon(u_-) \blact u_+  &=& u,  \\
\label{Sch9}
(s^\ell (a) t^\ell (b))_+ \otimes_\Aopp  (s^\ell (a) t^\ell (b) )_-
&=& s^\ell (a) \otimes_\Aopp s^\ell (b),
\end{eqnarray}
where in  \eqref{Sch1}  we mean the Takeuchi-Sweedler product
\begin{equation*}
\label{petrarca}
   U \! \times_\Aopp \! U   :=
   \big\{ {\textstyle \sum_i} u_i \otimes v_i \in {}_\blact U  \otimes_\Aopp  U_{\!\ract} \mid {\textstyle \sum_i} u_i \ract a \otimes v_i = {\textstyle \sum_i} u_i \otimes a \blact v_i, \ \forall a \in A \big\}.
\end{equation*}
\item
Let $(V,B)$ be a right Hopf algebroid over an underlying right bialgebroid. Then one has:
\begin{eqnarray}
\label{Rch1}
v^- \otimes_\Bopp  v^+ & \in
& V \times_{\Bopp} V,  \\
\label{Rch2}
v^-v^{+(1)} \otimes_\behhe v^{+(2)}  &=& 1 \otimes_\behhe v \quad \in V_{\!\bract} \! \otimes_\behhe \! {}_\blact V,  \\
\label{Rch3}
v^{(1)}v^{(2)-} \otimes_\Bopp v^{(2)+}  &=& 1 \otimes_\Bopp v \quad \in \due V \blact {} \!
\otimes_\Bopp \! \due V {} \ract,  \\
\label{Rch4}
v^{(1)-} \otimes_\Bopp v^{(1)+} \otimes_\behhe v^{(2)} &=& 
v^{-} \otimes_\Bopp v^{+(1)} \otimes_\behhe v^{+(2)},  
\\
\label{Rch5}
v^{-(1)} \otimes_\behhe v^{-(2)} \otimes_\Bopp v^{+} &=& 
v^{+-} \otimes_\behhe v^{-} \otimes_\Bopp v^{++},  
\\
\label{Rch6}
(vw)^- \otimes_\Bopp (vw)^+ &=& w^-v^- \otimes_\Bopp v^+w^+,  \\
\label{Rch7}
v^-v^+ &=& s^r (\pl (v)),  \\
\label{Rch8}
v^+ \ract \pl(v^-)   &=&  v,  \\
\label{Rch9}
(s^r (b) t^r (b'))^- \otimes_\Bopp (s^r (b) t^r (b') )^+
&=& s^r(b') \otimes_\Bopp s^r(b),
\end{eqnarray}
where in  \eqref{Rch1}  we mean the Sweedler-Takeuchi product
\begin{equation*}  
\label{petrarca2}
   V \times_\Bopp V   :=
   \big\{ {\textstyle \sum_i} v_i \otimes  w_i \in V \otimes_\Bopp  V  \mid {\textstyle \sum_i} v_i \ract b \otimes w_i = {\textstyle \sum_i} v_i \otimes b \blact w_i,  \ \forall b \in B  \big\}.
\end{equation*}
\end{enumerate}
\end{lem}

\medskip

\pagebreak

\subsection{Modules, comodules, and contramodules}
\subsubsection{Modules and comodules over bialgebroids}
We shall not discuss all details here of modules and comodules over bialgebroids as they have been written many times in the literature, see, for example, \cite{Boe:HA} for an overview. 

%\pagebreak

However, we want to repeat that --- in contrast to bialgebras --- the category $\modu$ of right modules over a left bialgebroid $U$ is in general {\em not} monoidal, whereas the category of left modules $\umod$ is so; they same holds with left and right interchanged for right bialgebroids. 

In \cite[Prop.\ 3.1.1]{CheGavKow:DFOLHA}, we listed a multitude of $U$-module structures over $\Hom$-spaces and tensor products of $U$-modules, two of which will be important in what follows: first, for $M \in \modu$ and $N \in \umod$, their tensor product $M \otimes_\Aopp N$ is again a right $U$-module by means of a sort of adjoint action
\begin{equation}
\label{lemma3}
(m \otimes_\Aopp n)u := mu_+ \otimes_\Aopp u_-n 
\end{equation}
for $m \in M$, $n \in N$, $u \in U$. Another crucial observation for our subsequent considerations is that for $M, M' \in \umod$,
the $\Ae$-module $\Hom_{\Aopp}(M,M') $ carries a sort of transposed adjoint action if $U$ is a left Hopf algebroid, that is, 
by
\begin{equation}
\label{gianduiotto1}
(u \rightslice f)(m)  :=  u_+ \big( f(u_-m) \big)
\end{equation}
one obtains a left  $ U $-module structure on $\Hom_{\Aopp}(M,M') $. In particular, if $M=U$ and $M'=A$, then \eqref{gianduiotto1} reads 
\begin{equation}
\label{gianduiotto1a}
(u \rightslice \phi)(v)  :=  \gve\big( \phi(u_-v) \blact u_+ \big)
\end{equation}
for $u,v \in U$ and $\phi \in U^* := \Hom_\Aop(U,A)$; we will come back to this situation in \S\ref{duals}. 

%See {\em op.~cit.}, \S3.1, for further module structures on $\Hom$-spaces and tensor products in presence of a left Hopf structure, but we will not need them here.

\medskip

A (say) {\em right comodule} $M$ over a (say) right bialgebroid $(V, B, \gD_r, \pl)$ (which is what we will need explicitly in the following) is a right $B$-module that is a right comodule with coaction 
$
m \mapsto m^{(0)} \otimes_\behhe m^{(1)}
$ 
of the coring underlying $V$, see \cite{BrzWis:CAC} for details; there is an induced left $B$-action on $M$ given by $ bm := m^{(0)}\pl(b \lact m^{(1)})$, hence $M$ becomes a $B$-bimodule and the coaction a $B$-bimodule map. 

The right coaction is then a $\Be$-module morphism
 $
 M \rightarrow
         M \times_\behhe V,
 $
where 
\begin{small}
 \begin{equation}
 \label{Takeuchicoaction}
 \textstyle
M \times_\behhe V := \{\textstyle\sum_i m_i \otimes_\behhe v_i \in M \otimes_\behhe V \mid \sum_i b m_i \otimes_\behhe v_i
 		  = \sum_i m_i \otimes_\behhe  v_i \ract b, \
 		  \forall b \in B \}
\end{equation}
\end{small}
is the {\em Takeuchi-Sweedler product}, similarly as for the bialgebroid coproduct, see \eqref{tellmemore}.

Both the categories $\comodv$ and $\vcomod$ of  right resp.\ left $V$-comodules are monoidal, and one has forgetful functors $\vcomod \to \bmodb$ and $\comodv \to \bmodb$. Same comments apply for a left bialgebroid $(U,A)$ to the categories $\ucomod$ and $\comodu$ that have forgetful functors to $\amoda$.

\subsubsection{Contramodules over bialgebroids}
Contramodules over coalgebras were introduced in \cite{EilMoo:FORHA} 
half a century ago and discussed along with comodules, but later somehow neglected and are not overwhelmingly present in the literature. They are dealt with in the direction we need, for example, in \cite{Brz:HCHWCC}, and enlarged to corings and further discussed in \cite{BoeBrzWis:MACOMC, Pos:HAOSAS}. 
In case of a finite dimensional bialgebra (or bialgebroid), one should think of a contramodule as a module over the dual (see Lemma \ref{sofocle2}).
However, contramodules pop up as natural coefficients in the cyclic theory of the $\Ext$-groups and in this sense are also hidden in the classical cyclic cohomology theory in Connes \cite{Con:NCDG}, see \S\ref{lyra}.

\begin{dfn}
\label{schoenwaers}
A {\em right contramodule} over a left bialgebroid $(U,A)$ is a right $A$-module $M$ together with a right $A$-module map 
$$
\gamma: \Hom_\Aopp(U_\ract,M) \to M, 
$$
called the {\em contraaction} (not to be confused with {\em contraction}) subject to
\pagebreak
\begin{small}
\begin{equation*}
\begin{split}
&	\xymatrix{\Hom_\Aopp(U, \Hom_\Aopp(U,M))
	\ar[rrr]^-{\scriptstyle{\Hom_\Aopp(U,\gamma)}} \ar[d]_-{\simeq}&
	& &
	\Hom_\Aopp(U,M) \ar[d]^-{\gamma} \\
	\Hom_\Aopp(U_\ract \otimes_\ahha \due U \lact {},M)
	\ar[rr]_-{\scriptstyle{\Hom_\Aopp(\gD_\ell, M)}} && \Hom_\Aopp(U,M) \ar[r]_-{\gamma}
&	M }
\\
\mbox{\normalsize{and}}
\\
& \xymatrix{\Hom_\Aop(A,M) \ar[rr]^-{\Hom_\Aop(\gve,M)} \ar[drr]_-{\simeq} & & \Hom_\Aop(U,M)
\ar[d]^-{\gamma} \\ & & M.  } 
\end{split}
\end{equation*}
\end{small}
\end{dfn}
The isomorphism of the left vertical arrow of the upper diagram is established by the right $A$-module structure on $\Hom_\Aopp(U_\ract,M)$ given by 
\begin{equation}
\label{carrefour3}
fa := f(a \lact -) 
%(f.a)(u) := f(a \lact u) 
\end{equation}
for $a \in A$, $u \in U$; the required right $A$-linearity of $\gamma$ then reads
\begin{equation}
\label{passionant}
\gamma\big(f(a \lact -) \big) = \gamma(f)a, \quad \forall a \in A.
\end{equation} 
One might be tempted to think that there always exists a trivial contraaction simply given by the evaluation $f \mapsto f(1_\uhhu)$, which obviously {\em is} a map $\Hom_\Aopp(U_\ract,M) \to M$, but it is precisely this right $A$-linearity \rmref{passionant} which excludes this possibility in general. 
%\red{For the same reason, the base algebra $A$ usually does {\em not} carry the structure of a $U$-contramodule.}
However, in the situation that $A=k$, that is, for Hopf algebras, this problem disappears, see \S\ref{toisondor}. 

We will denote the ``free entry'' in the structure map $\gamma$ by hyphens or dots depending on whatever we think is more readable in a specific situation, {\it i.e.}, for $f \in \Hom_\Aopp(U,M)$ write $\gamma(f(-))$ or $\gamma(f(\cdot))$ in explicit computations (see below). As in \cite{Brz:HCHWCC}, 
we explicitly write the condition given by the first diagram for $g \in \Hom_\Aopp(U_\ract \otimes_\ahha \due U \lact {},M)$
as
\begin{equation}
\label{carrefour1}
\dot\gamma\big(\ddot\gamma(g( \cdot \otimes_\ahha \cdot\cdot))\big) 
= \gamma\big(g(-_{(1)} \otimes_\ahha -_{(2)})\big),
\end{equation}
where the dots match the map $\gamma$ with the respective argument, and where the inner contraaction $\ddot\gamma$ has to be carried out first as evident from the first diagram in Definition \ref{schoenwaers}. The second diagram explicitly reads as
\begin{equation}
\label{carrefour2}
\gamma( m \gve(-)) = m
\end{equation}
for $m \in M$.
As mentioned in \cite[\S A.7]{BoeSte:ACATCD} and similarly as for comodules, a right $U$-contramodule additionally induces a left $A$-action given by
\begin{equation}
\label{alleskleber}
am := \gamma\big(m\gve(- \bract a) \big) = \gamma\big(m\gve(a \blact -) \big)
\end{equation}
turning $M$ into an $A$-bimodule, which defines a forgetful functor 
\begin{equation}
\label{gaeta}
\mathbf{Contramod}\mbox{-}U \to \amoda
\end{equation}
from the category of right $U$-contramodules to that of $A$-bimodules. In particular, $\gamma$ this way becomes an $A$-bimodule map with respect to the right $A$-action \eqref{carrefour3} and the left $A$-action $af := f(- \bract a)$ on $\Hom_\Aop(U,M)$, as we see by computing
\begin{equation}
\label{carrefour4}
\begin{split}
\gamma\big(f ( - \bract a)\big) 
& 
{\overset{\scriptscriptstyle{
%\eqref{}
}}{=}}
\gamma\big(f\big(\gve( -_{(1)} \bract a) \lact -_{(2)}\big)  \big) 
%\\
%& 
{\overset{\scriptscriptstyle{
\eqref{carrefour1}
}}{=}}
\dot\gamma\Big(\ddot\gamma\big(f\big(\gve( \cdot \bract a) \lact \cdot\!\cdot\big)  \big)\Big) 
\\
& 
{\overset{\scriptscriptstyle{
\eqref{passionant}
}}{=}}
\dot\gamma\big(\ddot\gamma(f(\cdot\cdot)) \gve( \cdot \bract a) \big) 
%\\
%& 
{\overset{\scriptscriptstyle{
\eqref{alleskleber}
}}{=}}
a \gamma\big(f(-)\big), 
\end{split}
\end{equation}
where we used counitality of $U$ in the first step.
% and the right $A$-linearity of $\gamma$ in the third.

\vspace*{.2cm}
\begin{center}
* \quad * \quad *
\end{center}

Similar definitions hold for left contramodules, or contramodules over right bialgebroids, but none of the respective categories is known to be monoidal. 
Finally, observe that whereas the base algebra $A$ of a left bialgebroid $U$ carries a canonical $U$-module structure by the counit as well as a canonical left resp.\ right $U$-coaction by the source resp.\ target map, there is no canonical $U$-contraaction on $A$.

\subsubsection{Yetter-Drinfel'd algebras and anti Yetter-Drinfel'd modules}

The following standard concept (see, for example, \cite[\S4]{Schau:DADOQGHA} for the first part, and \cite[\S4]{BrzMil:BBAD} for the third) will be needed when establishing operadic structures on the complex in question; since we shall also consider its dual version, we need various versions of it.

\begin{dfn}
\
\begin{enumerate}
\compactlist{99}
\item
A {\em Yetter-Drinfel'd (YD) module $N$ over a left bialgebroid $(U,A)$} is a left $U$-module with action $U_\bract \otimes_\ahha N \to N$, $u \otimes_\ahha n \mapsto un$ and left $U$-comodule with coaction $N \to U_\ract \otimes_\ahha N$, $n \mapsto n_{(-1)} \otimes_\ahha n_{(0)}$ such that the underlying $A$-bimodule structures coincide and such that
\begin{equation}
\label{druento1}
 (u_{(1)}n)_{(-1)} u_{(2)} \otimes_\ahha (u_{(1)}n)_{(0)}= 
		  u_{(1)} n_{(-1)} \otimes_\ahha u_{(2)} n_{(0)} %\quad \forall \ u \in U, n \in N. 
\end{equation} 
for all $  u \in U, \ n \in N$. 
The resulting monoidal category will be denoted by $\yd$.
\item
A {\em Yetter-Drinfel'd (YD) module $M$ over a right bialgebroid $(V,B)$} is a right $V$-module with action $M \otimes_\behhe {}_\lact V \to M$, $m \otimes_\behhe v \mapsto mv$ and right $V$-comodule with coaction $M \to M \otimes_\behhe {}_\blact V$, 
$
m \mapsto m^{(0)} \otimes_\behhe m^{(1)}
$
such that the underlying $B$-bimodule structures coincide and such that
\begin{equation}
\label{druento2}
  (m v^{(2)})^{(0)} \otimes_\behhe v^{(1)} (m v^{(2)})^{(1)}= 
 	  m^{(0)} v^{(1)} \otimes_\behhe m^{(1)} v^{(2)} %\quad \forall \ v \in V, m \in M, 
\end{equation} 
for all $ v \in V, \ m \in M $. We denote the resulting monoidal category by $\ydr$.
\item
A {\em Yetter-Drinfel'd algebra} is an element $N$ in $\yd$ (and {\it mutatis mutandis} in $\ydr$) that is both a monoid in $\umod$ as well as in $\ucomod$, and which is called {\em braided commutative} if it is commutative with respect to the Yetter-Drinfel'd braiding, that is
\begin{equation}
\label{druento3}
n \cdot_\enne n' = (n_{(-1)}n') \cdot_\enne n_{(0)}, %\quad \forall \ n, n' \in N,
\end{equation}
for all $n, \ n' \in N$, where $- \cdot_\enne -$ indicates the monoid structure on $N$.
\end{enumerate}
\end{dfn}

Let us still mention that for a monoid $N$ in $\umod$ (as for example a Yetter-Drinfel'd algebra), one has the following properties for $m, n \in M, \ a, b  \in A$,
\begin{equation}
\label{bilet}
a \lact (m \cdot_\emme n) \ract b = (a \lact m) \cdot_\emme (n \ract b), 
\quad m \cdot_\emme (a \lact n) = (m \ract a) \cdot_\emme n
\end{equation}
with respect to the induced $\Ae$-module structures from \eqref{pergolesi}.

%, which will be used in later computations.

\bigskip

Furthermore, to define cyclic modules, we will need the right bialgebroid version of certain coefficient modules as defined in \cite[Def.~2.15]{BoeSte:CCOBAAVC}, which are in some sense opposite to Yetter-Drinfel'd modules:
 
\begin{dfn}
An {\em anti Yetter-Drinfel'd (aYD) module} $M$ over a right bialgebroid $(V,B)$ is a left $V$-module and right $V$-comodule such that the underlying $B$-bimodule structures arising from the forgetful functors coincide, that is,
\begin{equation}
\label{orvieto1}
a \lact m \ract b = amb, \quad  m \in M, \ a, b \in B,
\end{equation}
and such that action followed by coaction is given as
\begin{equation}
\label{orvieto2}
(vm)^{(0)} \otimes_\behhe (vm)^{(1)} = v^{(1)+} m^{(0)} \otimes_\behhe v^{(2)} m^{(1)} v^{(1)-}, \quad v \in V, \ m \in M.
\end{equation}
An anti Yetter-Drinfel'd module is called {\em stable} if coaction followed by action yields the identity, that is, $m^{(1)} m^{(0)} = m$. 
\end{dfn}

\begin{rem}
\label{nunzia}
In \cite[Prop.~4.4]{Schau:DADOQGHA}, a monoidal equivalence between $\yd$ and the {\em weak} centre of $\umod$ is established, whereas in  \cite[Lem.~6.1]{Kow:GABVSOMOO} it was proven in the bialgebroid context that aYD modules form a module category over $\yd$, see the comments below Remark \ref{polline}. On the other hand, in \cite[Prop.~2.7]{Sha:OTAYDMCC} it was shown that at least for a Hopf algebra $H$ the category of aYD modules can be identified with the centre of the bimodule category given by the opposite category of $H$-comodules. We presume that an analogous statement holds for bialgebroids when using the {\em weak} centre of a bimodule category \cite[Def.~2.1]{KobSha:ACATCCOQHAAHA}, see also Remark \ref{tatatatataaa} at the end of section \S\ref{t'n'c}.
\end{rem}

As we want to compare bialgebroids with their duals later on, we will finally need the following concept of an anti Yetter-Drinfel'd {\em contra\/}module but in order not to swamp the reader with too many definitions in a row, we will postpone it until Definition \ref{chelabertaschen1}.

\subsection{Duals}
\label{duals}
In this section we briefly recall how the right dual $U^* := \Hom_\Aopp(U_\ract, A_\ahha)$ becomes a right bialgebroid if the right $A$-module $U_\ract$ for the left bialgebroid $(U,A)$ is finitely generated projective over $A$; a similar discussion also holds for the left dual $U_* := \Hom_\ahha(\due U \lact {}, \due A \ahha {})$, which we omit. See \cite{KadSzl:BAODTEAD} 
and \cite[\S3.1]{Kow:HAATCT} for all notions and conventions used here. In longer expressions, we will frequently write $\langle \phi, u \rangle$ to mean $\phi(u)$ for $\phi \in U^*$, $u \in U$ as this increases readability (at least in our opinion).

The dualisation of the left bialgebroid
structure $(U, A, s^\ell, t^\ell, \gD_\ell, \gve)$  yields a right
bialgebroid $(U^*, A, s^r, t^r, \gD_r, \pl)$ over the same base algebra
if one imposes suitable
finiteness and projectivity
assumptions \cite{KadSzl:BAODTEAD}: 
the monoid structure exists in full
generality and is given by 
\begin{equation}
\label{LDMon}
(\phi \psi)(u) 
:= \langle \psi, \langle \phi, u_{(1)} \rangle \lact u_{(2)} \rangle,
%\quad \mbox{for all} \ \phi, \psi \in {U_*}, \ u \in U  
\end{equation}
where $\phi, \psi \in {U^*}, \ u \in
U$. This product can be promoted to 
an $A^e$-ring structure if one
defines the source and target maps as
\begin{equation*}
\label{LDSource}
 s^r : A \to  U^*, \ a \mapsto \gve((\cdot) \bract a), \quad  t^r : A \to  U^*, \ a \mapsto \gve(a \lact (\cdot)) = a \gve(\cdot).
\end{equation*}
If we denote, exactly 
as in \eqref{pergolesi} for every bialgebroid (be it left or right), 
the four $A$-module structures on $U^*$ as 
$$
a \lact \phi \ract b = s^r(a) t^r(b) \phi, \quad a \blact \phi \bract b = 
\phi \, t^r(a) s^r(b)
$$
for $\phi \in U^*$, $a,b \in A$, 
one can write down the identities
\begin{equation}
\label{duedelue}
\!\!\!\!\!\!\begin{array}{rclrclrcl}
\langle\phi , u \ract a \rangle  \!\!\!&=\!\!\!&
\langle \phi , u \rangle a, &
\langle \phi , a \lact u \rangle  \!\!\!&=\!\!\!&
\langle \phi \ract a  , u \rangle, &
\langle \phi , u \bract a \rangle  \!\!\!&=\!\!\!&
\langle a \lact \phi , u \rangle, \\
\langle \phi , a \blact u \rangle   \!\!\!&=\!\!\!&
 \langle \phi \bract a , u \rangle, &
\langle a \blact \phi, u \rangle  \!\!\!&=\!\!\!&
a \langle \phi , u \rangle,
\end{array}
\end{equation}
hence the pairing between the $\Ae$-rings $U^*$ and $U$ corresponds to what was called a {\em right $\Ae$-pairing} in \cite{CheGav:DFFQG}.

In order to obtain an $A$-coring structure $(U^*, \gD_r, \pl)$, 
one sets
\begin{equation}
\label{weiszjanich}
\!\!\! \begin{array}{rll}
\gD_r 
\!\!\! &
: U^* \to \Hom_\Aopp(\due U \blact {} \otimes_\Aopp U_\ract, A), \
&\!\! 
\phi  \mapsto \{ u \otimes_\Aopp v \mapsto \phi(uv) \}, 
 \\
\pl   
\!\!\! &
: U^* \to A,     
&\!\!
\phi \mapsto  \phi(1_\uhhu).
\end{array}
\end{equation}
When $U_\ract$ and hence also
$\due {U^*} \blact {}$ is finitely generated $A$-projective, the map 
\begin{equation}
\label{zuschlag}
%\begin{array}
U^*_\bract \otimes_\ahha \due {U^*} \blact {} \to
\Hom_\Aopp(\due U \blact {} \otimes_\Aopp U_\ract, A), \ 
\phi \otimes_\ahha \psi \mapsto \{ u  \otimes_\Aopp v \mapsto 
\langle \phi, \langle \psi , v \rangle \blact u \rangle \}, 
%\end{array}
\end{equation}
is an isomorphism and
$\gD_r$ given above defines a
(right) coproduct, that is
\begin{equation}
\label{trattovideo}
\langle \phi^{(1)}, \langle \phi^{(2)} , v \rangle \blact u \rangle = \langle \phi, uv \rangle,
\end{equation}
and $\pl$ gives the (right) counit for this.
Choosing a dual basis $\{e_i\}_{1 \leq i \leq n}  \in
U, \ \{e^i\}_{1 \leq i \leq n}  \in U^*$, one can decompose
\begin{equation}
\label{schizzaestrappa1}
u = \textstyle\sum_i e_i \ract \langle e^i, u \rangle,
\end{equation}
and from \eqref{duedelue} follows that if $U_\ract$ is finitely generated $A$-projective, then $\due {U^*} \blact {}$ is so as well with decomposition
$
\langle \phi, u \rangle 
= \sum_i \langle \phi,  e_i \ract \langle e^i, u \rangle \rangle 
= \sum_i \langle \phi, e_i \rangle \langle e^i, u \rangle 
= \sum_i \langle \langle \phi,  e_i \rangle \blact e^i, u \rangle, 
$ 
hence 
\begin{equation}
\label{schizzaestrappa2}
\phi = \textstyle\sum_i \langle \phi,  e_i \rangle \blact e^i.
\end{equation}
We conclude by explicitly writing down the coproduct \eqref{weiszjanich} by means of the left $U$-module structure 
\begin{equation}
\label{gianduiotto7}
(v \rightharpoonup  \phi)(u) :=  \phi(uv) 
%\qquad \hbox{for} \  u, v \in U, \  \phi \in  U_*. 
\end{equation}
on $U^*$ as
\begin{equation}
\label{braummadas}
%\begin{array}{rll}
\gD_r 
%\!\!\! &
: U^* \to U^*_\bract \otimes_\ahha \due {U^*} \blact {},
\quad 
%&
\phi \mapsto 
\textstyle\sum_i  (e_i \rightharpoonup \phi) \otimes_\ahha e^i. 
%\\
%\pl_*   \!\!\! &: U_* \to A,                                         \quad &\phi \mapsto  \phi(1_{U}).
%\end{array}
\end{equation}

\section{The Hopf structure on the dual}

A question left open for years was whether the dual(s) of a bialgebroid carries a (left or right) Hopf structure as well. 
A certain progress in this direction was made in \cite{CheGavKow:DFOLHA}, where a map $S^*: U^* \to U_*$ was defined that turns out to be an isomorphism if $U$ is both a left and a right Hopf algebroid (as a left bialgebroid). The ultimate answer that the dual(s) do carry a Hopf structure as well was given in 
\cite{Schau:TDATDOAHAAHA} 
by an abstract categorical reasoning, but no explicit expression for the translation map was given, a gap we are going to fill in this section. The idea is pretty simple: whereas the coproduct on $U^*$ is somehow the transpose of the left $U$-module structure \eqref{gianduiotto7} on $U^*$, the translation map, which can be interpreted as the inverse of the coproduct, results as the transpose of the left $U$-module structure \eqref{gianduiotto1a} on $U^*$. More precisely:

\begin{theorem}
\label{Hopfdual}
Let $(U,A)$ be a left bialgebroid which is additionally a left Hopf algebroid, and let $U_\ract$ be finitely generated $A$-projective. Then the right dual $(U^*,A)$ is a right bialgebroid which is additionally a right Hopf algebroid. 
More precisely, the map $\tilde\beta^{-1}: U^* \to \due {U^*} \blact {} \otimes_\Aopp U^*_\ract$ given by
\begin{equation}
\label{viviverde1}
({\tilde\gb}^{-1}(\phi))(u \otimes_\ahha v) = (u \rightslice \phi)(v) 
= \gve(\phi(u_-v) \blact u_+)
%= \gve(u_+ t^\ell\phi(u_-v))
\end{equation}
%for $u \otimes_\ahha v \in U_\ract \otimes_\ahha {}_\lact U$ 
yields a translation map which defines an inverse to the Hopf-Galois map 
\begin{equation}
\label{favini}
\gb: \due {U^*} \blact {} \otimes_\Aopp U^*_\ract \to U^*_\bract \otimes_\ahha \due {U^*} \blact {}, \quad \phi \otimes_\Aopp \psi \mapsto \phi \psi^{(1)} \otimes_\ahha \psi^{(2)}
\end{equation}
on $U^*$.
Explicitly, if $\{e_i\}_{1 \leq i \leq n} \in U, \ \{ e^i\}_{1 \leq i \leq n} \in U^*$ is a dual basis, the translation 
map reads as
\begin{equation}
\label{viviverde2}
\phi^- \otimes_\Aopp \phi^+ := \textstyle\sum_i e^i \otimes_\Aopp (e_i \rightslice \phi).
\end{equation}
\end{theorem}

\begin{rem}
\label{nocheins}
Comparing the coproduct \eqref{braummadas} with the translation map \eqref{viviverde2}, we notice that these expressions are built the same way and only (apart from using different tensor products) differ by the way $U$ acts on $U^*$.
Observe that for $v=1$ this yields the map $S^*: U^* \to U_*$ introduced in \cite{CheGavKow:DFOLHA}; see there for a discussion of this map.
Of course, a similar statement also holds for the left dual $U_*$, that is, if the left bialgebroid $U$ additionally were a {\em right} Hopf algebroid and $\due U \lact {}$ finitely generated $A$-projective, then the right bialgebroid $U_*$ became a left Hopf algebroid. If $U$ happens to be both a left and a right Hopf algebroid (and $A$-projective in two senses), then the two duals are isomorphic via the map $S^*$, see \cite[Thm.~5.1.3]{CheGavKow:DFOLHA}.
\end{rem}

\begin{proof}[Proof of Theorem \ref{Hopfdual}]
As the Hopf-Galois map is a morphism of right $U^*$-modules, it is enough to verify \eqref{Rch2} and \eqref{Rch3} with respect to \eqref{viviverde2} to show that \eqref{viviverde1} defines the inverse to the Hopf-Galois map \eqref{favini} for the right bialgebroid $U^*$.
One has
\begin{equation*}
\begin{split}
\langle & \phi^- \phi^{+(1)} \otimes_\ahha \phi^{+(2)} , u \otimes_\Aopp v \rangle 
\overset{{\scriptscriptstyle{\eqref{zuschlag}}}}{=} \langle  \phi^- \phi^{+(1)}, \langle \phi^{+(2)}, v\rangle \blact u\rangle \\
&\overset{{\scriptscriptstyle{\eqref{LDMon}}}}{=} \big\langle \phi^{+(1)} ,  \langle \phi^-, u_{(1)}\rangle \lact (  \langle \phi^{+(2)}, v\rangle \blact u_{(2)} ) \big\rangle \\
&\overset{{\scriptscriptstyle{\eqref{trattovideo}}}}{=} \langle \phi^+ ,
\langle \phi^-, u_{(1)}\rangle \lact u_{(2)}v  \rangle \\
&\overset{{\scriptscriptstyle{\eqref{viviverde2}}}}{=} \textstyle\sum_i \langle e_i \rightslice \phi , \langle e^i, u_{(1)}\rangle \lact u_{(2)}v  \rangle \\
&\overset{{\scriptscriptstyle{\eqref{gianduiotto1a}}}}{=} \textstyle\sum_i \big\langle \gve , \langle \phi , (e_{i-} \bract \langle e^i , u_{(1)} \rangle) u_{(2)} v \rangle  \blact e_{i+}\big\rangle \\
&\overset{{\scriptscriptstyle{\eqref{Sch6}, \eqref{Sch9}}}}{=} \textstyle\sum_i \big\langle \gve, \langle \phi, (e_{i} \ract \langle e^i , u_{(1)} \rangle)_- u_{(2)} v \rangle \blact (e_{i} \ract \langle e^i , u_{(1)} \rangle)_+  \big\rangle \\
 &\overset{{\scriptscriptstyle{\eqref{schizzaestrappa1}}}}{=} \langle \gve, \langle \phi,  u_{(1)-} u_{(2)}v \rangle \blact u_{(1)+} \rangle \\
 &\overset{{\scriptscriptstyle{\eqref{Sch3}}}}{=} \langle \gve,  \langle \phi, v \rangle \blact u \rangle \overset{{\scriptscriptstyle{\eqref{zuschlag}}}}{=} \langle \gve \otimes_\ahha \phi, u \otimes_\Aopp v \rangle. 
\end{split}
\end{equation*}
Hence, $\phi^- \phi^{+(1)} \otimes_\ahha \phi^{+(2)} = \gve \otimes_\ahha \phi$, which is \eqref{Rch2} for the right bialgebroid $(U^*, A)$.
Verifying also \eqref{Rch3} along these lines is left to the reader.
\end{proof}
For convenience in later computations, note the following useful relation:
\begin{lem}
For $\phi \in U^*$ and $u \in U$, one can write
\begin{equation}
\label{sondreck}
\phi^+ \ract \langle \phi^-, u \rangle = u \rightslice \phi. 
\end{equation}
\end{lem}

\begin{proof}
As the desired identity lives in $U^*$, we are going to prove it by pairing both sides with an arbitrary element $w \in U$. 
One has
\begin{eqnarray*}
%\begin{split}
\big\langle \phi^+ \ract \langle \phi^-, u \rangle, w\big\rangle  
&
\!\!\!\!\!\!
\overset{{\scriptscriptstyle{\eqref{duedelue}}}}{=}
&
\!\!\!\!\!\!
\big\langle \phi^+, \langle \phi^-, u \rangle \lact  w\big\rangle  
\overset{{\scriptscriptstyle{\eqref{viviverde2}}}}{=}
\textstyle \sum_i \big\langle e_i \rightslice \phi, \langle e^i, u \rangle \lact  w\big\rangle   
 \\
 &
\!\!\!\!\!\!
 \overset{{\scriptscriptstyle{\eqref{gianduiotto1a}}}}{=}
&
\!\!\!\!\!\!
\textstyle \sum_i \gve\big( \big\langle \phi, (e_{i-} \bract \langle e^i, u \rangle) w \big\rangle\blact e_{i+}  \big)   
\\
&
\!\!\!\!\!\!
\overset{{\scriptscriptstyle{\eqref{Sch6}, \eqref{Sch9}, \eqref{schizzaestrappa1}}}}{=}
&
\!\!\!\!\!\!
 \textstyle \gve(\langle \phi, u_- w \rangle \blact u_+)
 = \langle u \rightslice \phi, w \rangle,
%\end{split}
\end{eqnarray*}
for all  $w \in U$. 
\end{proof}

\section{Cyclic structures on the complex computing $\Ext$}

The idea in this section is to obtain the structure of a cyclic module on the complex that computes the $\Ext$-groups as an application of the Hopf structure on the dual, since we now know that $(U^*,A)$ is a right Hopf algebroid. 
To start with, consider the following right bialgebroid version of \cite[Thm.~3.6--3.7]{KowKra:CSIACT}:

\begin{lem}
\label{sofocle1}
Let $(V,B,\gD_r, \pl)$ be a right bialgebroid which is additionally a right Hopf algebroid in the sense of Eqs.~\eqref{ilariabianchi2}--\eqref{ilariabianchi3}. Let $M$ be a left $V$-module right $V$-comodule with compatible respective right $B$-actions $mb = m \ract b$ for $m \in M$ and $b \in B$.
Then the quadruple $(C^\bull_\co(V,M), \gd'_\bull,
 \gs'_\bull, \tau')$ defined as 
$$
C^\bull_\co(V,M) := M \otimes_\behhe V \otimes_\behhe \cdots \otimes_\behhe V
$$
with operations
\begin{equation}
\!\!\!\!\!\! \begin{array}{rll}
\label{anightinpyongyang2}
\gd'_i(m \otimes_\behhe w) \!\!\!\!&= \left\{\!\!\!
\begin{array}{l} 
m \otimes_\behhe v_1 \otimes_\behhe \cdots
 \otimes_\behhe v_n \otimes_\behhe 1  
\\ 
m \otimes_\behhe v_1 \otimes_\behhe \cdots \otimes_\behhe \Delta_r (v_{n-i+1}) \otimes_\behhe \cdots
 \otimes_\behhe v_n
\\
m^{(0)} \otimes_\behhe m^{(1)} \otimes_\behhe v_1 \otimes_\behhe \cdots \otimes_\behhe v_n 
\end{array}\right. \!\!\!\!\!\!\!\! 
& \!\! \hspace*{-1.4cm}  \begin{array}{l} \mbox{if} \ i=0, \\ \mbox{if} \
  1 \leq i \leq n, \\ \mbox{if} \ i = n + 1,  \end{array} 
\\
\
\\
\gd'_i(m) \!\!\!\! &= \left\{ \!\!\!
\begin{array}{l}
		  m 
		  \otimes_\behhe 1  \quad
\\

m^{(0)} \otimes_\behhe m^{(1)}  \quad 
\end{array}\right. & \!\! \hspace*{-1.4cm} 
\begin{array}{l} \mbox{if} \ i=0, \\ \mbox{if} \
  i = 1 ,  \end{array} 
\\
\
\\
\gs'_j(m \otimes_\behhe w) \!\!\!\! 
&= m \otimes_\behhe v_1 \otimes_\behhe \cdots \otimes_\behhe
\pl (v_{n-j}) \otimes_\behhe \cdots \otimes_\behhe v_n, & \! \,
\hspace*{1pt} \hspace*{-1.8cm}  0 \leq j \leq n-1,
\\
\
\\
\tau'(m \otimes_\behhe w) \!\!\!\! 
&=  v^+_n m^{(0)} \otimes_\behhe m^{(1)}v_n^{-(1)}  \otimes_\behhe v_1 v_n^{-(2)}  \otimes_\behhe \cdots \otimes_\behhe v_{n-1} v_n^{-(n)} & 
\end{array}
\end{equation}
defines a para-cocyclic $k$-module (where we abbreviated $w:=v^1 \otimes_\behhe
\cdots \otimes_\behhe v^n$).
If $M$ is a stable aYD module over the right bialgebroid $V$, then $(C^\bull_\co(V,M), \gd'_\bull,
 \gs'_\bull, \tau')$ is a cocyclic 
$k$-module.
\end{lem}

\begin{proof}
The proof of this fact works along the lines of the corresponding result in {\it loc.~cit.}, which is why we omit it.
\end{proof}

The cosimplicial $k$-module in \eqref{anightinpyongyang2} only uses the right bialgebroid structure along with the right $V$-comodule structure of $M$. Putting 
 $\gb' := \sum^{n+1}_{i=0} (-1)^i \gd'_i$, we therefore define:

\begin{dfnlem}
The (simplicial) cohomology computed by $(C^\bull_\co(V,M), \gb')$ for a right bialgebroid $(V,B)$ and a right $V$-comodule $M$ is denoted by $H^\bull_\co(V,M)$ and called the {\em coHochschild cohomology} of $V$ with values in $M$; if $\due V \blact {}$ is flat as a $B$-module,
then $H^\bull_\co(V,M) \simeq \Cotor^\bull_V(M,B)$, see \cite{KowKra:CSIACT}.
\end{dfnlem}

We shall apply this lemma to the right bialgebroid $(U^* , A)$ given by the right dual of a left bialgebroid $(U,A)$ to obtain the structure of a cocylic $k$-module structure on 
\begin{equation}
\label{baumbluetenfest}
C^\bull(U,M) := \Hom_\Aopp(U^{\otimes_\Aopp \bull}, M),
\end{equation} 
where the tensor product is formed with respect to the $A$-bimodule structure $\due U \blact \ract$ and $M$ is a right $U$-contramodule and left $U$-module.

\subsection{Contramodules as coefficients}
\label{atacvantaggi}
To obtain cyclicity, we need to impose a compatibility condition between the $U$-action and the $U$-contraaction, similar to the case of aYD modules above, leading to so-called aYD {\em contra\/}modules, which is the content of this subsection. 
In the context of Hopf algebras, the corresponding definition is due to \cite{Brz:HCHWCC}
 again, but for noncommutative base rings, {\it i.e.}, for (left) Hopf algebroids we are not aware of any reference:

\begin{dfn}
\label{chelabertaschen1}
An {\em anti Yetter-Drinfel'd (aYD) contramodule} $M$ over a left Hopf algebroid $(U,A)$ is simultaneously a left $U$-module and a right $U$-contramodule such that both underlying $A$-bimodule structures from \eqref{pergolesi} and \eqref{gaeta} coincide, 
that is
\begin{equation}
\label{romaedintorni}
a \lact m \ract b = amb, \qquad m \in M, \ a,b \in A,
\end{equation}
and such that action and contraaction are compatible in the sense that 
\begin{equation}
\label{nawas1}
u (\gamma(f)) = \gamma \big(u_{+(2)} f(u_-(-)u_{+(1)}) \big), \qquad \forall u \in U, \ f \in \Hom_\Aopp(U,M).
%= \gamma \big((u_{(2)} \rightslice f)((-)u_{(1)}) \big)
\end{equation}
% for all $u \in U, \ f \in \Hom_\Aopp(U,M)$. 
An anti  Yetter-Drinfel'd contramodule is called {\em stable} if 
\begin{equation}
\label{stablehalt}
\gamma((-)m)= m
\end{equation}
for all $m \in M$, where we denote %$ \Hom_\Aopp(U,M) \ni 
$(-)m \colon u \mapsto um$ as a map in $ \Hom_\Aopp(U,M)$.
\end{dfn}

\begin{rem}
\quad
\begin{enumerate}
\compactlist{99}
\item
The left $U$-action that appears in the argument of $\gamma$ on the right hand side of \eqref{nawas1} is of course that on $M$: the argument of $\gamma$ has therefore to be understood as the element in $\Hom_\Aopp(U, M)$ given by $U \ni w \mapsto u_{+(2)} f(u_- w u_{+(1)})$, see below for the issue of well-definedness. For a conceptually somewhat cleaner notation (but with less similarity to aYD modules) one could rewrite the condition \eqref{nawas1} as
\begin{equation}
\label{nawas2}
u (\gamma(f)) = \gamma \big((u_{(2)} \rightslice f)((-) u_{(1)}) \big),
\end{equation}
using \eqref{Sch4} and the left $U$-action \eqref{gianduiotto1} on $\Hom_\Aopp(U, M)$.
\item
That the right hand side from \eqref{nawas1} or \eqref{nawas2} is actually well-defined might appear, at first glance, somewhat mysterious since on an arbitrary $U$-bimodule $N$ over a left Hopf algebroid $U$ there is in general {\em no} well-defined adjoint action of the kind $n \mapsto u_- n u_+$ for $n \in N, \ u \in U$. On the other hand, \eqref{nawas1} {\em does} indeed make sense: let us prove that the right hand side in \eqref{nawas1} does not depend on the choice of a representative for $u_+ \otimes_\Aopp u_-$ resp.\ $u_{(1)} \otimes_\ahha u_{(2)}$ in $U \otimes_\Aopp U$  resp.\ $U \otimes_\ahha U$. Whereas for the first case this is obvious by the $\Aop$-linearity of both $f$ and the left coproduct on $U$, the second case turns out to be slightly more tricky: first of all,
\begin{equation}
\label{umsonst}
\begin{split}
\gamma(a \lact f(-)) 
& 
{\overset{\scriptscriptstyle{
\eqref{romaedintorni}, \eqref{alleskleber} 
}}{=}}
\dot\gamma\big(\ddot\gamma\big(f(\cdot)\gve(a \blact \cdot\cdot)\big)\big) 
\\
& 
{\overset{\scriptscriptstyle{
\eqref{carrefour1} 
}}{=}}
\gamma\big(f(-_{(1)})\gve(a \blact -_{(2)})\big)
%\\
%& 
{\overset{\scriptscriptstyle{
%\eqref{} 
}}{=}}
\gamma\big(f(a \blact -)\big),
\end{split}
\end{equation} 
where we used counitality and the $\Aop$-linearity of $f$ in the last step. With this identity, it is then simple to see that the right hand side of \eqref{nawas2} does not differ for two representatives
$\sum_i u'_i \ract a \otimes u''_i$ and  $\sum_i u'_i  \otimes a \lact u''_i$ of the coproduct $\gD_\ell(u)$ for an element $u \in U$; hence, \eqref{nawas2} is well-defined. 
\item
Anti Yetter-Drinfel'd contramodules over a (left or right) bialgebroid form a category, which again is unlikely to be monoidal. To start with, the respective base algebra of a (left or right) bialgebroid in general is not an aYD contramodule, and, as said before, usually not even a contramodule at all.
\end{enumerate}
\end{rem}

\begin{example}
In case of the group Hopf algebra $kG$ for an infinite discrete group $G$,
the concept of contramodules is illustrated in \cite[\S4.1]{Sha:OTAYDMCC}: 
the category of $kG$-contramodules turns out to be equivalent to that of $G$-graded vector spaces, whereas the category of aYD-contramodules over $G$ is equivalent to that of $G$-graded $G$-equivariant vector spaces, as is the category of aYD modules over $kG$. We refer to {\em loc.~cit.}~for details as well as for a discussion about stability and cyclic cohomology in this case.
\end{example}

To continue, we need to discuss how these coefficient modules transform if passing to the dual. 
Most of the statements in the following lemma are well-known but perhaps not explicitly written down (and proven) for the bialgebroid setting:

\begin{lem}
\label{sofocle2}
Let $(U,A)$ be a left bialgebroid.
\begin{enumerate}
\compactlist{99}
\item
\label{uhu1}
There is a functor  
$  
\ucomod \to \mathbf{Mod}\mbox{-}{U^*},
$
 and
%
%  namely, if $N$ is a left $U$-comodule with coaction
% $n \mapsto n_{(-1)} \otimes_\ahha n_{(0)}  $, then
% \begin{equation}
% \label{pepe}
% N \otimes_k U^* \to N  , \quad n \otimes_k \phi \mapsto \phi(n_{(-1)})n_{(0)}  ,
% \end{equation}
% defines a right module structure over the $\Ae$-ring $U^*$.
if $\due U {}\ract$ is finitely generated
 $A$-projective,
 this functor is monoidal and has a quasi-inverse $\mathbf{Mod}\mbox{-}{U^*} \to \ucomod$ giving a (strict) monoidal equivalence
$$
 \ucomod \simeq \mathbf{Mod}\mbox{-}{U^*}
 $$
 of categories.
\item
\label{uhu2}
If $U_\ract$ is finitely generated $A$-projective, then 
there is a (strict) monoidal 
equivalence 
$$
\mathbf{Comod}\mbox{-}U^* \simeq {U\mbox{-}\mathbf{Mod}} 
$$
of categories.
% between the categories of right $U^*$-comodules and left $U$-modules.
\item
\label{uhu3}
If $U_\ract$ is finitely generated $A$-projective, then by means of the equivalences in (i) and (ii) 
there is a braided
(strict)
 monoidal 
equivalence 
$$
\yddual \simeq \yd.
$$
\item
\label{uhu4}
There is a 
functor
$
\mathbf{Contramod}\mbox{-}U \to {U^*\mbox{-}\mathbf{Mod}},  
$
%between the categories of 
%right $U$-contramodules and left $U^*$-modules, 
which, if $U_\ract$ is finitely generated $A$-projective, has a quasi-inverse 
$ {U^*\mbox{-}\mathbf{Mod}} \to \mathbf{Contramod}\mbox{-}U$ 
giving an equivalence
$$
 \mathbf{Contramod}\mbox{-}U \simeq {U^*\mbox{-}\mathbf{Mod}}  
$$ 
of categories.
\item
\label{uhu5}
Let $U$ be additionally a left Hopf algebroid and let 
 $U_\ract$ be finitely generated $A$-projective.
A left $U^*$-module right $U^*$-comodule (with compatible right $A$-actions) is (stable) aYD if it is so as a left $U$-module right $U$-contramodule by means of the equivalences in \eqref{uhu2} and \eqref{uhu4}; hence one has an 
equivalence of categories
$$
{}_{\scriptscriptstyle{U^*}}\mathbf{aYD}^{\scriptscriptstyle{U^*}} \simeq {}_\uhhu \mathbf{aYD}^{\scriptscriptstyle{\rm contra-}\uhhu}
$$
between the categories of (stable) aYD modules over the right bialgebroid $U^*$ and (stable) aYD contramodules over $U$.
\end{enumerate}
\end{lem}

\begin{proof} 
The first four parts of this lemma are not too surprising results, some of them well-known, which is why we only give some hints and not discuss, for example, the respective morphisms.
% Let $U_\ract$ be finitely generated $A$-projective and let correspondingly 
% $\{e_i\}_{1 \leq i \leq n} \in U, \ \{ e^i\}_{1 \leq i \leq n} \in U^*$ be 
% a dual basis.
%\begin{enumerate} 
%\compactlist{99}
%\item

Part (i):
this was proven, along with an analogous consideration for the left dual, in \cite[Thm.~3.1.11]{Kow:HAATCT} and also appeared earlier in \cite[\S5]{Schau:DADOQGHA}; see \cite[Prop.~4.2.1]{CheGavKow:DFOLHA} as well. For later use, we mention that given an ($\Aop$-balanced) right $U^*$-module structure %$M \otimes_\ahha \due {U^*} \lact {} \to M, \ 
$m \otimes_\Aopp \phi \mapsto m\phi$ on a left $A$-module $M$, then 
\begin{equation}
\label{lesbrigittes1}
%M \to U_\ract \otimes_\ahha M, \quad 
m \mapsto \textstyle\sum_i e_i \otimes_\ahha me^i
\end{equation}
defines a left $U$-comodule structure on $M$, where 
$\{e_i\}_{1 \leq i \leq n} \in U, \ \{ e^i\}_{1 \leq i \leq n} \in U^*$ is a dual basis; vice versa, for a left $U$-comodule $M$, the assignment
\begin{equation}
\label{lesbrigittes2}
m\phi := \langle \phi, m_{(-1)}\rangle m_{(0)}
\end{equation}
defines a right $U^*$-module structure on $M$.

%\item
Part (ii):
similarly to the first part, consider the map
\begin{equation}
\label{vetrorotto1}
M \otimes_\ahha \due {U^*} \blact {} \to \Hom_\Aopp(U,M), \quad m \otimes_\ahha \phi 
\mapsto \big\{ u \mapsto m \langle \phi, u \rangle \big\},
\end{equation}
which is an isomorphism if $U_\ract$ is finitely generated $A$-projective,
with inverse
\begin{equation}
\label{vetrorotto2}
 \Hom_\Aopp(U,M) \to M \otimes_\ahha U^*, \quad f \mapsto \textstyle\sum_i f(e_i) \otimes_\ahha e^i,
\end{equation}
where $\{e_i\}_{1 \leq i \leq n} \in U, \ \{ e^i\}_{1 \leq i \leq n} \in U^*$ is a dual basis.
As is a straightforward check using \eqref{trattovideo} and the coassociativity of $M$, the map
\begin{equation}
\label{vetrorotto3}
U_\bract \otimes_\ahha M \to M, \quad u \otimes_\ahha m \mapsto um := m^{(0)} \langle m^{(1)}, u \rangle
\end{equation}
defines a {\em left} $U$-module structure on $M$. Conversely, if $u \otimes_\ahha m \mapsto um$ is a left $U$-module structure, then
\begin{equation}
\label{vetrorotto4}
M \to M \otimes_\ahha \due {U^*} \blact {}, \quad m \mapsto \textstyle\sum_i e_im \otimes_\ahha e^i 
\end{equation}
defines a right $U^*$-comodule structure on $M$, which is easily verified by means of \eqref{trattovideo}. That both constructions are mutually inverse is easily seen by the decomposition 
% $u = \sum_i e_i \ract \langle e^i, u \rangle$ 
\eqref{schizzaestrappa1}
as well as the forgetful functor \eqref{pergolesi}. 
As for the claim that the so-given functor $F:\mathbf{Comod}\mbox{-}U^* \to \umod$ is monoidal, recall that the monoidal structure on $\umod$ is given by the diagonal action, that is, the tensor product $M' \otimes_\ahha N'$ of two $U$-modules is a $U$-module by defining the left action $u(m' \otimes_\ahha n') := u_{(1)}m' \otimes_\ahha u_{(2)} n'$, whereas the tensor product $M \otimes_\ahha N$ of two right $U^*$-comodules becomes a right $U^*$-comodule by the codiagonal coaction, that is, by the right coaction $m \otimes_\ahha n \mapsto (m^{(0)} \otimes_\ahha n^{(0)}) \otimes_\ahha m^{(1)} n^{(1)}$. The prescription \eqref{vetrorotto3} then defines on $M \otimes_\ahha N$ the left $U$-action
\begin{equation*}
\begin{split}
u(m \otimes_\ahha n) &= (m^{(0)} \otimes_\ahha n^{(0)})\langle m^{(1)} n^{(1)}, u \rangle 
\\
&
{\overset{\scriptscriptstyle{
\eqref{LDMon}
}}{=}} 
m^{(0)} \otimes_\ahha n^{(0)} \langle n^{(1)}, \langle m^{(1)}, u_{(1)} \rangle \lact u_{(2)} \rangle 
\\
&
{\overset{\scriptscriptstyle{
\eqref{duedelue}
}}{=}} 
m^{(0)} \otimes_\ahha n^{(0)} \langle n^{(1)} \ract  \langle m^{(1)}, u_{(1)} \rangle, u_{(2)} \rangle
\\
&
{\overset{\scriptscriptstyle{
\eqref{Takeuchicoaction}
}}{=}}
m^{(0)} \langle m^{(1)}, u_{(1)} \rangle \otimes_\ahha n^{(0)} \langle n^{(1)}, u_{(2)} \rangle
{\overset{\scriptscriptstyle{
}}{=}}
u_{(1)} m \otimes_\ahha u_{(2)} n,
\end{split}
\end{equation*}
hence $F(M \otimes_\ahha N) = FM \otimes_\ahha FN$, and that $F(A) = A$ for the unit object in both categories is also easy to see.

%\item
Part (iii):
for this, use the first two parts of this lemma and show that under these equivalences the (right bialgebroid) Yetter-Drinfel'd condition \eqref{druento2} follows if the (left bialgebroid) Yetter-Drinfel'd condition \eqref{druento1} is fulfilled (and vice versa, but we only show the first case). So, let $M \in \yd$ and using the notation $\phi = \langle \phi, \cdot \rangle$ for an element in $U^*$, the left hand side in \eqref{druento2} can then for any $u \in U$ be expressed as
\begin{eqnarray*}
%\begin{split}
&&  (m  \phi^{(2)})^{(0)} \otimes_\ahha \langle \phi^{(1)} (m \phi^{(2)})^{(1)}, u \rangle 
\\
&
{\overset{\scriptscriptstyle{
\eqref{LDMon}
}}{=}} 
&
 (m \phi^{(2)})^{(0)} \otimes_\ahha \big\langle (m \phi^{(2)})^{(1)}, \langle \phi^{(1)} , u_{(1)} \rangle \lact u_{(2)} \big\rangle 
\\
&
{\overset{\scriptscriptstyle{
\eqref{lesbrigittes2}
}}{=}} 
&
(\langle \phi^{(2)}, m_{(-1)} \rangle m_{(0)} )^{(0)} \otimes_\ahha \big\langle (\langle \phi^{(2)}, m_{(-1)} \rangle m_{(0)})^{(1)}, \langle \phi^{(1)} , u_{(1)} \rangle \lact u_{(2)} \big\rangle 
\\
&
{\overset{\scriptscriptstyle{
\eqref{duedelue}, \eqref{tellmemore}
}}{=}} 
&
{m_{(0)}}^{(0)} \otimes_\ahha \Big\langle {m_{(0)}}^{(1)}, \big\langle \phi^{(1)} , \langle \phi^{(2)}, m_{(-1)} \rangle \blact u_{(1)} \big\rangle \lact u_{(2)} \Big\rangle 
\\
&
{\overset{\scriptscriptstyle{
\eqref{trattovideo}
}}{=}} 
&
{m_{(0)}}^{(0)} \otimes_\ahha \big\langle {m_{(0)}}^{(1)}, \langle \phi, u_{(1)} m_{(-1)} \rangle \lact u_{(2)} \big\rangle 
\\
&
{\overset{\scriptscriptstyle{
\eqref{Takeuchicoaction}, \eqref{vetrorotto4}
}}{=}} 
&
\textstyle\sum_i \langle \phi, u_{(1)} m_{(-1)} \rangle \lact (e_i m_{(0)}) \otimes_\ahha \langle e^i, u_{(2)} \rangle 
\\
&
{\overset{\scriptscriptstyle{
%\eqref{}
}}{=}} 
&
\textstyle\sum_i \langle \phi, u_{(1)} m_{(-1)} \rangle \lact \big((e_i \ract \langle e^i, u_{(2)} \rangle )  m_{(0)}\big) \otimes_\ahha 1 
\\
&
{\overset{\scriptscriptstyle{
\eqref{schizzaestrappa1}
}}{=}} 
&
\langle \phi, u_{(1)} m_{(-1)} \rangle \lact (u_{(2)}m_{(0)}) \otimes_\ahha 1 
\\
&
{\overset{\scriptscriptstyle{
\eqref{druento1}
}}{=}} 
&
\langle \phi, (u_{(1)} m)_{(-1)} u_{(2)} \rangle \lact (u_{(1)}m)_{(0)} \otimes_\ahha 1 
\\
&
{\overset{\scriptscriptstyle{
\eqref{trattovideo}
}}{=}}
&
\big\langle \phi^{(1)}, \langle \phi^{(2)} , u_{(2)} \rangle \blact (u_{(1)} m)_{(-1)}  \big\rangle \lact (u_{(1)}m)_{(0)} \otimes_\ahha 1 
\\
&
{\overset{\scriptscriptstyle{
%\eqref{trattovideo}
}}{=}}
&
\big\langle \phi^{(1)}, (u_{(1)} m)_{(-1)}  \big\rangle (u_{(1)}m)_{(0)} \otimes_\ahha \langle \phi^{(2)} , u_{(2)} \rangle 
\\
&
{\overset{\scriptscriptstyle{
\eqref{lesbrigittes2}
}}{=}}
&
(u_{(1)} m) \phi^{(1)}  \otimes_\ahha \langle \phi^{(2)} , u_{(2)} \rangle 
\\
&
{\overset{\scriptscriptstyle{
\eqref{vetrorotto3}
}}{=}}
&
(m^{(0)} \langle m^{(1)}, u_{(1)} \rangle) \phi^{(1)}  \otimes_\ahha \langle \phi^{(2)} , u_{(2)} \rangle 
\\
&
{\overset{\scriptscriptstyle{
\eqref{tellmemore}
}}{=}}
&
m^{(0)} \phi^{(1)}  \otimes_\ahha \big\langle \phi^{(2)} \ract  \langle m^{(1)}, u_{(1)} \rangle , u_{(2)} \big\rangle 
\\
&
{\overset{\scriptscriptstyle{
\eqref{duedelue}, \eqref{LDMon}
}}{=}}
&
m^{(0)} \phi^{(1)} \otimes_\ahha \langle m^{(1)} \phi^{(2)}, u \rangle,
%\end{split}
\end{eqnarray*}
where in step three we additionally used the $A$-linearity of the right coaction, in step six that the left and right $A$-actions on $U$ commute, and in the tenth step the left bialgebroid comodule version of the subspace property \eqref{Takeuchicoaction}. Since this computation holds for arbitrary $u \in U$, Equation \eqref{druento2} follows for the right bialgebroid $(U^*, A)$, as desired. As for the claim that the equivalence is braided, similar to the Hopf algebra case (see, for example, \cite{AndGra:BHAONAFG}) we have to show that the respective braidings $\gs_\uhhu$ in $\yd$ and $\gs_{\scriptscriptstyle{U^*}}$ in $\yddual$ commute with the functor $F: \yddual \to \yd$, that is, for objects $M, N \in \yd$ one has $F\gs_\uhhu = \gs_{\scriptscriptstyle{U^*}} F: M \otimes_\ahha N \to N \otimes_\ahha M$. One computes
\begin{equation*}
\begin{split}
\gs_{\scriptscriptstyle{U^*}}(m \otimes_\ahha n) 
&=
n^{(0)} \otimes_\ahha mn^{(1)} 
{\overset{\scriptscriptstyle{
\eqref{vetrorotto4}
}}{=}}
\textstyle\sum_i e_in \otimes_\ahha me^i 
{\overset{\scriptscriptstyle{
\eqref{lesbrigittes2}
}}{=}}
\textstyle\sum_i e_in \otimes_\ahha \langle e^i, m_{(-1)} \rangle m_{(0)}
\\
&= 
 \textstyle\sum_i (e_i \ract  \langle e^i, m_{(-1)} \rangle) n \otimes_\ahha m_{(0)}
 {\overset{\scriptscriptstyle{
\eqref{schizzaestrappa1}
}}{=}}
 m_{(-1)} n \otimes_\ahha m_{(0)}
=
\gs_\uhhu(m \otimes_\ahha n), 
\end{split}
\end{equation*}
where we wrote $m \otimes_\ahha n$ for both an element in $M \otimes_\ahha N$ as well as for the corresponding element in $F(M) \otimes_\ahha F(N)$.

%\item
Part (iv):
here, we simply make use of the map \eqref{vetrorotto1} along with \eqref{vetrorotto2} again: if $\gamma: \Hom_\Aopp(U,M) \to M$ is a right $U$-contramodule and $m \langle \phi, \cdot \rangle$ denotes the element in $ \Hom_\Aopp(U,M)$ defined by $ u \mapsto m \langle \phi, u \rangle $ for $\phi \in U^*$ and $m \in M$, then 
\begin{equation}
\label{bellitalia}
\due {U^*} \blact {} \otimes_\Aopp M_\ahha \to M, \quad \phi \cdot m := \gamma(m \langle \phi, - \rangle)
\end{equation}
yields a left $U^*$-module structure on $M$. Indeed, for $\phi, \psi \in U^*$ and $m \in M$, one has
\begin{equation*}
\begin{split}
(\phi\psi) \cdot m &= \gamma(m \langle \phi\psi, - \rangle)
\\
&
{\overset{\scriptscriptstyle{
\eqref{LDMon}
}}{=}} 
\gamma(m \langle \psi, \langle  \phi, (-)_{(1)} \rangle \lact (-)_{(2)} \rangle)
\\
&
{\overset{\scriptscriptstyle{
\eqref{carrefour1}
}}{=}} 
\dot\gamma\big(\ddot\gamma(m \langle \psi, \langle \phi, \cdot \rangle \lact \cdot\cdot \rangle)\big)
\\
&
{\overset{\scriptscriptstyle{
\eqref{passionant}
}}{=}} 
\dot\gamma\big(\ddot\gamma(m \langle \psi,\cdot\cdot \rangle)  \langle \phi, \cdot \rangle \big)
\\
&
{\overset{\scriptscriptstyle{
\eqref{bellitalia}
}}{=}} 
\gamma\big((\psi \cdot m) \langle \phi, - \rangle \big)
= \phi \cdot (\psi \cdot m).
\end{split}
\end{equation*}
Conversely, if $U_\ract$ is finitely generated $A$-projective and 
$\due {U^*} \blact {} \otimes_\Aopp M \to M$, 
$\phi \otimes_\Aopp m \mapsto \phi \cdot m$ a left $U^*$-action on $M$, then 
\begin{equation}
\label{mistmistmist}
\gamma(f) := \textstyle\sum_i e^i \cdot f(e_i), \qquad f \in \Hom_\Aopp(U,M)
\end{equation}
gives a right $U$-contraaction: we only verify \eqref{carrefour1} and leave the rest to the reader. To this end, first recall from, {\it e.g.}, \cite[\S3]{Kow:HAATCT} the isomorphism 
\begin{equation*}
\begin{split}
U_\ract \otimes_\ahha \due U \lact {} &\to \Hom_\ahha(\due {U^*} \blact {} \otimes_\Aopp U^*_\ract, A), 
\\
% \quad
% $
% given by 
% $
u \otimes_\ahha v &\mapsto \{\phi \otimes_\Aopp \psi \mapsto \langle \psi, \langle \phi, u \rangle \lact v \rangle \}
\end{split}
\end{equation*} 
with inverse
$
\Psi \mapsto \textstyle\sum_{i,j} e_i \otimes_\ahha e_j \ract \Psi(e^j \otimes_\Aopp e^i),
$
which allows to write the coproduct on $U$ as
$$
\gD_\ell u = \textstyle\sum_{i,j} e_i \otimes_\ahha e_j \ract \langle e^i e^j , u \rangle.
$$
We then have for $g \in \Hom_\Aopp(U_\ract \otimes_\ahha \due U \lact {},M)$
\begin{equation*}
\begin{array}{rcl}
\gamma\big(g(-_{(1)} \otimes_\ahha -_{(2)})\big)
&
{\overset{\scriptscriptstyle{
\eqref{mistmistmist}
}}{=}} 
&
 \textstyle\sum_{k} e^k \cdot g(e_{k (1)} \otimes_\ahha e_{k (2)})
\\
&
{\overset{\scriptscriptstyle{
}}{=}}
&
 \textstyle\sum_{i,j,k} e^k \cdot g( e_i \otimes_\ahha e_j \ract \langle e^i e^j , e_k \rangle )
\\
&
{\overset{\scriptscriptstyle{
}}{=}}
&
 \textstyle\sum_{i,j,k} e^k \cdot \big( g( e_i \otimes_\ahha e_j) \langle e^i e^j , e_k \rangle \big)
\\
&
{\overset{\scriptscriptstyle{
\eqref{pergolesi}
}}{=}}
&
 \textstyle\sum_{i,j,k} (\langle e^i e^j , e_k \rangle \blact e^k) \cdot g( e_i \otimes_\ahha e_j)
\\
&
{\overset{\scriptscriptstyle{
\eqref{schizzaestrappa2}
}}{=}}
&
 \textstyle\sum_{i,j} (e^i e^j) \cdot g( e_i \otimes_\ahha e_j)
% \\
% &
% {\overset{\scriptscriptstyle{
% %
% }}{=}} 
% \textstyle\sum_{i,j} e^i \cdot (e^j \cdot g( e_i \otimes_\ahha e_j))
\\
&
{\overset{\scriptscriptstyle{
\eqref{mistmistmist}
}}{=}}
&
 \textstyle\sum_{i} e^i \cdot \big(\gamma( g( e_i \otimes_\ahha -)) \big)
=  \dot\gamma\big(\ddot\gamma( g(\cdot \otimes_\ahha \cdot\cdot)) \big),
\end{array}
\end{equation*}
that is, we showed \eqref{carrefour1}.

%\item
Part (v):
here, we only show that the condition \eqref{orvieto2} is transformed into \eqref{nawas1} when applying the functors $\mathbf{Comod}\mbox{-}U^* \to {U\mbox{-}\mathbf{Mod}}$ and ${U^*\mbox{-}\mathbf{Mod}} \to \mathbf{Contramod}\mbox{-}U$ (and their quasi-inverses) from parts \eqref{uhu2} and \eqref{uhu4}, and leave the missing details to the reader. So, let $M$ be an aYD module over $U^*$, that is, a left $U^*$-module and right $U^*$-comodule that fulfils \eqref{orvieto2}, and let as before  $\{e_i\}_{1 \leq i \leq n} \in U, \ \{ e^i\}_{1 \leq i \leq n} \in U^*$ be a dual basis. We compute
\begin{equation*}
\begin{array}{rcl}
u\big(\gamma(f)\big) 
&
%\!\!\!\!\!\!\!\!
\!\!\!\!\!\!\!\!\!\!
{\overset{\scriptscriptstyle{
\eqref{mistmistmist}
}}{=}} 
&
\!\!\!\!\!\!
\textstyle\sum_i u \big(e^i \cdot f(e_i)\big)
\\
&
%\!\!\!\!\!\!\!\!
\!\!\!\!\!\!\!\!\!\!
{\overset{\scriptscriptstyle{
\eqref{vetrorotto3}
}}{=}}
&
\!\!\!\!\!\!
\textstyle\sum_i (e^i \cdot f(e_i))^{(0)} \langle (e^i \cdot f(e_i))^{(1)}, u \rangle
\\
&
%\!\!\!\!\!\!\!\!
\!\!\!\!\!\!\!\!\!\!
{\overset{\scriptscriptstyle{
\eqref{orvieto2}
}}{=}}
&
\!\!\!\!\!\!
\textstyle\sum_i \big( e^{i(1)+} \cdot f(e_i)^{(0)}\big) \langle e^{i(2)} f(e_i)^{(1)} e^{i(1)-}, u \rangle
\\
&
%\!\!\!\!\!\!\!\!
\!\!\!\!\!\!\!\!\!\!
{\overset{\scriptscriptstyle{
\eqref{orvieto1}, \eqref{LDMon}
}}{=}}
&
\!\!\!\!\!\!
\textstyle\sum_i \big( e^{i(1)+} \ract  \big\langle  e^{i(1)-}, \langle e^{i(2)} f(e_i)^{(1)}, u_{(1)} \rangle \lact u_{(2)} \big\rangle \big) \cdot f(e_i)^{(0)}
\\
&
%\!\!\!\!\!\!\!\!
\!\!\!\!\!\!\!\!\!\!
{\overset{\scriptscriptstyle{
\eqref{duedelue}, \eqref{Rch1}, \eqref{orvieto1}
}}{=}}
&
\!\!\!\!\!\!
\textstyle\sum_i \big( e^{i(1)+} \ract  \langle  e^{i(1)-}, u_{(2)} \rangle \big) \cdot \big( f(e_i)^{(0)}  \langle e^{i(2)} f(e_i)^{(1)}, u_{(1)} \rangle\big)
\\
&
%\!\!\!\!\!\!\!\!
\!\!\!\!\!\!\!\!\!\!
{\overset{\scriptscriptstyle{
\eqref{LDMon}
}}{=}}
&
\!\!\!\!\!\!
\textstyle\sum_i \big( e^{i(1)+} \ract  \langle  e^{i(1)-}, u_{(3)} \rangle \big) \cdot \big( f(e_i)^{(0)}  \big\langle  f(e_i)^{(1)}, \langle e^{i(2)}, u_{(1)}\rangle \lact u_{(2)} \big\rangle\big)
\\
&
%\!\!\!\!\!\!\!\!
\!\!\!\!\!\!\!\!\!\!
{\overset{\scriptscriptstyle{
\eqref{vetrorotto3}
}}{=}}
&
\!\!\!\!\!\!
\textstyle\sum_i \big( e^{i(1)+} \ract  \langle  e^{i(1)-}, u_{(3)} \rangle \big) \cdot \big(\langle e^{i(2)}, u_{(1)}\rangle \lact u_{(2)} f(e_i)\big)
\\
&
%\!\!\!\!\!\!\!\!
\!\!\!\!\!\!\!\!\!\!
{\overset{\scriptscriptstyle{
\eqref{bellitalia}, \eqref{sondreck}
}}{=}}
&
\!\!\!\!\!\!
\textstyle\sum_i \gamma \Big(\big(\langle e^{i(2)}, u_{(1)}\rangle \lact u_{(2)} f(e_i)
\big) \langle u_{(3)} \rightslice e^{i(1)}, - \rangle \Big)
\\
&
%\!\!\!\!\!\!\!\!
\!\!\!\!\!\!\!\!\!\!
{\overset{\scriptscriptstyle{
%\eqref{}
}}{=}}
&
\!\!\!\!\!\!
\textstyle\sum_i \gamma \Big(\langle e^{i(2)}, u_{(1)}\rangle \lact \big(u_{(2)} f(e_i)
 \langle u_{(3)} \rightslice e^{i(1)}, - \rangle \big)\Big)
\\
&
%\!\!\!\!\!\!\!\!
\!\!\!\!\!\!\!\!\!\!
{\overset{\scriptscriptstyle{
\eqref{umsonst}
}}{=}}
&
\!\!\!\!\!\!
\textstyle\sum_i \gamma \Big(u_{(2)} f(e_i)
\langle u_{(3)} \rightslice e^{i(1)}, \langle e^{i(2)}, u_{(1)}\rangle \blact - \rangle \Big)
\\
&
%\!\!\!\!\!\!\!\!
\!\!\!\!\!\!\!\!\!\!
{\overset{\scriptscriptstyle{
\eqref{LDMon}
}}{=}}
&
\!\!\!\!\!\!
\textstyle\sum_i \gamma \big( u_{(2)} f(e_i)
\langle u_{(3)} \rightslice e^{i}, (-)u_{(1)} \rangle \big)
\\
&
%\!\!\!\!\!\!\!\!
\!\!\!\!\!\!\!\!\!\!
{\overset{\scriptscriptstyle{
\eqref{Sch3}
}}{=}}
&
\!\!\!\!\!\!
\textstyle\sum_i \gamma \big( (u_{(2)} \rightslice f) (u_{(3)} e_i)
\langle u_{(4)} \rightslice e^{i}, (-)u_{(1)} \rangle \big)
\\
&
%\!\!\!\!\!\!\!\!
\!\!\!\!\!\!\!\!\!\!
{\overset{\scriptscriptstyle{
\eqref{gianduiotto1a}
}}{=}}
&
\!\!\!\!\!\!
\textstyle\sum_i \gamma \Big( \big(u_{(2)} \rightslice f) \big(( u_{(3)}  \ract \gve( \langle e^i, u_{(4)-} (-) u_{(1)} \rangle \blact u_{(4)+}) \big)  e_i \big) \Big) 
\\
&
%\!\!\!\!\!\!\!\!
\!\!\!\!\!\!\!\!\!\!
{\overset{\scriptscriptstyle{
\eqref{Sch4}
}}{=}}
&
\!\!\!\!\!\!
\textstyle\sum_i \gamma \Big( (u_{+(2)} \rightslice f) \big(u_{+(3)} (e_i \ract \langle e^i, u_{-} (-) u_{+(1)} \rangle)  \big) \Big) 
\\
&
%\!\!\!\!\!\!\!\!
\!\!\!\!\!\!\!\!\!\!
{\overset{\scriptscriptstyle{
\eqref{schizzaestrappa1}
}}{=}}
&
\!\!\!\!\!\!
\gamma \big( (u_{+(2)} \rightslice f) (u_{+(3)}u_{-} (-) u_{+(1)}) \big) 
\\
&
%\!\!\!\!\!\!\!\!
\!\!\!\!\!\!\!\!\!\!
{\overset{\scriptscriptstyle{
\eqref{Sch2}
}}{=}}
&
\!\!\!\!\!\!
\gamma \big( (u_{(2)} \rightslice f) ((-) u_{(1)}) \big),
\end{array}
\end{equation*}
which is \eqref{nawas2} resp.\ \eqref{nawas1}, hence $M$ is also an aYD contramodule over $U$.
As for stability, if $M$ is stable in ${}_{\scriptscriptstyle{U^*}}\mathbf{aYD}^{\scriptscriptstyle{U^*}}$, one has
$$
m = m^{(1)} \cdot m^{(0)} 
{\overset{\scriptscriptstyle{
\eqref{bellitalia}
}}{=}}
\gamma(  m^{(0)} \langle m^{(1)}, - \rangle ) 
{\overset{\scriptscriptstyle{
\eqref{vetrorotto3}
}}{=}}
\gamma((-)m), \quad \forall \ m \in M, 
$$
hence it is stable in ${}_\uhhu \mathbf{aYD}^{\scriptscriptstyle{\rm contra-}\uhhu}$ (and vice versa).
%\end{enumerate}
\end{proof}

\begin{rem}
\label{polline}
\
\begin{enumerate}
\compactlist{99}
\item
The first two parts look exactly the way it would be for bialgebras and might therefore appear somehow banal, but we do not want to deprive the reader of the fact that starting with the {\em left dual} $U_* = \Hom_\ahha(\due U \lact {}, A)$ one perhaps somewhat unexpectedly obtains a functor 
$\comodu \to  \mathbf{Mod}\mbox{-}{U_*}$, see \cite[Prop.~4.2.1]{CheGavKow:DFOLHA}.
\item
The functor 
$
\mathbf{Contramod}\mbox{-}U \to {U^*\mbox{-}\mathbf{Mod}}  
$
is not monoidal as in general ${U^*\mbox{-}\mathbf{Mod}}$ resp.\ $\mathbf{Contramod}\mbox{-}U$ is known not to be monoidal resp.\ not known to be monoidal, as mentioned before.
\end{enumerate}
\end{rem}

The idea behind the subsequent proposition is as follows: as briefly mentioned in Remark \ref{nunzia}, whereas Yetter-Drinfel'd modules form a monoidal category $\yd$, {\em anti} Yetter-Drinfel'd modules generally do not; on the other hand, they do constitute a module category over $\yd$, that is to say, the tensor product of an aYD module with a YD module yields an aYD module again, see \cite[Lem.~6.1]{Kow:GABVSOMOO} in the bialgebroid context. Here, along with the codiagonal left coaction on the tensor product, one uses the right action from \eqref{lemma3}.
If $U_\ract$ is finitely generated projective, in quite the same way tensoring a left $U^*$-module with a right $U^*$-module gives a {\em left} module again, which in view of Lemma \ref{sofocle2}, parts \eqref{uhu4} and \eqref{uhu1} amounts to tensoring a right $U$-contramodule with a left $U$-comodule to obtain a right $U$-contramodule again. Using again Lemma \ref{sofocle2}, parts \eqref{uhu3} and \eqref{uhu5}, the aYD property then translates by saying that the aYD contramodules form a model category over $\yd$ as well (where the left $U$-module structure on the tensor product is given by the diagonal left action):

\begin{prop}
\label{jetztnkaffee}
Let $(U,A)$ be a left Hopf algebroid with $U_\ract$ finitely generated $A$-projective.
\begin{enumerate}
\compactlist{99}
\item
Let $M \in \mathbf{Contramod}\mbox{-}U$ and $N \in \ucomod$. Then $M \otimes_\ahha N$ is a right $U$-contramodule again.
\item
\label{heuteabendtiberinselkino}
Let $M \in  {}_\uhhu \mathbf{aYD}^{\scriptscriptstyle{\rm contra-}\uhhu}$ and $N \in \yd$. Then $M \otimes_\ahha N$ with its diagonal left $U$-action is an aYD contramodule again.
%, which is stable if $M$ is stable.
\end{enumerate}
\end{prop}

\begin{rem}
One might be tempted to think that this is somehow also true in the non-finite case but at present we did not manage to prove this. Also, it is not clear whether the possible stability of $M$ implied stability of $M \otimes_\ahha N$.
\end{rem}

\begin{proof}[Proof of Proposition \ref{jetztnkaffee}]
Let $f \in \Hom_\Aopp(U, M \otimes_\ahha N)$. For $u \in U$, one may write this as $\sum_i f(u)'_i \otimes_\ahha f(u)''_i \in M \otimes_\ahha N$, but to reduce the quantity of sub- or superscripts in order to lighten notation, we will simply denote this as $f(u)' \otimes_\ahha f(u)''$, with summation understood.

Part (i): we claim that for $ M \in \mathbf{Contramod}\mbox{-}U$ and $N \in \ucomod$,
\begin{equation}
\label{ludovisi}
\gamma(f) := \textstyle\sum_j \gamma_\emme\big( f(e_j)'\langle f(e_j)''_{(-1)} \rightslice e^j, -\rangle \big) \otimes_\ahha f(e_j)''_{(0)}
\end{equation}
defines a right $U$-contraaction on $M \otimes_\ahha N$, where $\gamma_\emme$ is the right $U$-contraaction on $M$, and 
$\{e_j\}_{1 \leq i \leq n}  \in
U, \ \{e^j\}_{1 \leq i \leq n}  \in U^*$ a dual basis. 
To prove this, we either might directly verify the defining  Eqs.~\eqref{passionant}--\eqref{carrefour2}, or show how this contraaction can be obtained from Lemma \ref{sofocle2} described right before the proposition. To this end, consider the adjoint left action on $M \otimes_\ahha N$ over the right Hopf algebroid $U^*$ given by $(\phi, m \otimes_\ahha n) \mapsto \phi^+ m \otimes_\ahha n\phi^-$.
%, see, for example, \cite[Prop.\ 3.1.1]{CheGavKow:DFOLHA} for a dicussion of this action in the left bialgebroid context. 
Using then the isomorphism \eqref{vetrorotto2} along with \eqref{mistmistmist}, 
%and the notation mentioned above for $f \in \Hom_\Aop(U, M \otimes_\ahha N)$, 
we see that the contraaction on $M \otimes_\ahha N$ is given by
\begin{equation*}
\begin{array}{rcl}
\gamma(f) &=& \textstyle \sum_j e^j \cdot f(e_j)   
=
e^{j+} f(e_j)' \otimes_\ahha f(e_j)'' e^{j-} 
\\
&
{\overset{\scriptscriptstyle{
\eqref{lesbrigittes2}, \eqref{bellitalia}
}}{=}}
& 
\textstyle \sum_j \gamma_\emme\big(f(e_j)'\langle e^{j+} , -\rangle \big) \otimes_\ahha \langle e^{j-}, f(e_j)''_{(-1)} \rangle  f(e_j)''_{(0)} 
%\\
%&
%=
% \textstyle \sum_j \gamma_\emme\big(f(e_j)'\langle e^{j+} , -\rangle \big)\langle e^{j-}, f(e_j)''_{(-1)} \rangle \otimes_\ahha  f(e_j)''_{(0)} 
\\
&
{\overset{\scriptscriptstyle{
\eqref{passionant}
}}{=}}
&
\textstyle \sum_j \gamma_\emme\big(f(e_j)'\big\langle e^{j+} , \langle e^{j-}, f(e_j)''_{(-1)} \rangle \lact - \big\rangle \big) \otimes_\ahha  f(e_j)''_{(0)} 
\\
&
{\overset{\scriptscriptstyle{
\eqref{duedelue}, \eqref{sondreck} 
}}{=}}
&
\textstyle \sum_j \gamma_\emme\big(f(e_j)'\langle f(e_j)''_{(-1)} \rightslice e^j, - \rangle \big) \otimes_\ahha  f(e_j)''_{(0)}, 
\end{array}
\end{equation*}
which is what was claimed.

Part (ii): we need to show that \eqref{nawas1} holds for $\gamma$ from \eqref{ludovisi} under the condition that \eqref{nawas1} already holds for $\gamma_\emme$, and likewise with respect to stability. On the one hand, using the diagonal left $U$-action on $M \otimes_\ahha N$, one has for $ u \in U$:
\begin{footnotesize}
\begin{equation}
\label{quantodevopagare}
\!\!\!\!\!\!\!\!\!
\begin{array}{rcl}
u\big(\gamma(f)\big) 
&
\!\!\!\!\!\!\!
{\overset{\scriptscriptstyle{
%\eqref{ludovisi}
}}{=}}
&
\!\!\!\!\!\! 
\textstyle \sum_j 
u_{(1)} \Big( \gamma_\emme\big(f(e_j)'\langle f(e_j)''_{(-1)} \rightslice e^j, - \rangle \big)\Big) \otimes_\ahha  u_{(2)}f(e_j)''_{(0)}
\\
&
\!\!\!\!\!\!\!
{\overset{\scriptscriptstyle{
\eqref{ludovisi}
}}{=}}
&
\!\!\!\!\!\! 
\textstyle \sum_j 
\gamma_\emme\big(u_{(1)+(2)}\big(f(e_j)'\langle f(e_j)''_{(-1)} \rightslice e^j, u_{(1)-}(-) u_{(1)+(1)} \rangle \big)\big) \otimes_\ahha  u_{(2)}f(e_j)''_{(0)}.
\end{array}
\end{equation}
\end{footnotesize}
On the other hand, one has
\begin{footnotesize}
\begin{equation*}
\begin{array}{rcl}
&
\!\!\!\!\!\!\!\!\!\!\!\!\!\!\!\!\!\!
& 
\!\!\!\!\!\!
\!\!\!\!\!\!
\gamma\big(u_{+(2)} f (u_- (-) u_{+(1)} ) \big) 
\\
&
\!\!\!\!\!\!\!\!\!\!\!\!\!\!\!\!\!\!
{\overset{\scriptscriptstyle{
\eqref{ludovisi}, \eqref{Sch4}
}}{=}}
&
\!\!\!\!\!\!
\textstyle \sum_j
\gamma_\emme\Big(\big(u_{(2)+(1)} f(u_{(2)-} e_j u_{(1)})'\big)\langle \big(u_{(2)+(2)}f(u_{(2)-} e_j u_{(1)})''\big)_{(-1)} \rightslice e^j, - \rangle\Big) \\
&& \qquad \qquad \qquad
\otimes_\ahha \big(u_{(2)+(2)} f(u_{(2)-} e_j u_{(1)})'' \big)_{(0)}
\\
&
\!\!\!\!\!\!\!\!\!\!\!\!\!\!\!\!\!\!
{\overset{\scriptscriptstyle{
\eqref{schizzaestrappa1}, \eqref{trattovideo}
}}{=}}
&
\!\!\!\!\!\!
\textstyle \sum_{j,k}
\gamma_\emme\Big(\big( u_{(2)+(1)} f(e_k)'\big)
\\
&& \qquad \quad
\Big\langle \Big(\big\langle e^{k(1)}, \langle e^{k(2)}, \langle e^{k(3)} , u_{(1)} \rangle \blact e_j  \rangle \blact u_{(2)-} \big\rangle \blact u_{(2)+(2)}f(e_k)''\Big)_{(-1)} 
\\
&& \qquad \qquad \qquad
\rightslice e^j, - \Big\rangle\Big) 
\otimes_\ahha \big(u_{(2)+(2)} f(e_k)'' \big)_{(0)}
\\
&
\!\!\!\!\!\!\!\!\!\!\!\!\!\!\!\!\!\!
{\overset{\scriptscriptstyle{
\eqref{schizzaestrappa1}, \eqref{Sch1}
}}{=}}
&
\!\!\!\!\!\!
\textstyle \sum_{j,k}
\gamma_\emme\Big(\big(u_{(2)+(1)} f(e_k)'\big)
\\
&& \qquad \quad
\Big\langle \Big(\langle e^{k(1)}, u_{(2)-} \rangle \blact u_{(2)+(2)}  
 \ract \langle e^{k(2)}, e_j  \rangle 
f(e_k)''\Big)_{(-1)} 
\\
&& \qquad \qquad \qquad
\rightslice \big(e^j \bract \langle e^{k(3)} , u_{(1)} \rangle\big), - \Big\rangle\Big) 
\otimes_\ahha \big(u_{(2)+(2)} f(e_k)'' \big)_{(0)}
\\
&
\!\!\!\!\!\!\!\!\!\!\!\!\!\!\!\!\!\!
{\overset{\scriptscriptstyle{
\eqref{druento1}, \eqref{Sch9}
}}{=}}
&
\!\!\!\!\!\!
\textstyle \sum_{j,k}
\gamma_\emme\Big(\big(u_{(2)+(1)} f(e_k)'\big)
\\
&& \qquad \quad
\Big\langle \Big(
 u_{(2)+(2)+(1)} f(e_k)''_{(-1)}
\langle e^{k(1)}, u_{(2)-} \rangle \lact u_{(2)+(2)-}  
 \bract \langle e^{k(2)}, e_j  \rangle \Big) 
\\
&& \qquad \qquad \qquad
\rightslice \big(e^j \bract \langle e^{k(3)} , u_{(1)} \rangle\big), - \Big\rangle\Big) 
\otimes_\ahha u_{(2)+(2)} f(e_k)''_{(0)}
%\\
\end{array}
\end{equation*}
\end{footnotesize}
\begin{footnotesize}
\begin{equation*}
\begin{array}{rcl}
&
\!\!\!\!\!\!\!\!\!\!\!\!\!\!\!\!\!\!
{\overset{\scriptscriptstyle{
\eqref{Sch4}, \eqref{Sch9}, \eqref{gianduiotto1a}
}}{=}}
&
\!\!\!\!\!\!
\textstyle \sum_{j,k}
\gamma_\emme\Big( \big(u_{(2)+(1)} f(e_k)'\big)
\gve\Big(u_{(2)+(2)+} f(e_k)''_{(-1)+} 
\langle e^{k(1)}, u_{(2)-} \rangle \lact u_{(2)-(2)+} 
\\
&& \qquad
\bract \big(\langle e^{k(2)}, e_j  \rangle 
\big\langle e^j \bract \langle e^{k(3)} , u_{(1)} \rangle, u_{(2)-(2)-} f(e_k)''_{(-1)-} u_{(2)+(2)-} (-)\big\rangle\big)\Big)\Big) 
\\
&& \qquad \qquad \qquad
\otimes_\ahha u_{(2)+(3)} f(e_k)''_{(0)}
\\
&
\!\!\!\!\!\!\!\!\!\!\!\!\!\!\!\!\!\!
{\overset{\scriptscriptstyle{
\eqref{Sch4}, \eqref{duedelue}, \eqref{trattovideo}, \eqref{schizzaestrappa2}
}}{=}}
&
\!\!\!\!\!\!
\textstyle \sum_{k}
\gamma_\emme\Big( \big(u_{+(2)} f(e_k)'\big)
\gve\Big(u_{+(3)+} f(e_k)''_{(-1)+} 
\langle e^{k(1)}, u_{-+(1)} \rangle \lact u_{-+(2)} 
\\
&& \qquad \qquad
\bract \langle e^{k(2)}, u_{--} f(e_k)''_{(-1)-} u_{+(3)-} (-) u_{+(1)}\rangle\Big)\Big) 
%\\
%&& \qquad \qquad \qquad
\otimes_\ahha u_{+(4)} f(e_k)''_{(0)}
\\
&
\!\!\!\!\!\!\!\!\!\!\!\!\!\!\!\!\!\!
{\overset{\scriptscriptstyle{
\eqref{tellmemore}, \eqref{alsodoch}, \eqref{trattovideo}
}}{=}}
&
\!\!\!\!\!\!
\textstyle \sum_{k}
\gamma_\emme\Big( \big(u_{+(2)} f(e_k)'\big)
\gve\Big(u_{+(3)+} f(e_k)''_{(-1)+} 
\\
&& \qquad \qquad
\bract
\big\langle e^{k}, u_{-+} u_{--} f(e_k)''_{(-1)-} u_{+(3)-} (-) u_{+(1)}\big\rangle\Big)\Big) 
\otimes_\ahha u_{+(4)} f(e_k)''_{(0)}
\\
&
\!\!\!\!\!\!\!\!\!\!\!\!\!\!\!\!\!\!
{\overset{\scriptscriptstyle{
\eqref{Sch7}, \eqref{Takeuchicoaction}, \eqref{pergolesi}, \eqref{Sch8}
}}{=}}
&
\!\!\!\!\!\!
\textstyle \sum_{k}
\gamma_\emme\Big( \big(u_{(2)} f(e_k)'\big)
\gve\Big(u_{(3)+} f(e_k)''_{(-1)+} 
\\
&& \qquad \qquad
\bract
\big\langle e^{k}, f(e_k)''_{(-1)-} u_{(3)-} (-) u_{(1)}\big\rangle\Big)\Big) 
\otimes_\ahha u_{(4)} f(e_k)''_{(0)}
\\
&
\!\!\!\!\!\!\!\!\!\!\!\!\!\!\!\!\!\!
{\overset{\scriptscriptstyle{
\eqref{alsodoch}, \eqref{Sch9}, \eqref{tellmemore}, \eqref{Sch4}
}}{=}}
&
\!\!\!\!\!\!
\textstyle \sum_{k}
\gamma_\emme\Big(\Big(\Big( \gve\Big( f(e_k)''_{(-1)+} \bract
\big\langle e^{k}, f(e_k)''_{(-1)-} u_{(1)-} (-) u_{(1)+(1)}\big\rangle\Big) 
\\
&& \qquad \qquad
\blact u_{(1)+(2)}\Big) f(e_k)'\Big) \gve(u_{(1)+(3)}) \Big)
\otimes_\ahha u_{(2)} f(e_k)''_{(0)}
\\
&
\!\!\!\!\!\!\!\!\!\!\!\!\!\!\!\!\!\!
{\overset{\scriptscriptstyle{
\eqref{pergolesi}, \eqref{gianduiotto1a}
}}{=}}
&
\!\!\!\!\!\!
\textstyle \sum_k 
\gamma_\emme\big(u_{(1)+(2)}\big(f(e_k)'\langle f(e_k)''_{(-1)} \rightslice e^k, u_{(1)-}(-) u_{(1)+(1)} \rangle \big)\big) \otimes_\ahha  u_{(2)}f(e_k)''_{(0)}
\\
&
\!\!\!\!\!\!\!\!\!\!\!\!\!\!\!\!\!\!
{\overset{\scriptscriptstyle{
\eqref{quantodevopagare}
}}{=}}
&
\!\!\!\!\!\!
u\big(\gamma(f)\big), 
\end{array}
\end{equation*}
\end{footnotesize}
that is, we showed \eqref{nawas1} for the contraaction \eqref{ludovisi}.
\end{proof}

\begin{rem}
In particular, if $M = A$, that is, if the base algebra itself is an aYD contramodule over $U$ and $N \in \yd$, then 
\begin{equation}
\label{ludovisi2}
\gamma(f) := \textstyle\sum_j \gamma_\ahha\big(\langle f(e_j)_{(-1)} \rightslice e^j, -\rangle \big)f(e_j)_{(0)}.
\end{equation}
defines a right $U$-contraaction on $N$ which turns it into an aYD contramodule.
\end{rem}

\vspace*{.1cm}
\begin{center}
* \quad * \quad *
\end{center}

%\medskip

\subsection{The cocyclic module}
%\begin{rmvarthm}
We are now in a position to define a set of operators $(\gd_i, \gs_j, \tau)$ that will turn out to define a cocyclic module structure on the $k$-modules $C^\bull(U,M)$ from \eqref{baumbluetenfest}. To this end, set
\begin{equation}
\hspace*{-.2cm} \begin{array}{rll}
\label{anightinpyongyang1}
(\gd_i f)(u^1, \ldots, u^{n+1}) \!\!\!\!&= \left\{\!\!\!
\begin{array}{l} 
 f(u^1, \ldots, \gve(u^{n+1}) \blact u^{n})
\\ 
 f(u^1, \ldots, u^{n-i+1} u^{n-i+2}, \ldots, u^{n+1})
\\
u^1 f(u^2, \ldots, u^{n+1})
\end{array}\right.  
& \!\! \hspace*{-.5cm}  \begin{array}{l} \mbox{if} \ i=0, \\ \mbox{if} \
  1 \leq i \leq n, \\ \mbox{if} \ i = n + 1,  \end{array} 
\\
%
%\gd_j(m) \!\!\!\! &= \left\{ \!\!\!
%\begin{array}{l}
%		  1%_\uhhu 
%		  \otimes_\ahha m  \quad
%\\
%m_{(-1)} \otimes_\ahha m_{(0)}  \quad 
%\end{array}\right. & \!\! \hspace*{-2cm} 
%\begin{array}{l} \mbox{if} \ j=0, \\ \mbox{if} \
%  j = 1 ,  \end{array} \\
\
\\
(\gs_j f)(u^1, \ldots, u^{n-1}) \!\!\!\! 
&= f(u^1, \ldots, u^{n-j}, 1, u^{n-j+1}, \ldots, u^n) & \! \,
\hspace*{1pt} \hspace*{-.5cm}  0 \leq j \leq n-1,
\\
\
\\
(\tau f)(u^1, \ldots, u^n) \!\!\!\! 
&= \gamma\big(
u^1_+ f(u^2_+, \ldots, u^n_+, u^n_- \cdots u^1_-(-))\big) & 
\end{array}
\end{equation}
on $C^n(U,M)$. Zero cochains are identified with elements of $M$ by means of $\Hom_\Aopp(\Aop, M) \simeq M$, and the corresponding cofaces read %$(\gd_0 m)(u) = m\gve(u)$ and $(\gd_1 m)(u) = um$.
$$
(\gd_i m)(u) = \left\{ \!\!\!
\begin{array}{ll}
m\gve(u) & \mbox{if} \ i = 0, \\
um & \mbox{if} \ i = 1.
\end{array}
\right.
$$
%\end{rmvarthm}

As always, set $\gb := \sum^{n+1}_{i=0} (-1)^i \gd_i$. It is a straightforward check that $\big(C^\bull(U,M), \gd_\bull, \gs_\bull\big)$ defines a cosimplicial $k$-module, and only needs the left $U$-bialgebroid structure of $U$ along with the left $U$-module structure of $M$. We can therefore make the following definition:

\begin{lemdfn} 
The simplicial cohomology computed by $(C^\bull(U,M), \gb)$ for a left bialgebroid $(U,A)$ and a left $U$-module $M$ is denoted by $H^\bull(U,M)$ and called the {\em Hochschild cohomology} of $U$ with values in $M$; if $U_\ract$ is flat as an $A$-module,
then $H^\bull(U,M) \simeq \Ext^\bull_U(A,M)$, see \cite{KowKra:BVSOEAT}.
\end{lemdfn}

To prove that $(C^\bull(U,M), \gd_\bull, \gs_\bull, \tau)$ also defines a cocyclic $k$-module, we will pass through the dual:  
a higher degree version of \eqref{vetrorotto1} gives a map between the complexes we are interested in. More precisely, define
\begin{equation}
\label{mondrian1}
\begin{split}
\xi:  C^n_\co( & U^*, M)  \to C^n(U,M), \quad m \otimes_\ahha \phi_1 \otimes_\ahha \cdots \otimes_\ahha \phi_n \mapsto 
\\
& \big\{ u^1 \otimes_\Aopp \cdots \otimes_\Aopp u^n \mapsto m\langle \phi_1, \langle \phi_2 , \langle \ldots \langle \phi_n , u^n \rangle \blact \ldots \rangle \blact u^2 \rangle \blact u^1 \rangle \big\},
\end{split}
\end{equation}
which, as before, is an isomorphism if $U_\ract$ is finitely generated $A$-projective in which case the inverse reads
\begin{equation}
\label{mondrian2}
\begin{split}
\xi^{-1}: C^n(U,M) &\to C^n_\co(U^*,M), \\
 f &\mapsto 
%\textstyle
\sum_{i_1, \ldots, i_n} f(e_{i_1}, \ldots, e_{i_n}) \otimes_\ahha e^{i_1} \otimes_\ahha \cdots \otimes_\ahha e^{i_n}.
\end{split}
\end{equation}

\begin{prop}
\label{zitty}
Let $(U,A)$ be a left Hopf algebroid, let $U_\ract$ be finitely generated $A$-projective, and $M$ a left $U^*$-module right $U^*$-comodule with compatible left $A$-actions. Then $M$ can be seen as a right $U$-contramodule and left $U$-module with compatible left $A$-actions and the operators \eqref{anightinpyongyang1} can be obtained as
\begin{equation*}
\begin{array}{rcl}
\label{figc}
\gd_i &=& \xi \circ  \gd'_i \circ \xi^{-1}, \\
\gs_j &=& \xi \circ \gs'_j \circ \xi^{-1}, \\
\tau &=& \xi \circ \tau' \circ \xi^{-1},
\end{array}
\end{equation*}
where $\xi$ is the isomorphism from
\eqref{mondrian1} and $\delta_i', \sigma'_j, \tau'$ 
are the para-cocyclic operators on
$C^\bull_\co(U^*, M)$ for the right bialgebroid $(U^*,A)$ and the left module right comodule $M$ as in \eqref{anightinpyongyang2}.
\end{prop}

% \begin{rem}
% We shall compute the operators \eqref{anightinpyongyang1} by passing through the cyclic module \eqref{anightinpyongyang2} for $U^*$ which hinges on the translation map \eqref{viviverde1}, but it is clear that once obtained these operators perfectly make sense even if the projectivity assumption is dropped, and even give a cyclic module in full generality as we will see below.
% \end{rem}

\begin{proof}
%once shown that the operators \eqref{anightinpyongyang1} are indeed obtained by \eqref{figc} from those in \eqref{anightinpyongyang2}; 
%
We will only prove this for the most difficult case, that is, for the cyclic operator, and leave the respective computations for cofaces and codegeneracies to the reader; however, we even restrict to the case $n=2$ for reasons of space and to avoid too messy expressions (you will soon understand) as for example appear in \eqref{mondrian1}, the case for general $n$ then being obvious (no induction needed). 

So, let $f \in C^\bull(U,M)$ and assume for a minute that $U_\ract$ is finitely generated $A$-projective with dual basis $\{e_i\}_{1 \leq i \leq n}  \in
U, \ \{e^i\}_{1 \leq i \leq n}  \in U^*$, and let $\cdot$ denote the left $U^*$-action on $M$. One computes for $u, v \in U$
\begin{small}
\begin{equation}
\begin{array}{cl}
\label{nochnichtfertig}
%\begin{split}
& (\xi  \circ \tau' \circ \xi^{-1})(u,v) 
\\
{\overset{\scriptscriptstyle{
\eqref{anightinpyongyang2}, \eqref{mondrian1}, \eqref{mondrian2}
}}{=}}
& 
\textstyle\sum_{i,j} \big (e^{j+} \cdot f(e_i, e_j)^{(0)} \big) \big\langle f(e_i, e_j)^{(1)} e^{j-(1)} , \langle e^i  e^{j-(2)} , v \rangle \blact u \big\rangle
\\ 
{\overset{\scriptscriptstyle{
\eqref{LDMon}, \eqref{duedelue}
}}{=}}
& 
\textstyle\sum_{i,j}
\Big( e^{j+} \ract \Big\langle  e^{j-(1)} \ract \langle  f(e_i, e_j)^{(1)}, u_{(1)} \rangle, \langle e^i  e^{j-(2)} , v \rangle \blact u_{(2)} \Big\rangle
\Big)\cdot f(e_i, e_j)^{(0)}
\\ 
{\overset{\scriptscriptstyle{
\eqref{Rch1}, \eqref{Rch5}
}}{=}}
& 
\textstyle\sum_{i,j}
\Big(\langle  f(e_i, e_j)^{(1)}, u_{(1)} \rangle \blact  e^{j+} \ract \big\langle  e^{j-(1)} , \langle e^i  e^{j-(2)} , v \rangle \blact u_{(2)} \big\rangle
\Big)\cdot f(e_i, e_j)^{(0)}
\\
{\overset{\scriptscriptstyle{
\eqref{pergolesi}
}}{=}} 
&
\textstyle\sum_{i,j}
\big(e^{j+} \ract \big\langle  e^{j-(1)} , \langle e^i  e^{j-(2)} , v \rangle \blact u_{(2)} \big\rangle
\big)\cdot \big( f(e_i, e_j)^{(0)} \langle  f(e_i, e_j)^{(1)}, u_{(1)} \rangle \big)
\\
{\overset{\scriptscriptstyle{
\eqref{vetrorotto3}
}}{=}} 
&
\textstyle\sum_{i,j}
\big(e^{j+} \ract \big\langle  e^{j-(1)} , \langle e^i  e^{j-(2)} , v \rangle \blact u_{(2)} \big\rangle
\big)\cdot \big( u_{(1)} f(e_i, e_j)\big),
%\end{split}
\end{array}
\end{equation}
\end{small}
that is, a certain element in $U^*$ acting on an element in $M$, and we proceed by simplifying this element in $U^*$: for $u, v, w \in U$, one has
\begin{equation*}
\begin{array}{rcl}
%\begin{split}
%\textstyle\sum_{i,j}
&&
\big\langle
e^{j+} \ract \big\langle  e^{j-(1)} , \langle e^i  e^{j-(2)} , v \rangle \blact u \big\rangle, 
w
\big\rangle
\\
&
{\overset{\scriptscriptstyle{
\eqref{duedelue}
}}{=}} 
&
%\textstyle\sum_{i,j}
\big\langle
e^{j+}  , \big\langle  e^{j-(1)} , \langle e^i  e^{j-(2)} , v \rangle \blact u \big\rangle \lact w
\big\rangle 
\\
&
{\overset{\scriptscriptstyle{
\eqref{viviverde2}
}}{=}}
&
\textstyle\sum_{k}
\big\langle
e_k \rightslice e^{j}  , \big\langle  e^{k(1)} , \langle e^i  e^{k(2)} , v \rangle \blact u \big\rangle \lact w
\big\rangle 
\\
&
{\overset{\scriptscriptstyle{
\eqref{gianduiotto1a}
}}{=}}
&
\textstyle\sum_{k}
\gve\Big(  
\Big\langle   
e^j , e_{k-} \big(\langle e^{k(1)} , \langle  e^i e^{k(2)}  , v \rangle \blact u \rangle \lact w \big) 
\Big\rangle
\blact e_{k+}
\Big)
\\
&
{\overset{\scriptscriptstyle{
\eqref{LDMon}
}}{=}}
&
\textstyle\sum_{k}
\gve\Big(  
\Big\langle   
e^j , e_{k-} \big(\big\langle e^{k(1)} , \langle  e^{k(2)}  , \langle e^i, v_{(1)} \rangle \lact v_{(2)} \rangle \blact u \big\rangle \lact w \big) 
\Big\rangle
\blact e_{k+}
\Big)
\\
&
{\overset{\scriptscriptstyle{
\eqref{trattovideo}
}}{=}}
&
\textstyle\sum_{k}
\gve\Big(  
\Big\langle   
e^j , \big(e_{k-} \bract \big\langle e^{k} , u (\langle  e^i  , v_{(1)} \rangle \lact v_{(2)}) \big\rangle\big) w
\Big\rangle
\blact e_{k+}
\Big)
\\
&
{\overset{\scriptscriptstyle{
\eqref{Sch6}, \eqref{Sch9}, \eqref{schizzaestrappa1}
}}{=}}
&
\gve\Big( 
\big\langle e^j, \big( u (\langle  e^i  , v_{(1)} \rangle \lact v_{(2)}) \big)_- w \big\rangle \blact
\big(  u (\langle  e^i  , v_{(1)} \rangle \lact v_{(2)}) \big)_+ \Big)  
\\
&
{\overset{\scriptscriptstyle{
\eqref{Sch6}, \eqref{Sch9}
}}{=}}
&
\gve\Big( 
\langle e^j, v_{(2)-} u_- w \rangle \blact
\big(u_+ (\langle  e^i  , v_{(1)} \rangle \lact v_{(2)+}) \big)\Big)  
\\
&
{\overset{\scriptscriptstyle{
\eqref{gianduiotto1a}
}}{=}}
&
\Big( 
\big(u (\langle  e^i  , v_{(1)} \rangle \lact v_{(2)})  \big) \rightslice e^j
\Big)(w).
%\end{split}
\end{array}
\end{equation*}
Hence, we can resume our computation in \eqref{nochnichtfertig} and continue by
\begin{equation}
\label{schnellzubett}
%\begin{split}
\begin{array}{rcl}
&&
\hspace*{-3cm} 
(\xi  \circ \tau' \circ \xi^{-1})(u,v) 
= 
\textstyle\sum_{i,j}
\Big( 
\big(u_{(2)} (\langle  e^i  , v_{(1)} \rangle \lact v_{(2)})  \big) \rightslice e^j
\Big) \cdot  \big( u_{(1)} f(e_i, e_j)\big)
\\
&
{\overset{\scriptscriptstyle{
\eqref{tellmemore}, \eqref{pergolesi}, \eqref{schizzaestrappa1}
}}{=}}
&
\textstyle\sum_{j}
\big((u_{(2)}v_{(2)}) \rightslice e^j
\big) \cdot  \big( u_{(1)} f(v_{(1)}, e_j)\big)
\\
&
{\overset{\scriptscriptstyle{
\eqref{bellitalia}
}}{=}}
&
\textstyle\sum_{j}
\gamma \Big(\big( u_{(1)} f(v_{(1)}, e_j)\big) \langle (u_{(2)}v_{(2)}) \rightslice e^j , - \rangle \Big),
%% \\
%% &
%% {\overset{\scriptscriptstyle{
%% \eqref{Takeuchiforbialgebroids}, \eqref{pergolesi}
%% }}{=}}
%% \textstyle\sum_{i,j}
%% \big((u_{(2)}v_{(2)}) \rightslice e^j
%% \big) \cdot  \big( u_{(1)} \big(f(e_i, e_j)\langle  e^i  , v_{(1)} \rangle\big)\big)
%% \\
%% &
%% {\overset{\scriptscriptstyle{
%% \eqref{bellitalia}
%% }}{=}}
%% \textstyle\sum_{i,j}
%% \gamma \Big(\big( u_{(1)} \big(f(e_i, e_j)\langle  e^i  , v_{(1)} \rangle\big)\big) \big\langle (u_{(2)}v_{(2)}) \rightslice e^j , \cdot \rangle \Big)
\end{array}
%\end{split}
\end{equation}
where $\gamma: \Hom_\Aopp(U,M) \to M$ is the right $U$-contramodule structure on $M$ that corresponds to the left $U^*$-action as in Lemma \ref{sofocle2}, Eq.~\eqref{bellitalia}. We proceed by simplifying the element in $\Hom_\Aopp(U,M)$ given by $w \mapsto \big( u_{(1)} f(v_{(1)}, e_j)\big) \big\langle (u_{(2)}v_{(2)}) \rightslice e^j , w \rangle$. More precisely, 
\begin{eqnarray*}
%\begin{split}
&&
\textstyle\sum_j 
 \big(  u_{(1)} f(v_{(1)}, e_j)\big) \langle (u_{(2)}v_{(2)}) \rightslice e^j , w \rangle 
\\
&
 {\overset{\scriptscriptstyle{
 \eqref{Sch3}
 }}{=}}
&
\textstyle\sum_j 
\big(u_{(1)+}f(u_{(1)-} u_{(2)} v_{(1)}, e_j) \big) \langle (u_{(3)}v_{(2)}) \rightslice e^j , w \rangle 
\\
&
 {\overset{\scriptscriptstyle{
\eqref{gianduiotto1}
 }}{=}}
&
\textstyle\sum_j 
(u_{(1)} \rightslice f)(u_{(2)} v_{(1)} \ract  \langle (u_{(3)}v_{(2)}) \rightslice e^j , w \rangle, e_j) 
\\
&
 {\overset{\scriptscriptstyle{
\eqref{gianduiotto1a}
 }}{=}}
&
\textstyle\sum_j 
\big(u_{(1)} \rightslice f\big)\Big(u_{(2)} v_{(1)} \ract  
\gve\big(\langle e^j , v_{(2)-}  u_{(3)-} w \rangle \blact u_{(3)+}v_{(2)+} \big), e_j \Big)
\\
&
 {\overset{\scriptscriptstyle{
\eqref{Sch4}
 }}{=}}
&
\textstyle\sum_j 
\big(u_{+(1)} \rightslice f\big)\Big(u_{+(2)} v_{+(1)} \ract  
\gve\big(  u_{+(3)}v_{+(2)} \bract \langle e^j , v_{-}  u_{-} w \rangle \big), e_j \Big)
\\
&
 {\overset{\scriptscriptstyle{
\eqref{tellmemore}
 }}{=}}
&
(u_{+(1)} \rightslice f)\big(u_{+(2)} v_{+(1)} \ract  
\gve(u_{+(3)}v_{+(2)}) , v_{-}  u_{-} w \big)
\\
&
 {\overset{\scriptscriptstyle{
\eqref{gianduiotto1}
 }}{=}}
&
u_+ f(v_+, v_-  u_- w),
%\end{split}
\end{eqnarray*}
where in the fourth and in the last step we used the properties \eqref{alsodoch} of a left bialgebroid counit.
Putting this now back into \eqref{schnellzubett}, we obtain
$$
(\xi  \circ \tau' \circ \xi^{-1})(u,v) =
\gamma\big(u_+ f(v_+, v_-  u_- (-))\big),
$$
and this proves that for $\tau$ (and likewise $\gd_i$ and $\gs_j$) in \eqref{anightinpyongyang1} the identity $\tau = \xi  \circ \tau' \circ \xi^{-1}$ holds under the given assumptions.
\end{proof}

\begin{cor}
\label{chelabertaschen2}
If $U$ is a left Hopf algebroid and 
$M$ is a left $U$-module right $U$-contramodule with compatible left $A$-actions, 
then $C^\bull(U,M)$ 
with the operators 
$(\gd_\bull,
 \gs_\bull, \tau)$ from \eqref{anightinpyongyang1} forms a para-cocyclic 
$k$-module, which is cocyclic if $M$ is a stable aYD contramodule.
\end{cor}

\begin{proof}
The first statement follows from Lemma \ref{sofocle1}  and since the triple $(\delta'_i, \gs'_j, \tau')$ determines a para-cocyclic $k$-module, $(\delta_i, \gs_j, \tau)$ do so as well. 
However, observe at this point that now the para-cocyclic relations for the operators \eqref{anightinpyongyang1} hold in full generality, whether $U$ is finitely generated projective or not.
The second statement about cyclicity follows from Lemma \ref{sofocle2} \eqref{uhu4} and \eqref{uhu5}.
%\ref{sofocle2} (v).
\end{proof}

For later use, note that the cyclicity condition $\tau^{n+1} = \id$ on elements of degree $n$ precisely amounts to the stability \rmref{stablehalt}, that is
\begin{equation}
\label{altberlin}
(\tau^{n+1}f)(u^1, \ldots, u^n) = \gamma\big((-)f(u^1, \ldots, u^n)\big),
\end{equation} 
which is the identity if the contramodule is stable.

\subsection{Trace functors}
\label{t'n'c}

In this subsection, we will briefly address the question of how cyclic (co)homology with values in aYD contramodules is related to trace functors (in the sense of \cite[Def.~2.1]{Kal:TTAL}, see also \cite{Kal:CHWC}); 
%and centres of bimodule categories (in the sense of \cite[\S2.8]{EtiNikOst:FCAHT})
most details will be skipped and published elsewhere. 
%Note first that throughout this section the terminology {\em left} and {\em right} Hopf algebroid are always meant to be over an underlying {\em left} bialgebroid structure, {\em cf.}~Remark \ref{evenmoreconfusion}.
% 
%
The following definition is due to \cite[Def.~2.1]{Kal:TTAL}:

\begin{dfn}
\label{kaledin}
A {\em trace functor} consists of a functor $T: \cC \to \cE$ between a (unital, associative) monoidal category $(\cC, \otimes, \mathbb{1})$ and a category $\cE$, together with isomorphisms
$$
\tau_{X,Y}: T(X \otimes Y) \simeq T(Y \otimes X)
$$ 
for all $X, Y \in \cC$ that are unital (that is, $\tau_{\mathbb{1},Y} = \id$), functorial in $X$ and $Y$, as well as fulfil the property 
$$
\tau_{Z, X \otimes Y} \circ \tau_{Y, Z \otimes X} \circ \tau_{X, Y \otimes Z} = \id
$$ 
for all $X, Y, Z \in \cC$.
\end{dfn}

We illustrate the above by looking at aYD contramodules, 
%for left Hopf algebroids, 
inspired by but
slightly generalising the approach given in \cite[\S7]{KobSha:ACATCCOQHAAHA}. Let $\cC^\op$ denote the opposite category to a given category $\cC$.

\begin{theorem}
\label{ganzleerheute}
If $(U,A)$ is a left Hopf algebroid over an underlying left bialgebroid and if $M$ is a stable aYD contramodule over $U$ with contraaction $\gamma$, then $T := \Hom_\uhhu(-, M)$ yields a trace functor $(\umod)^\op \to \kmod$,
%\footnote{check whether it is really in $\Ae$-mod}
that is, we have 
$$
\tilde{\tau}: \Hom_\uhhu(P \otimes_\ahha N, M) \simeq \Hom_\uhhu(N \otimes_\ahha P, M) 
$$
for any  $N, P \in \umod$, given by
\begin{equation}
\label{hunga}
(\tilde\tau f)(n \otimes_\ahha p) := \gamma\big(f(p \otimes_\ahha (\cdot)n)\big),
\end{equation}
for $n \in N, p \in P$.
\end{theorem}

In order to prove this theorem, %as in \cite[\S7]{KobSha:ACATCCOQHAAHA}, 
we need the following lemma:

\begin{lem}
\label{keineDokumente}
\
\begin{enumerate}
\item
For every left bialgebroid $(U,A)$, the category $\umod$ is left closed monoidal, that is, has left internal Hom functors.
\item
If the left bialgebroid $(U,A)$ is a left Hopf algebroid, the category $\umod$ is right closed monoidal, that is, has right internal Hom functors.
\item
Consequently, for a left Hopf algebroid $(U,A)$ over an underlying left bialgebroid, the category $\umod$ is biclosed monoidal, that is, has both left and right internal Hom functors.
\end{enumerate}
\end{lem}

\begin{proof}
The first part is, 
in a standard way,
%inspired by \cite[Prop.~3.3]{Schau:DADOQGHA} and 
seen as follows: for $M, N, P \in \umod$ over a left bialgebroid $(U,A)$, the customary isomorphism 
$
N \otimes_\ahha P \to (N \otimes_\ahha U) \otimes_\uhhu P, \ n \otimes_\ahha p \mapsto (n \otimes_\ahha 1) \otimes_\uhhu  p, 
$
with inverse $(n \otimes_\ahha u) \otimes_\uhhu p \mapsto n \otimes_\ahha up$, 
induces an adjunction
\begin{equation}
\label{ad1}
\xi: \Hom_\uhhu(N \otimes_\ahha P, M) \to \Hom_\uhhu(P, \hom^\ell(N, M))
\end{equation}
by 
$
[(\xi f)(p)](n \otimes_\ahha u) := f(n \otimes_\ahha up)
%f \mapsto \{p \mapsto \{n \otimes_\ahha u \mapsto \{f(n \otimes_\ahha up)\}\}\}
$
and inverse $(\xi^{-1} g)(n \otimes_\ahha p) := [g(p)](n \otimes_\ahha 1_\uhhu)$,
where we set  $\hom^\ell(N, M) := \Hom_\uhhu(N \otimes_\ahha U, M)$ equipped with the left $U$-action given by right multiplication on $U$ in the argument.

As for the second part, 
let again $M, N, P \in \umod$.
The right internal Homs are defined by $\hom^r(M, N) := \Hom_\Aopp(M, N)$ equipped with the left $U$-action \rmref{gianduiotto1}, along with the adjunction morphism 
$
\zeta: \Hom_\uhhu(P \otimes_\ahha N, M)
\to \Hom_\uhhu(P, \hom^r(N, M))
$
given by 
$(\zeta f)(p) := f(p \otimes_\ahha -)$ for $p \in P$. 
%
%$f \mapsto \{p \mapsto f(p \otimes_\ahha -) \}$
To see that $\zeta f$ indeed lands in $\Hom_\uhhu(P, \hom^r(N, M))$, we check the $U$-linearity: 
\begin{equation*}
\begin{split}
(\zeta f)(up) &= f(up \otimes_\ahha -) = f(u_{+(1)}p \otimes_\ahha u_{+(2)}u_-(-))
= u_+f(p \otimes_\ahha u_-(-)) 
\\
&= u \rightslice ((\zeta f)(p)),
\end{split}
\end{equation*}
using \rmref{Sch2} and \rmref{gianduiotto1} and the diagonal action on $P \otimes_\ahha N$. In the other direction, define 
$
\eta: \Hom_\uhhu(P, \hom^r(N, M)) \to \Hom_\uhhu(P \otimes_\ahha N, M)
$
by $(\eta g)(p \otimes_\ahha n) := [g(p)](n)$, and that $\eta g$ indeed lands in $\Hom_\uhhu(P \otimes_\ahha N, M)$ follows from 
\begin{equation*}
\begin{split}
(\eta g)(u_{(1)}p \otimes_\ahha u_{(2)}n) &= [g(u_{(1)}p)](u_{(2)}n) 
= [u_{(1)} \rightslice (g(p))](u_{(2)}n) 
\\
&= u_{(1)+} \big([g(p)](u_{(1)-} u_{(2)}n)\big)
=u \big([g(p)](n)\big),
\end{split}
\end{equation*}
for $p \in P$, $n \in N$, using \rmref{Sch3} and \rmref{gianduiotto1} again. Finally, that $\zeta$ and $\eta$ are indeed mutual inverses follows as in the classical Hom-tensor adjunction, and we obtain 
\begin{equation}
\label{ad2}
 \Hom_\uhhu(P \otimes_\ahha N, M)
\simeq \Hom_\uhhu(P, \hom^r(N, M)).
\end{equation}
The third part is an obvious consequence of the first two statements.
\end{proof}

\begin{proof}[Proof of Theorem \ref{ganzleerheute}]
Comparing the two adjunctions \rmref{ad1} and \rmref{ad2}
from Lemma \ref{keineDokumente} if $(U,A)$ is a left Hopf algebroid, one sees that it is enough to find a $U$-module isomorphism from $\hom^r(N, M)$ to $\hom^\ell(N, M)$ to prove the statement. 
Note first that for $N, M \in \umod$ and $f \in \hom^r(N,M)$, one obviously has $f((-)n) \in \Hom_\Aopp(U,M)$. 
Hence, if $M$ is also a right $U$-contramodule, it makes sense to define
$$
\tau_\enne: \hom^r(N, M) \to \hom^\ell(N, M), \quad f \mapsto \{n \otimes_\ahha u \mapsto \gamma\big((u \rightslice f)((\cdot)n)\big)\},
$$
that is, $(\tau_\enne f)(n \otimes_\ahha u) = \gamma\big((u \rightslice f)((\cdot)n)\big)$. 
That indeed $\tau_\enne f$ lands in $\hom^\ell(N,M)$ follows if $M$ fulfils the aYD condition \rmref{nawas2}: one has  $v \big((\tau_\enne f)(n \otimes_\ahha u)\big) = \gamma\big(((v_{(2)}u) \rightslice f)((\cdot)v_{(1)}n)\big) = (\tau_\enne f)(v_{(1)} n \otimes_\ahha v_{(2)} u) =  (\tau_\enne f)(v (n \otimes_\ahha u))$ for all $v \in U$, and that this map is a morphism of left $U$-modules is also straightforward: for $v \in V$, we have 
$$
(v(\tau_\enne f))(n \otimes_\ahha u) = (\tau_\enne f)(n \otimes_\ahha uv) = 
 \gamma\big(((uv) \rightslice f)((\cdot)n)\big) = (\tau_\enne (v \rightslice f))(n \otimes_\ahha u),
$$
where we denoted the left $U$-action on $ \hom^\ell(N, M)$ just by juxtaposition.
As for its inverse, set
$$
\tau^\enne: \hom^\ell(N,M) \to \hom^r(N,M), \quad g \mapsto \{n \mapsto \gamma(g(n \otimes_\ahha -))\}, 
$$
that is, $(\tau^\enne g)(n) = \gamma(g(n \otimes_\ahha -))$, and by 
\begin{eqnarray*}
%\begin{split}
(u \rightslice \tau^\enne g)(n) &=& u_+ (\tau^\enne g)(u_- n)
\\
&=&
u_+ \gamma(g(u_-n \otimes_\ahha -)) 
\\
&
\overset{\scriptscriptstyle{\rmref{nawas1}}}{=} 
&
\gamma\big(u_{++(2)} g(u_-n \otimes_\ahha u_{+-} (-) u_{++(1)})\big) 
\\
&
\overset{\scriptscriptstyle{\rmref{Sch5}}}{=} 
&
\gamma\big(u_{+(2)} g(u_{-(1)} n \otimes_\ahha u_{-(2)} (-) u_{+(1)})\big) 
\\
&
=
&
%\overset{\scriptsciptstyle{\rmref{Sch5}}}{=} 
\gamma\big(u_{+(2)}u_- g(n \otimes_\ahha (-) u_{+(1)})\big) 
\\
&
\overset{\scriptscriptstyle{\rmref{Sch2}}}{=} 
&
\gamma\big(g(n \otimes_\ahha (-) u)\big)
%\end{split}
\end{eqnarray*}
for all $u \in U$, %using \rmref{nawas1} and \rmref{Sch5}, 
we see that this is a map of $U$-modules as well.
We then compute
\begin{equation*}
\begin{split}
(\tau^\enne(\tau_\enne f))(n) &= \gamma\big((\tau_\enne f)(n \otimes_\ahha -) \big)
\\
&
= \dot\gamma\big(\ddot\gamma\big((\cdot)_{+} f((\cdot)_{-}(\cdot\cdot)n)\big)\big)
\\
&
= \dot\gamma\big((\cdot)_{(1)+} f((\cdot)_{(1)-}(\cdot)_{(2)}n)\big)
\\
&
= \gamma((-)f(n)), 
\end{split}
\end{equation*}
where we used the contraassociativity \rmref{carrefour1} in the third step and 
\rmref{Sch3} in the fourth. Hence, if the aYD contramodule $M$ is stable, then $\gamma((-)f(n)) = f(n)$, and therefore $\tau^\enne \circ \tau_\enne = \id$; likewise, one proves  $\tau_\enne \circ \tau^\enne = \id$. These maps are therefore mutual inverses and $\tau_\enne$ is a $U$-module isomorphism if $M$ is stable. 

In total, we get a commutative diagram
$$
 \xymatrix{
\Hom_\uhhu(P \otimes_\ahha N, M)  \ar[rr]^-{\eta}\ar@{.>}[d]%_{\tilde\tau} 
&& \Hom_\uhhu(P, \hom^r(N, M)) \ar[d]^{\Hom_\uhhu(P,\tau_\enne)} \\
 \Hom_\uhhu(N \otimes_\ahha P, M) && \Hom_\uhhu(P, \hom^\ell(N, M)), \ar[ll]_{\xi^{-1}}
}
$$
and we need to show that $\tilde\tau$ given in \rmref{hunga} fits into this diagram at the dotted arrow, that is, that $\tilde\tau = \xi^{-1} \circ \Hom_\uhhu(P,\tau_\enne) \circ \eta$. More precisely, for $f \in \Hom_\uhhu(P \otimes_\ahha N, M)$, we have
\begin{equation*}
\begin{split}
(\xi^{-1} \circ \Hom_\uhhu(P,\tau_\enne) \circ \eta f)(n \otimes_\ahha p)
&= [(\Hom_\uhhu(P,\tau_\enne) \circ \eta f)(p)](n \otimes_\ahha 1_\uhhu) \\
&= \gamma\big( [(\eta f)(p)]((-)n)\big) \\
&= \gamma\big( f(p \otimes_\ahha (-)n)\big) \\
&= (\tilde\tau f)(n \otimes_\ahha p). 
\end{split}
\end{equation*}
The proof of the second displayed equation in the Definition \ref{kaledin} of a trace functor is left to the reader.
This concludes the proof of the theorem.
\end{proof}

We proceed by showing how to obtain from Theorem \ref{ganzleerheute}, that is, from a trace functor the cyclic operator $\tau$ in \rmref{anightinpyongyang1} on the complex $C^\bull(U,M)$. 

Computing the homology of  $C^\bull(U,M)$ could be achieved by considering the homology of $\Hom_\uhhu({\rm Bar}_\bull(U), M)$, where ${\rm Bar}_n(U) := (\due U \blact \ract)^{\otimes_\Aopp n+1}$ is the bar resolution of $A$, which is a $U$-module by multiplication on the first factor. On the other hand, ${\rm Bar}_\bull(U)$ is {\em not} a monoidal product of two $U$-modules in $\umod$, which is what we need in Theorem \ref{ganzleerheute}; but thanks to the Hopf-Galois map (or rather its ``higher'' version \cite[Lem.~4.10]{KowKra:CSIACT}), there is a $U$-module isomorphism 
$$
\alpha_\ell: {\rm Bar}_n(U) = (\due U \blact \ract)^{\otimes_\Aopp n+1} \overset{\simeq}{\lra} (\due U \lact \ract)^{\otimes_\ahha n+1}.
$$
Observe then that for any $N,M \in \umod$, one has an isomorphism of $k$-modules
$$
\xi: \Hom_\Aopp(N,M) \to \Hom_\uhhu(U \otimes_\ahha N, M), \quad f \mapsto (\cdot) \rightslice f,
$$
with inverse $\xi^{-1}: g \mapsto g(1 \otimes_\ahha -)$. This follows from 
$$
g(u \otimes_\ahha n) \overset{\scriptscriptstyle{\rmref{Sch2}}}{=} g(u_{+(1)} \otimes_\ahha u_{+(2)} u_- n) = u_+ g(1 \otimes_\ahha u_- n)
$$
for $g \in  \Hom_\uhhu(U \otimes_\ahha N, M)$. 
 We then have for $N = (\due U \lact \ract)^{\otimes_\ahha n}$ the isomorphisms
 $$
 \Hom_\uhhu(U \otimes_\ahha U^{\otimes_\ahha n}, M) \simeq \Hom_\Aopp(U^{\otimes_\ahha n}, M)
 \simeq \Hom_\Aopp(U^{\otimes_\Aopp n}, M) = C^n(U,M),
$$
 and can apply Theorem \ref{ganzleerheute} on the left hand side to obtain (or rather reproduce) the cyclic operator on the right hand side from Theorems \ref{B} \& \ref{B1}; %resp.\ Eq.~\rmref{anightinpyongyang1}; 
more precisely, define for $f \in C^\bull(U,M)$
$$
  \tau f :=  \big(\Hom_\Aopp(\ga_\ell, M) \circ \xi^{-1} \circ \tilde\tau \circ \xi \circ  \Hom_\Aopp(\ga^{-1}_\ell, M)\big)(f), 
 $$
 which will be shown to coincide with the cyclic operator in \rmref{anightinpyongyang1}. For the sake of readability, we will do this in degree one only (since the maps  $\Hom_\Aopp(\ga_\ell, M)$ are trivial in this case):
\begin{equation*}
\begin{split}
(\tau f)(u) &
=   
\Big(\big(\Hom_\Aopp(\ga_\ell, M) \circ \xi^{-1} \circ \tilde\tau \circ \xi \circ  \Hom_\Aopp(\ga^{-1}_\ell, M)\big)(f)\Big)(u)
\\
&
=   
\Big(\big(\tilde\tau \circ \xi \circ  \Hom_\Aopp(\ga^{-1}_\ell, M)\big)(f)\Big)(1 \otimes_\ahha u)
\\
&
=   
\gamma\Big(\big(\big(\xi \circ  \Hom_\Aopp(\ga^{-1}_\ell, M)\big)(f)\big)(u \otimes_\ahha (-))\Big)
\\
&
=   
\gamma\big((u \rightslice f)(-)\big) = \gamma\big((u_+ f)(u_-(-))\big),
\end{split}
 \end{equation*}
which is $\tau$ in \rmref{anightinpyongyang1} for $n=1$. For higher degrees the computation is similar, but much more technical due to the non-trivial maps $\Hom_\Aopp(\ga_\ell, M)$.

\begin{rem}
\label{tatatatataaa}
%We conclude this section by some remarks hinting at more or less open problems that will be investigated in a future project.
In \cite[\S6]{KobSha:ACATCCOQHAAHA}, the biclosedness of $\umod$ is proven if $U$ is assumed to be a full Hopf algebroid with invertible antipode; Lemma \ref{keineDokumente} is slightly more general inasmuch as there are left Hopf algebroids which are not full Hopf algebroids (in particular, do not carry neither a right bialgebroid structure nor an antipode): an example is given by the universal enveloping algebra of a Lie-Rinehart algebra \cite{KowPos:TCTOHA}. Still in \cite{KobSha:ACATCCOQHAAHA}, it is shown that the category of aYD contramodules over a full Hopf algebroid is equivalent to the (weak) center (in the sense of \cite[\S2.8]{EtiNikOst:FCAHT})
of $(\umod)^\op$ as a bimodule category over $\umod$. We guess that the proof carries over when relaxing to left Hopf algebroids, but leave this to a future project.
\end{rem}

\begin{center}
* \quad * \quad *
\end{center}

Higher structures on the cohomology groups $H^\bull(U,M)$ resp.\ $\Ext^\bull_U(A,M)$ will be the main objects of study in the next section.

\section{Operadic structures and $\Ext$ as a Batalin-Vilkovisky algebra}

In \cite{Kow:BVASOCAPB}, we showed that $H^\bull(U,M)$ resp.\ $H^\bull_\co(U,M)$ (that is, $\Ext^\bull_U(A,M)$ resp.\ $\Cotor^\bull_U(A,M)$ if $U_\ract$ is projective resp.\ flat over $A$) are Gerstenhaber algebras if $M$ is a braided commutative YD algebra over a left bialgebroid $U$. On top, if $U$ were a left Hopf algebroid, then $H^\bull_\co(U,M)$ even is a Batalin-Vilkovisky algebra.

In this section, we want to add 
in which cases  $\Ext^\bull_U(A,M)$ becomes a Batalin-Vilkovisky algebra as well:

\begin{theorem}
\label{corsoumbertoI}
\
\begin{enumerate}
\compactlist{99}
\item
Let $(U,A)$ be a left Hopf algebroid and let $A$ be a stable aYD contramodule over $U$. Then $C(U,A)$ becomes a cyclic operad
with multiplication.
\item
Let $(U,A)$ be a left Hopf algebroid with $U_\ract$ finitely generated $A$-projective, and let 
$M$ be a braided commutative YD algebra. If $A$ is a stable aYD contramodule, $C(U,M)$ becomes a cyclic operad with multiplication.
\end{enumerate}
\end{theorem}

\begin{rem}
We believe that the second part is true even if $U_\ract $ is {\em not} finitely generated $A$-projective (the first part shows that there are instances of this situation), but for the moment we were not able to prove this.
\end{rem}

\begin{proof}[Proof of Theorem \ref{corsoumbertoI}]

Part (i): recall that
the operadic composition on $C(U,A)$ is given by
\begin{footnotesize}
\begin{equation}
\label{maxdudler1}
\begin{split}
        & (\varphi \circ_i \psi)(u^1, \ldots, u^{p+q-1}) 
\\ &
\! := \varphi(u^1, \ldots, u^{p-i}, 
\psi(u^{p-i+1}_{(1)}, \ldots, u^{p+q-i}_{(1)}) \lact u^{p-i+1}_{(2)} \cdots u^{p+q-i}_{(2)},        
u^{p+q-i+1}, \ldots, u^{p+q-1}), 
\end{split}
\end{equation}
\end{footnotesize}
see \cite[Eq.~(6.15)]{Kow:GABVSOMOO} (or \cite{KowKra:BVSOEAT} for the opposite composition) for further details. We want to show that this composition together with the cyclic operator $\tau$ from \eqref{anightinpyongyang1} fulfils the criteria in Definition \ref{danton} (ii).
With the multiplication element on $C(U,A)$ given by $\mu = \gve m_\uhhu$, where $m_\uhhu$ is the product in $U$, the property $\tau\mu = \mu$ is a straightforward check. We furthermore need to check Eqs.~\eqref{superfluorescent1}, but we shall limit ourselves to only show the first one. For $\phi \in C^\bull(U,A)$ and $\psi \in C^q(U,A)$, one has
\begin{footnotesize}
\begin{equation*}
\begin{split}
        & \tau(\phi \circ_1 \psi)(u^1, \ldots, u^{p+q-1}) 
\\
&
 {\overset{\scriptscriptstyle{
\eqref{anightinpyongyang1}
 }}{=}}
\gamma\big(u^1_+ (\phi \circ_1 \psi)(u^2_+, \ldots, u^{p+q-1}, u^{p+q-1}_- \cdots u^{1}_- (-))\big) 
\\
&
 {\overset{\scriptscriptstyle{
\eqref{maxdudler1}
 }}{=}}
\gamma\Big(u^1_+ \phi\Big(u^2_+, \ldots, u^p_+, \psi(u^{p+1}_{+(1)}, \ldots, u^{p+q-1}_{+(1)}, u^{p+q-1}_{-(1)} \cdots u^1_{-(1)} (-)_{(1)}) \lact \\
& \qquad \qquad \qquad \qquad \qquad \qquad \qquad \qquad 
u^{p+1}_{+(2)} \cdots u^{p+q-1}_{+(2)}
u^{p+q-1}_{-(2)} \cdots u^1_{-(2)} (-)_{(2)} \Big)\Big) 
\\
&
 {\overset{\scriptscriptstyle{
\eqref{Sch5}
 }}{=}}
\gamma\Big(u^1_{++} \phi\Big(u^2_{++}, \ldots, u^p_{++}, \psi(u^{p+1}_{++(1)}, \ldots, u^{p+q-1}_{++(1)}, u^{p+q-1}_{-} \cdots u^1_{-} (-)_{(1)}) \lact \\
& \qquad \qquad \qquad \qquad \qquad \qquad \qquad \qquad 
u^{p+1}_{++(2)} \cdots u^{p+q-1}_{++(2)}
u^{p+q-1}_{+-} \cdots u^1_{+-} (-)_{(2)} \Big)\Big) 
\\
&
 {\overset{\scriptscriptstyle{
\eqref{Sch2}
 }}{=}}
\gamma\Big(u^1_{++} \phi\Big(u^2_{++}, \ldots, u^p_{++}, \psi(u^{p+1}_{+}, \ldots, u^{p+q-1}_{+}, u^{p+q-1}_{-} \cdots u^1_{-} (-)_{(1)}) \lact u^{p}_{+-} \cdots u^1_{+-} (-)_{(2)} \Big)\Big), 
\end{split}
\end{equation*}
\end{footnotesize}
whereas, on the other hand,
\begin{footnotesize}
\begin{equation*}
\begin{split}
        & (\tau\psi \circ_q \tau\psi)(u^1, \ldots, u^{p+q-1}) = \tau\psi\big(\tau\phi(u^1_{(1)}, \ldots, u^p_{(1)}) \lact u^1_{(2)} \cdots u^p_{(2)}, u^{p+1}, \ldots, u^{p+q-1} \big)
\\
&
 {\overset{\scriptscriptstyle{
\eqref{anightinpyongyang1}, \eqref{Sch9}
 }}{=}}
\dot\gamma\Big(\ddot\gamma\big( u^1_{(1)+} \phi(u^2_{(1)+}, \ldots, u^p_{(1)+}, u^p_{(1)-} \cdots  u^1_{(1)-} (\cdot\cdot))\big) 
\\
&
\qquad\qquad\qquad
\gve\big( u^1_{(2)+} \cdots u^p_{(2)+} \bract \psi\big(u^{p+1}_+, \ldots, u^{p+q-1}_+,  u^{p+q-1}_- \cdots  u^{p+1}_-  u^p_{(2)-} \cdots  u^1_{(2)-} (\cdot)\big)\big)\Big)
\\
&
 {\overset{\scriptscriptstyle{
\eqref{Sch9}, \eqref{tellmemore}
 }}{=}}
\dot\gamma\Big(\ddot\gamma\Big( u^1_{(1)+} \phi\big(u^2_{(1)+}, \ldots, u^p_{(1)+}, \psi\big(u^{p+1}_+, \ldots, u^{p+q-1}_+,  u^{p+q-1}_- \cdots  u^{p+1}_-  u^p_{(2)-} \cdots  u^1_{(2)-} (\cdot)\big) \lact 
\\
&
\qquad\qquad\qquad\qquad
u^p_{(1)-} \cdots  u^1_{(1)-} (\cdot\cdot)\big)\Big) 
\gve( u^1_{(2)+} \cdots u^p_{(2)+})\Big)
\\
&
 {\overset{\scriptscriptstyle{
\eqref{passionant}
 }}{=}}
\dot\gamma\Big(\ddot\gamma\Big( u^1_{(1)+} \phi\big(u^2_{(1)+}, \ldots, u^p_{(1)+}, \psi\big(u^{p+1}_+, \ldots, u^{p+q-1}_+,  u^{p+q-1}_- \cdots  u^{p+1}_-  u^p_{(2)-} \cdots  u^1_{(2)-} (\cdot)\big) \lact 
\\
&
\qquad\qquad\qquad\qquad
u^p_{(1)-} \cdots  u^1_{(1)-} \big(\gve( u^1_{(2)+} \cdots u^p_{(2)+}) \lact \cdot\cdot\big)\big)\Big) 
\Big)
\\
&
 {\overset{\scriptscriptstyle{
\eqref{Sch4}, \eqref{Sch9}
 }}{=}}
\dot\gamma\Big(\ddot\gamma\Big(u^1_{++} \phi\big(u^2_{++}, \ldots, u^p_{++}, \psi\big(u^{p+1}_+, \ldots, u^{p+q-1}_+,  u^{p+q-1}_- \cdots  u^1_{-} (\cdot)\big) \lact 
u^p_{+-} \cdots  u^1_{+-} (\cdot\cdot)\big)\Big)\Big), 
\end{split}
\end{equation*}
\end{footnotesize}
where in the second step we additionally used the canonical left $U$-action on $A$ given by $ua := \gve(u \bract a)$ along with the left $A$-linearity of $\gve$. Now the two last lines in the two respective computations above are equal by \eqref{carrefour1}, and hence the first of Eqs.~\eqref{superfluorescent1} defining a cyclic operad is proven. The remaining ones are left to the reader. Observe, however, that the (stable) aYD contramodule condition from Definition \ref{chelabertaschen1} is only needed for the property $\tau^{n+1} = \id$ in \eqref{superfluorescent1}, which has already been obtained by Corollary \ref{chelabertaschen2}.

Part (ii): 
in \cite[\S3.1]{Kow:BVASOCAPB}, we showed that for a braided commutative YD algebra $M$ over any bialgebroid $U$ (finitely generated or not), the map 
$$
\circ_i : C^p(U,M) \otimes C^q(U,M) \rightarrow 
        C^{p+q-1}(U,M), \qquad i = 1, \ldots p,
$$
%for $i = 1, \ldots p$ 
given by 
%\begin{footnotesize}
\begin{equation}
\label{maxdudler1a}
\begin{split}
        &\ (f \circ_i g)(u^1, \ldots, u^{p+q-1}) 
\\ &\ 
:= f(u^1_{(1)}, \ldots, u^{i-1}_{(1)}, 
         g(u^{i}_{(1)}, \ldots, u^{i+q-1}_{(1)})_{(-1)} u^{i}_{(2)} \cdots u^{i+q-1}_{(2)}, 
        u^{i+q}, \ldots, u^{p+q-1}) \\
&\hspace*{4cm} \cdot_\emme \big(u^1_{(2)} \cdots u^{i-1}_{(2)} g(u^{i}_{(1)}, \ldots, u^{i+q-1}_{(1)})_{(0)}\big), 
\end{split}
\end{equation}
%\end{footnotesize}
yields, together with 
$
        \mu := (\varepsilon \, m_\uhhu(\cdot,\cdot)) \lact 1_\emme,
$
where $m_\uhhu$ is the multiplication map of $U$, and
$
\mathbb{1} := \gve(\cdot) \lact 1_\emme
$ 
as well as 
$
e := 1_\emme
$ 
an operad with multiplication. 

On the other hand, a straightforward right bialgebroid adaptation of the formulae given in \S3.2 of {\it op.~cit.}~yields that if $U_\ract$ is finitely generated $A$-projective and $M$ again a braided commutative YD algebra (over $U^*$ this time), even $C_\co(U^*,M)$ by means of the composition
$$
\circ_i^\co : C^p_\co(U^*,M) \otimes C^q_\co(U^*,M) \rightarrow 
        C^{p+q-1}_\co(U^*,M), \qquad i = 1, \ldots p,
$$
%for $i = 1, \ldots p$
given by
\begin{equation}
\label{maxdudler2}
\begin{split}
        &\ (m, \phi_1, \ldots, \phi_n) \circ_i^\co (n, \psi_1, \ldots, \psi_n ) 
\\ &\ 
:= 
\Big(m \cdot_\emme ( n \phi_i^{(1)})^{(0)}, \phi_1  ( n \phi_i^{(1)} )^{(1)}, 
\ldots,
\phi_{i-1}  ( n \phi_i^{(1)} )^{(i-1)},
\psi_1 \phi_i^{(2)}, 
\\
& \qquad\qquad\qquad
\ldots, \psi_q \phi_i^{(q+1)}, \phi_{i+1}, 
\ldots, \phi_p\Big),
\end{split}
\end{equation}
where we abbreviate $(m, \phi_1, \ldots, \phi_n) := m \otimes_\ahha \phi_1 \otimes_\ahha \cdots \otimes_\ahha \phi_n$,
along with the multiplication element
$
\mu := (1_{U^*}, 1_{U^*}, 1_\emme)
$,
$
\mathbb{1} := (1_{U^*}, 1_\emme),
$
and
$
e:=1_\emme
$
becomes an operad with multiplication. On top, in Theorem 3.10 in {\it op.~cit.}, this operad was shown to be cyclic with respect to the cyclic operator $\tau'$ from \eqref{anightinpyongyang2}:
 for this, one requires an aYD module structure on the base algebra $A$ over the right bialgebroid $U^*$ to produce such a structure on the braided commutative YD algebra $A \otimes_\Aopp M \simeq M$ (see the comments right before Proposition \ref{jetztnkaffee} for a short explanation of this construction), and if $M$ is stable, $C_\co(U^*,M)$ is a cyclic operad with multiplication, indeed.

To now obtain from this that $C(U,M)$ is a cyclic operad with multiplication as well, one simply dualises the aforementioned construction: to start with, from Lemma \ref{sofocle2}, part \eqref{uhu3} one obtains that a braided commutative YD algebra over the right bialgebroid $U^*$ corresponds to a braided commutative 
YD algebra over the left bialgebroid $U$; the aYD contramodule structure on $M$ over $U$ was obtained by Proposition \ref{jetztnkaffee}, part \eqref{heuteabendtiberinselkino}, along with Lemma \ref{sofocle2}, part \eqref{uhu5}, and is hence a construction precisely dual to the one that gives $M$ the structure of an aYD module over $U^*$; between the cyclic operators $\tau'$ on $C^\bull_\co(U^*,M)$ and $\tau$ on $C^\bull(U,M)$ we established in Proposition \ref{zitty} the relation 
$
\tau  = \xi \circ \tau' \circ \xi^{-1},
$
with respect to the isomorphism $\xi$ from \eqref{mondrian1};
hence, the only thing that is left to show is
$$
f \circ_j g = \xi\big(\xi^{-1}(f) \circ^\co_j \xi^{-1}(g)\big), \qquad \mbox{for} \ f \in C^p(U,M), g \in C^q(U,M),
$$
with respect to the composition operations \eqref{maxdudler1a} and \eqref{maxdudler2} above; which is a messy, but straightforward check, and fills a page or two. As an illustration, we will compute the case for $p=q=j=2$ (in the opposite direction), which contains all important steps; it is then immediately clear how to transfer this to the general case. So, 
let $m \otimes_\ahha \phi_1 \otimes_\ahha \phi_2 \in C^2_\co(U^*,M)$ and  $n \otimes_\ahha \psi_1 \otimes_\ahha \psi_2 \in C^2_\co(U^*,M)$. 
One has
%\pagebreak
\begin{footnotesize}
\begin{equation*}
\begin{array}{rcl}
&& \xi^{-1}\big(\xi(m \otimes_\ahha \phi_1 \otimes_\ahha \phi_2) \circ_2 \xi(n \otimes_\ahha \psi_1 \otimes_\ahha \psi_2)\big) 
\\
&
\!\!\!\!\!\!\!\!\!\!\!\!\!\!\!\!\!\!
{\overset{\scriptscriptstyle{
\eqref{mondrian1}, \eqref{mondrian2}, \eqref{maxdudler1a}
}}{=}} 
&
\!\!\!\!\!\!
\displaystyle \sum_{i_1, i_2, i_3} m \Big\langle \phi_1, \big\langle \phi_2, n_{(-1)} \langle \psi_1 , \langle\psi_2, {e_{i_3}}_{(1)} \  \rangle \blact {e_{i_2}}_{(1)} \rangle \lact  {e_{i_2}}_{(2)}  {e_{i_3}}_{(2)}  
\big\rangle \blact {e_{i_1}}_{(1)} \Big\rangle 
\\
&& \qquad \qquad \qquad
\cdot_\emme ({e_{i_1}}_{(2)}n_{(0)}) \otimes_\ahha e^{i_1} \otimes_\ahha e^{i_2} \otimes_\ahha e^{i_3} 
\\
&
\!\!\!\!\!\!\!\!\!\!\!\!\!\!\!\!\!\!
{\overset{\scriptscriptstyle{
\eqref{trattovideo}, \eqref{tellmemore}, \eqref{duedelue}
}}{=}} 
&
\!\!\!\!\!\!
\displaystyle \sum_{i_1, i_2, i_3} m \Big\langle \phi_1, \big\langle \phi_2^{(1)}, 
\big\langle \phi_2^{(2)}, 
\langle \phi_2^{(3)} ,  \langle \psi_2 , {e_{i_3}}_{(1)} \rangle \lact  {e_{i_3}}_{(2)} \rangle \blact
\langle \psi_1,  {e_{i_2}}_{(1)} \rangle \lact
 {e_{i_2}}_{(2)}
\big\rangle \blact 
n_{(-1)} \big\rangle 
\\
&& \qquad \qquad \qquad
\blact  {e_{i_1}}_{(1)} \Big\rangle
\cdot_\emme ({e_{i_1}}_{(2)}n_{(0)}) \otimes_\ahha e^{i_1} \otimes_\ahha e^{i_2} \otimes_\ahha e^{i_3} 
%\\
\end{array}
\end{equation*}
\end{footnotesize}
\begin{footnotesize}
\begin{equation*}
\begin{array}{rcl}
&
\!\!\!\!\!\!\!\!\!\!\!\!\!\!\!\!\!\!
{\overset{\scriptscriptstyle{
\eqref{LDMon}
}}{=}} 
&
\!\!\!\!\!\!
\displaystyle \sum_{i_1, i_2, i_3} m \Big\langle \phi_1, \big\langle \phi_2^{(1)}, 
\big\langle \psi_1 \phi_2^{(2)}, 
\langle \psi_2 \phi_2^{(3)} , {e_{i_3}} \rangle \blact  e_{i_2} \big\rangle \blact n_{(-1)} \big\rangle 
\blact {e_{i_1}}_{(1)}
\Big\rangle
\\
&& \qquad \qquad \qquad
\cdot_\emme ({e_{i_1}}_{(2)}n_{(0)}) \otimes_\ahha e^{i_1} \otimes_\ahha e^{i_2} \otimes_\ahha e^{i_3} 
\\
&
\!\!\!\!\!\!\!\!\!\!\!\!\!\!\!\!\!\!
{\overset{\scriptscriptstyle{
\eqref{Takeuchicoaction}, \eqref{pergolesi}, \eqref{bilet}
}}{=}} 
&
\!\!\!\!\!\!
\displaystyle \sum_{i_1, i_2, i_3} m 
\cdot_\emme 
\big\langle \phi_1, \langle \phi_2^{(1)}, n_{(-1)} \rangle
\blact {e_{i_1}}_{(1)}
\big\rangle
\\
&& \qquad \qquad \qquad
\lact \big(\big\langle \psi_1 \phi_2^{(2)}, 
\langle \psi_2 \phi_2^{(3)} , {e_{i_3}} \rangle \blact  e_{i_2} \big\rangle \blact {e_{i_1}}_{(2)}\big) n_{(0)} \otimes_\ahha e^{i_1} \otimes_\ahha e^{i_2} \otimes_\ahha e^{i_3} 
\\
&
\!\!\!\!\!\!\!\!\!\!\!\!\!\!\!\!\!\!
{\overset{\scriptscriptstyle{
\eqref{lesbrigittes2}, \eqref{LDMon}, \eqref{vetrorotto3}
}}{=}}
&
\!\!\!\!\!\!
\displaystyle \sum_{i_1, i_2, i_3} m 
\cdot_\emme
(n   \phi_2^{(1)})^{(0)} 
\Big\langle \phi_1 (n\phi_2^{(1)})^{(1)},  
\big\langle \psi_1 \phi_2^{(2)}, 
\langle \psi_2 \phi_2^{(3)} , {e_{i_3}} \rangle \blact  e_{i_2} \big\rangle \blact e_{i_1}
\Big\rangle
\\
&& \qquad \qquad \qquad
\otimes_\ahha e^{i_1} \otimes_\ahha e^{i_2} \otimes_\ahha e^{i_3} 
\\
&
\!\!\!\!\!\!\!\!\!\!\!\!\!\!\!\!\!\!
{\overset{\scriptscriptstyle{
\eqref{schizzaestrappa2}, \eqref{duedelue}
}}{=}}
&
\!\!\!\!\!\!
\displaystyle \sum_{i_2, i_3} m 
\cdot_\emme
(n   \phi_2^{(1)})^{(0)} 
\otimes_\ahha \phi_1 (n\phi_2^{(1)})^{(1)} \bract  
\big\langle \psi_1 \phi_2^{(2)} \bract \langle \psi_2 \phi_2^{(3)} , {e_{i_3}} \rangle, e_{i_2} \big\rangle
%\\
%&& \qquad \qquad \qquad
\otimes_\ahha e^{i_2} \otimes_\ahha e^{i_3} 
\\
&
\!\!\!\!\!\!\!\!\!\!\!\!\!\!\!\!\!\!
{\overset{\scriptscriptstyle{
\eqref{schizzaestrappa2}
}}{=}}
&
\!\!\!\!\!\!
m 
\cdot_\emme
(n   \phi_2^{(1)})^{(0)} 
\otimes_\ahha \phi_1 (n\phi_2^{(1)})^{(1)} 
\otimes_\ahha
\psi_1 \phi_2^{(2)} 
\otimes_\ahha
\psi_1 \phi_2^{(3)}, 
\end{array}
\end{equation*}
\end{footnotesize}
which is \eqref{maxdudler2} for this case.
%% Let $(m, \phi_1, \ldots, \phi_p) \in C^p_\co(U^*,M)$ and  $(n, \psi_1, \ldots, \psi_q) \in C^q_\co(U^*,M)$. Then one has for all 
%% $j = 1, \ldots, p$
%% \begin{footnotesize}
%% \begin{equation*}
%% \begin{array}{rcl}
%% && \xi^{-1}\big(\xi(m, \phi_1, \ldots, \phi_p) \circ_j \xi(n, \psi_1, \ldots, \psi_q)\big) 
%% \\
%% &
%% \!\!\!\!\!\!\!\!\!\!\!\!\!\!\!\!\!\!
%% {\overset{\scriptscriptstyle{
%% \eqref{mondrian1}, \eqref{mondrian2}, \eqref{maxdudler1}
%% }}{=}} 
%% &
%% \!\!\!\!\!\!
%% \displaystyle \sum_{i_1, \ldots, i_{p+q-1}} \Big( m \Big\langle \phi_1, \langle \ldots \langle \phi_i, n_{(-1)} \langle \psi_1 , e_j \rangle \lact      
%% {e_{i_j}}_{(1)} \cdots {e_{i_{p+q-1}}}_{(1)}
%% \rangle \ldots \blact {e_{i_1}}_{(1)} \Big\rangle \cdot_\emme \big({e_{i_1}}_{(2)} \cdots {e_{i_{j-1}}}_{(2)} n_{(0)} \big), e^{i_1}, \ldots, e^{i_{p+q-1}}\Big) 
%% \end{array}
%% \end{equation*}
%% \end{footnotesize}
With this property, the statement that $C(U,M)$ under the given conditions defines a cyclic operad with multiplication now follows from the respective property of $C_\co(U^*,M)$.
\end{proof}

Although we already know that  $(C^\bull(U,M), \gd_\bull,
 \gs_\bull, \tau)$ defines (under the mentioned assumptions) a cocyclic $k$-module, we still want to apply Theorem \ref{holl}, that is, \cite[Thm.~1.4]{Men:BVAACCOHA} to add the statement about Batalin-Vilkovisky algebras:

\begin{cor}
\label{pourquoilamuit}
\
\begin{enumerate}
\compactlist{99}
\item
Under the assumptions given in Theorem \ref{corsoumbertoI} (i), the cohomology groups $H^\bull(U,A)$ (resp.\ $\Ext^\bull_U(A,A)$ if $U_\ract$ is projective) form
a Batalin-Vilkovisky algebra.
\item
Under the assumptions given in Theorem \ref{corsoumbertoI} (ii), the cohomology groups $H^\bull(U,M)$ (resp.\ $\Ext^\bull_U(A,M)$ if $U_\ract$ is projective) form
a Batalin-Vilkovisky algebra.
% if $(U,A)$ is a left Hopf algebroid, $M$ a braided commutative Yetter-Drinfel'd algebra and $A$ is an anti Yetter-Drinfel'd contramodule.
\end{enumerate}
\end{cor}

\section{Examples and applications}
\label{lyra}

In this example section, we will first consider how to apply the results developed so far in the specific case of a bialgebroid resp.\ left Hopf algebroid that leads to the well-known Hochschild complex, and then treat some examples of algebras that allow for a cyclic structure on the Hochschild complex and hence for a BV algebra structure on Hochschild cohomology. In the last section, we shortly deal with Hopf algebras and recover the results from Menichi in \cite{Men:CMCMIALAM}.

\subsection{Batalin-Vilkovisky algebra structures on Hochschild cohomology}
% of associative algebras}
\label{scottex}

Recall the left Hopf algebroid $(U,A) := (\Ae, A)$ for an associative $k$-algebra $A$ with structure maps ({\it cf.}~\cite{Schau:DADOQGHA})
\begin{equation}
\label{umbria1}
s^\ell(a) = a \otimes 1, \quad t^\ell(b) = 1 \otimes b, \quad \gD_\ell(a \otimes b) = (a \otimes 1) \otimes_\ahha (1 \otimes b), \quad \gve(a \otimes b) = ab, 
%\ \mbox{along with} \linebreak
\end{equation}
along with 
\begin{equation}
\label{umbria2}
(a \otimes b)_+ \otimes_\Aopp (a \otimes b)_- = (a \otimes 1) \otimes_\Aopp (b \otimes 1)
\end{equation}
for all $a , b \in A$.

In the subsequent examples, we are particularly interested in the case in which $A$ itself is a right $\Ae$-contramodule resp.\ an aYD contramodule over $\Ae$. Let us first consider this specific situation: the aYD condition \eqref{nawas1} in case $M = A$ reads for $f \in \Hom_\Aopp(\Ae, A)$ and a right $\Ae$-contraaction $\gamma:  \Hom_\Aopp(\Ae, A) \to A$:
$$
(a \otimes b) (\gamma(f)) 
= \gamma\big(f((b \otimes 1)(-)(a \otimes 1)\big), 
$$
or, equivalently, 
\begin{equation}
\label{hofgarten1824}
a \lact \gamma(f) \ract b 
= \gamma\big(f( b \lact - \bract a)\big) \overset{\scriptscriptstyle{\eqref{carrefour4}, \eqref{passionant}}}{=} a \gamma(f) b,
\end{equation}
that is, coincides with Eq.~\rmref{romaedintorni}.
Observe here that Eq.~\rmref{hofgarten1824} in this case does not constitute a condition: whereas the right $A$-action on $M = A$ is fixed right from the beginning and the contraaction, if it exists, is modelled according to right $A$-linearity, 
simply {\em define} the left $A$-action so as to match \rmref{hofgarten1824}: this does not necessarily coincide with the canonical left action of the bialgebroid $\Ae$ on its base algebra $A$, that is, left and right multiplication (see below).
%and by \eqref{carrefour4} along with \eqref{passionant} 
It then follows that whenever $A$ is a right $\Ae$-contramodule, it is automatically an aYD contramodule. We obtain from Corollary \ref{pourquoilamuit}:

%\pagebreak

\begin{cor}
\label{stehtdranne}
If for a $k$-algebra $A$ there exists an $\Ae$-contraaction which is stable with respect to the induced left $\Ae$-action in the sense of \rmref{stablehalt},
then its Hochschild cohomology groups $H^\bull(A,A)$ (resp.\  $\Ext^\bull_\Aee(A,A)$ if $A$ is $k$-projective) form
a Batalin-Vilkovisky algebra.
\end{cor}

\begin{rem}
Observe that if $A$ is not stable in the aforementioned sense, the Hochschild cochain spaces still form a para-cyclic operad (with multiplication).
%, that is, one that fulfils all identities in \rmref{superfluorescent1} except $\tau^{n+1} = \id$. 
\end{rem}

For more general coefficients $M$, it is convenient to simplify the contramodule axioms in Definition \ref{schoenwaers} using the bialgebroid structure \eqref{umbria1} of $(\Ae, A)$: identifying $\Hom_\Aopp(\Ae, M) \simeq \Hom_k(A,M)$, we can rewrite the conditions \eqref{passionant}--\eqref{carrefour2} as follows: a right $\Ae$-contramodule is a right $A$-module $M$ together with a map 
$$
\gamma: \Hom_k(A,M) \to M
$$ 
such that
\begin{equation}
\label{federicoII}
\begin{array}{rcll}
\gamma\big(f(a(-))\big) &=& \gamma(f)a, & \forall \ f \in \Hom_k(A,M), \\
\dot\gamma\big(\ddot\gamma(g(\cdot \otimes \cdot\cdot))\big) &=& \gamma\big(g(\cdot \otimes 1_\ahha)\big), 
& \forall \ g \in \Hom_k(A \otimes A, M), \\
\gamma(m \id_\ahha(-)) &=& m, & \forall \ m \in M,
\end{array}
\end{equation}
where the first again simply expresses the fact that $\gamma$ be a right $A$-module morphism with respect to the right $A$-action
$
%\begin{equation}
fa := f(a(-)) 
%\end{equation}
$
for  $f \in \Hom_k(A,M)$ and $a \in A$,
and where again in the second line the dots over the maps are meant to match the respective argument. As before, these conditions imply that by means of
$$
am := \gamma(m \id_\ahha(-) a), \qquad a \in A, \ m \in M,
$$
there is also a left $A$-module structure on $M$
and then
the analogue of Eq.~\eqref{carrefour4} reads
$$
\gamma\big(f((-)a)\big) = a\gamma(f), \qquad f \in \Hom_k(A,A), 
$$
that is, $\gamma$ is an $A$-bimodule map. 

As above, the aYD conditions \eqref{romaedintorni}--\eqref{nawas1} are trivially fulfilled once a contraaction is found. The stability \eqref{stablehalt} now becomes
\begin{equation}
\label{stablenochmal}
\gamma((-)m) = m,
\end{equation}
where as before we denote $(-)m \colon a \mapsto am$ as a map in $\Hom_k(A,M)$.

%
% We want to add the almost tautological remark that when twisting the left $A$-module structure by an 
% endomorphism $\gs \in \End_k(A)$, that is, 
% $
% am := \gamma(m \id_\ahha(-) \gs(a)),
% $
% the contraaction is still an $A$-bimodule map when dealing with the twisted left $A$-action 
% \begin{equation}
% \label{norica2}
% a \lact f := f((-)\gs(a)),  \quad f \in \Hom_k(A,M), \ a \in A,
% \end{equation} 
% which is the situation we will be dealing with in the applications below.
% \footnote{Moment ... below $M=k$, but $k$ is not an $A$-module, hence non c'entra nulla ... the notation with 
% the white etriangles is also confusing here}

With the formulae in \eqref{umbria1}, it is easy to see that one obtains 
\begin{equation}
\label{pannacotta}
C^\bull(\Ae,M) \simeq \Hom_k(A^{\otimes \bull},M) =: C^\bull(A,M), 
\end{equation}
that is, one obtains the conventional Hochschild complex for a left $\Ae$-module (resp.\ $A$-bimodule) $M$.
It is also quite straightforward to see that
the cosimplicial structure in \eqref{anightinpyongyang1} reduces to the well-known one from \cite{Hoch:OTCGOAAA}, up to the sign $(-1)^{n+1}$. The cocyclic operator from \eqref{anightinpyongyang1} in this case then reads 
\begin{equation}
\label{russischeinfachlernen}
(\tau f)(a_1, \ldots, a_n) = \gamma\big(a_1 f(a_2, \ldots, a_n, -) \big)
\end{equation}
for $f \in C^\bull(A,M)$, using \eqref{umbria2} along with a contraaction as in \eqref{federicoII}.

\begin{rem}
\label{spina}
As already remarked by Connes \cite{Con:NCDG}, it is a priori not clear how to define a cocyclic operator on the Hochschild complex $C^\bull(A,A)$ for an arbitrary associative algebra with coefficients in the algebra itself. To circumvent this problem, cyclic cohomology was defined to be the cyclic cohomology of the complex $C^\bull(A, A^*)$, where $A^* := \Hom_k(A,k)$. One then has $C^\bull(A, A^*) \simeq C^{\bull+1}(A,k)$ and the cocyclic operator is essentially the pull-back of the cyclic operator on the Hochschild homology complex, that is, cyclic permutation. See, {\it e.g.}, \cite[\S2.1.0]{Lod:CH} and also \S1.5.5 in {\it op.~cit.}~for further comments on this with respect to functoriality. However, Eq.~\eqref{russischeinfachlernen} does define indeed a cocyclic operator on  $C^\bull(A,A)$ in case the algebra in question is equipped with an extra structure.   
\end{rem}

\subsubsection{Symmetric algebras}
\label{castelnuovo1}
This subsection could be of course integrated in the subsequent one about Frobenius algebras, but since in this case the argument is notably simpler, we decided to present it separately:
recall from, {\it e.g.}, \cite[\S16F]{Lam:LOMAR} or \cite{EilNak:OTDOMAAIIFAAQFR} that an algebra is called {\em symmetric} (in the sense of representation theory) if there is an isomorphism $A \to A^* := \Hom_k(A,k)$ of $A$-bimodules, where the $A$-bimodule structure on $A^*$ is given by
\begin{equation}
\label{norica1}
a \lact \phi \ract b := \phi(b(-)a) = \langle \phi, b(-)a \rangle
\end{equation}
for $a, b \in A, \ \phi \in A^*$, where $\langle . , . \rangle$ denotes the canonical pairing between $A$ and $A^*$ given by evaluation.

Hence, an $\Ae$-contramodule structure 
$
\Hom_k(A, A) \to A
$
on $A$ amounts to a map 
$$
\Hom_k(A, A^*) \simeq \Hom_k(A \otimes A, k ) \to A^*
$$
subject to \eqref{federicoII}, and it is a simple check that the assignment
\begin{equation}
\label{sahnetorte}
\gamma: \Hom_k(A \otimes A, k ) \to A^*, \quad g \mapsto g(- \otimes 1_\ahha)
\end{equation}
gives such a map such that $A$ is stable over $\Ae$ in the sense of \rmref{stablenochmal}.
We therefore have

\begin{cor}
\label{dresden}
The Hochschild cohomology 
%$H^\bull(A, A)$ (resp.\ $\Ext_\Ae(A,A)$ when $A$ is $k$-projective) 
of a symmetric algebra 
%$A$ 
is a Batalin-Vilkovisky algebra.
\end{cor}

This was the result obtained in \cite{Tra:TBVAOHCIBIIP} and later in \cite[\S4]{Men:BVAACCOHA} as well as \cite{EuSche:CYFA}.

\subsubsection{Frobenius algebras}
\label{castelnuovo3}
In this section, let $k$ be a field, which, for simplicity, we assume to be algebraically closed (if not, one can generalise the subsequent considerations along the lines in \cite[\S4]{LamZhoZim:THCROAFAWSSNAIABVA}).

For Frobenius algebras, there exists a considerable amount of equivalent definitions, most of which are listed in, for example, \cite{Str:FMAP}. We use here the following formulation and give a few well-known details that are needed in the sequel:

 \begin{dfn}
 A {\em Frobenius algebra} is an algebra $A$ with a 
 functional 
 $ \varepsilon : A \rightarrow k$ 
 such that the map $A \rightarrow A^*, a \mapsto
 \varepsilon_a$ with $ \varepsilon_a(b):=\varepsilon(ab)$ 
 is bijective. The functional $ \varepsilon $
 is called a {\em Frobenius functional}.
 \end{dfn}

 One can show that not only $A$ is finite-dimensional, but also that on any Frobenius algebra there is a $k$-coalgebra structure the coproduct of which is an $A$-bimodule map and the counit of which is given by the Frobenius functional: to this end, let $\{e_i\}_{1 \leq i \leq n}  \in A$ be a basis and define another basis $\{e^i\}_{1 \leq i \leq n}  \in A$ by means of $\gve(e^j e_i) = \delta^j_i$. Set
 $$
 	\Delta (1)=\textstyle\sum_{i}  e_i \otimes e^i,
 $$
 and by $A$-bilinearity we have for all $a \in A$:
 \begin{equation}
 \label{sonne}
 	\Delta (a)=\textstyle\sum_{i} ae_i \otimes e^i = \textstyle\sum_{i} e_i \otimes e^ia.
 \end{equation}
Counitality amounts to
 \begin{equation}
 \label{mond1}
 	a=\textstyle\sum_{i} \varepsilon (ae_i)e^i=
 	\textstyle\sum_{i} \varepsilon (e^ia)e_i
 \end{equation}
 for all $a \in A$, and in particular
 \begin{equation}
 \label{sterne}
 1 =\textstyle\sum_{i} \varepsilon (e_i)e^i=
 	\textstyle\sum_{i} \varepsilon (e^i)e_i.
 \end{equation}
 The map $a \mapsto \gve_a$ is right $A$-linear, where the right action on $A^*$ is given by $\phi \otimes a \mapsto \phi \ract a := \phi(a(-))$ and
 the right $A$-action on $A$ is right multiplication; hence, $A \simeq A^*$ as right $A$-modules. This, in turn, implies that  $a \mapsto \gve_a$ also becomes a map of left $A$-modules, where the left action on $A^*$ is given by $a \otimes \phi \mapsto a \lact \phi:= \phi((-)a)$, whereas
 the left $A$-action on $A$ is twisted by an automorphism $\gs \in \Aut(A)$, that is $a \otimes b \mapsto \gs(a)b$. Therefore,  
 $ {}_\gs A \simeq A^*$ as $A$-bimodules resp.\ left $\Ae$-modules. The automorphism $\gs$ is determined up to inner automorphisms and called the {\em Nakayama automorphism}. This leads to the identity $\gve(ab) = \gve(\gs(b)a)$, and also
$
\gve(e_ia)e^i = \textstyle \sum_i \gve(\gs(a)e_i) e^i = \gs(a) = \sum_i \gve(e^i \gs(a)) e_i.
$
Hence, 
\begin{equation}
\label{mond2}
	\Delta (\gs(a))=\textstyle\sum_{i} e_ia \otimes e^i = \textstyle\sum_{i} e_i \otimes \gs(a) e^i.
\end{equation}

% \rmvarthm{\bf Convention}
% \begin{equation*}
% \begin{split}
% \gve(ab) &= \gve(\gs(b)a)
% \end{split}
% \end{equation*}

\begin{lem}
\label{dasnochhier}
Let $A$ be an arbitrary Frobenius algebra % with the notation introduced above, 
and let ${}_\gs A$ be $A$ as a $k$-module but with the twisted left $\Ae$-action given by $(a \otimes b)c = a \lact c \ract b := \gs(a)cb$. Then the assignment
$$
\gamma: \Hom_k(A, {}_\gs A) \to {}_\gs A, \quad f \mapsto \textstyle\sum_i \gve\big(f(e_i)\big) e^i
$$
defines a right $\Ae$-contraaction on ${}_\gs A$.
\end{lem}

\begin{proof}
We have to check the three identities
$$
\gamma\big(f(a(-))\big) = \gamma(f)a, 
\quad \dot\gamma\big(\ddot\gamma(g(\cdot \otimes \cdot\cdot))\big) = \gamma\big(g(\cdot \otimes 1_\ahha)\big), 
\quad
\gamma(a \id_\ahha(-)) = a,
$$
for all $a \in A$, where $f \in \Hom_k(A,A)$ and $g  \in \Hom_k(A \otimes A, A)$.
The first and the third identity both follow from \rmref{sonne}. As for the second, compute
\begin{equation*}
\begin{split}
\dot\gamma\big(\ddot\gamma(g(\cdot \otimes \cdot\cdot))\big) 
&= \textstyle\sum_i \gamma\big(\gve\big(g((-) \otimes e_i)\big)e^i\big) \\
&= \textstyle\sum_{i,j} \gve\Big(\gve\big(g(e_j \otimes e_i)\big)e^i\Big)e^j \\
&= \textstyle\sum_{i,j} \gve(e^i) \gve\big(g(e_j \otimes e_i)\big)  e^j \\
&= \textstyle\sum_{i,j}  \gve\big(g(e_j \otimes \gve(e^i) e_i)\big)  e^j \\
&= \textstyle\sum_{j}  \gve\big(g(e_j \otimes 1_\ahha)\big)  e^j \\
&= \gamma\big(g(\cdot \otimes 1_\ahha)\big), 
\end{split}
\end{equation*}
where we used the fact that $\gve$ lands in $k$ and is therefore central as well as \rmref{sterne} in the penultimate step.
%using the right $A$-module structure $f.a := f(a(-))$ for $f \in \Hom_k(A,A)$.
\end{proof}

Let us repeat that the contraaction is by construction automatically left $\Ae$-linear with respect to the twisted action on ${}_\gs A$, as is also directly seen using \rmref{sonne} and \rmref{mond2}. Hence, the Hochschild cochain spaces for an arbitrary Frobenius algebra constitute a para-cyclic operad with multiplication.
However, using the same Eq.~\rmref{mond2}, one immediately sees that 
\begin{equation}
\label{selbstbedienung}
\gamma((-)a) = \gs(a), \qquad \forall a \in A,
\end{equation}
that is, the contramodule ${}_\gs A$ in general is not stable. To obtain cyclicity and thereby the desired Batalin-Vilkovisky algebra structure on Hochschild cohomology, one has to restrict the class of Frobenius algebras under consideration.

\subsubsection{The case of semisimple $\gs$}
 In the following (similar to \cite{KowKra:BVSOEAT, LamZhoZim:THCROAFAWSSNAIABVA} but more direct as the cocyclic operator $\tau$ on the Hochschild cochain spaces is directly available and we do not need to pass through duals), we restrict to the case in which $\gs$ is diagonalisable (semisimple), meaning that there is a subset $\gS \subseteq k \backslash \{0\}$ and a decomposition of $k$-vector spaces of the form
 $$
 A =\bigoplus_{\lambda \in \Sigma} A_\lambda, \quad A_\lambda = \{a \in A \mid \sigma (a)=\lambda a\}.
 $$
 Note that $1 \in \Sigma $ because $\sigma (1)=1$ and that $A_\lambda A_\mu \subseteq A_{\lambda\mu}$. Denote by $\langle \Sigma \rangle$ the monoid generated by $\Sigma$.
Extending Corollary \ref{dresden}, we obtain:

\begin{cor}
For a Frobenius algebra over a field with diagonalisable Nakayama automorphism, its Hochschild cohomology is a Batalin-Vilkovisky algebra.
\end{cor}

\begin{proof}
As in \cite[p.~57]{KowKra:BVSOEAT}, write
$$
        C^p(A,A)_\lambda := \{ \varphi \in C^p(A,A)
        \mid 
        \varphi ((A^{\otimes_k p})_\mu) \,{\subseteq}\,
        A_{\lambda \mu}\}
$$ 
for the set of all $\lambda$-homogeneous
$p$-cochains, where 
$$
        (A^{\otimes_k p})_\mu :=
        \bigoplus_{\mu_1,\ldots,\mu_p
          \in \langle \Sigma \rangle \atop \mu_1 \cdots \mu_p=\mu} 
        A_{\mu_1} \otimes_k \cdots \otimes_k A_{\mu_p}
$$ 
denotes the homogeneous component of elements of $A^{\otimes_k p}$ of total degree  $ \mu \in \langle \Sigma \rangle$. One has an isomorphism 
\begin{equation}
\label{rfi}
        C^\bull(A,A) \simeq C^\bull(A,A)_1 \oplus 
        \Bigl(C^\bull(A,A) \cap \textstyle \prod_{\lambda \in k \setminus \{0,1\}} 
        C^\bull(A,A)_\lambda\Bigr),
\end{equation}
which splits off the homogeneous degree $1$ part. As in \cite[Cor.~3.8]{LamZhoZim:THCROAFAWSSNAIABVA}, one can show that the restriction $\gb_\lambda$ of the Hochschild differential $\gb$ to the $\lambda$-homogeneous $p$-cochains maps them into the $\lambda$-homogeneous $(p+1)$-cochains; accordingly, set $H^\bull(A,A)_\lambda := H\big(C^\bull(A,A)_\lambda, \gb_\lambda \big)$. On the other hand, for a para-cocyclic object, one has the general identity $\beta B + B \beta = \id - \tau^{n+1}$ in degree $n$ (see, for example, \cite[Eq.~(2.18)]{KowKra:BVSOEAT}), and by means of \rmref{selbstbedienung}, one observes that restricting this identity to $C^\bull(A,A)_\lambda$ reads
 $\beta_\gl B + B \beta_\gl = (1- \gl)\id$. Hence, for $\gl \neq 1$ the operator $B$ is a contracting homotopy for  $C^\bull(A,A)_\lambda$ and therefore  $H^\bull(A,A)_\lambda = 0$. By the splitting \rmref{rfi}, we then obtain $H^\bull(A,A)_1 \simeq H^\bull(A,{}_\gs A)$ and since by Lemma \ref{dasnochhier} along with \rmref{selbstbedienung} we know that the cochain spaces $C^\bull(A,A)_1$ yield a cyclic operad with multiplication, the statement is proven using Corollary \ref{stehtdranne}. 
\end{proof}

This was the main result obtained in \cite{LamZhoZim:THCROAFAWSSNAIABVA} and independently also in \cite{Vol:BVDOHCOFA}. We refer again to  \cite[\S5]{LamZhoZim:THCROAFAWSSNAIABVA} and in particular Proposition~5.7 in {\it op.~cit.}~for a large amount of examples resp.\ classes of algebras where the semisimplicity condition for the Nakayama automorphism is met, and also for a couple of examples where this is {\em not} the case. However, up to our knowledge it is not proven (or disproven) whether in the latter situation the existence of a BV structure on Hochschild cohomology is automatically excluded for some reason.

\subsection{Hopf algebras}
\label{toisondor}
If $U := H$ is a Hopf algebra over $A := k$, then the theory developed in \S5 notably simplifies since there is a canonical contraaction on any $k$-module $M$ simply given by evaluation on $1 \in H$; remember that this is in general not allowed, see the comment after \eqref{passionant} and is parallel to the fact that not every $A$-module can be given the structure of a trivial {\em co}\/module over a bialgebroid, in contrast to the bialgebra case. In case the Hopf algebra possesses further grouplike elements (apart from $1$), that is, elements $\varsigma \in H$ fulfilling $\gD(\varsigma) = \varsigma \otimes \varsigma$ and $\gve(\varsigma) = 1$, one can even make the following simple observation:

\begin{lem}
Let $H$ be a Hopf algebra over $k$. Then every $k$-module $M$ carries a trivial $H$-contraaction given by
$$
\gamma: \Hom_k(H,M) \to M, \quad f \mapsto f(\varsigma)
$$
for a grouplike element $\varsigma \in H$.
\end{lem}

This explains why when defining a cocyclic structure on $C^\bull(H,M)$ or $C^\bull(H,k)$ as for example done in \cite{Men:CMCMIALAM}, 
the contraaction becomes somehow ``invisible''. In fact, the cocyclic operator from \eqref{anightinpyongyang1} now reads
$$
(\tau f)(h^1, \ldots, h^n) = 
f\big(h^2_{(1)}, \ldots, h^n_{(2)}, S(h^n_{(2)}) \cdots S(h^2_{(2)}) S(h^1)\varsigma\big),
$$
where we used $h_+ \otimes h_- = h_{(1)} \otimes S(h_{(2)})$, the bialgebra counitality as well as the left action \eqref{gianduiotto1a} for $A = k$ along with the $k$-linearity of the bialgebra counit. 

In case $M = k$, as for the anti Yetter-Drinfel'd contramodule structure on $k$, the condition \eqref{nawas1} now reads
$$
\gve(h)f(\varsigma) = h\big(\gamma(f)\big) =  \gamma \Big(\gve(h_{(2)}) f\big(S(h_{(3)})(-)h_{(1)}\big) \Big) = f\big(S(h_{(2)})\varsigma h_{(1)}\big),
$$
and using $S(\varsigma)\varsigma = 1 = S^{-1}(\varsigma)\varsigma$ one directly checks that $k$ is a stable aYD contramodule (relative to the grouplike element $\varsigma$, denoted by $k_\varsigma$) if $S^2(h) = \varsigma h \varsigma^{-1}$. 
As a consequence, we recover Theorem 50 in \cite{Men:CMCMIALAM}:

\begin{cor}
Let $H$ be a Hopf algebra over $k$ and $\varsigma$ a grouplike element such that the antipode is a twisted involution, that is, $S^2(h) = \varsigma h \varsigma^{-1}$ for all $h \in H$. Then $H^\bull(H,k_\varsigma)$ resp.\ $\Ext^\bull_H(k,k_\varsigma)$ is a Batalin-Vilkovisky algebra.
\end{cor}

\providecommand{\bysame}{\leavevmode\hbox to3em{\hrulefill}\thinspace}
\providecommand{\MR}{\relax\ifhmode\unskip\space\fi M`R }
% \MRhref is called by the amsart/book/proc definition of \MR.
\providecommand{\MRhref}[2]{%
  \href{http://www.ams.org/mathscinet-getitem?mr=#1}{#2}
}
\providecommand{\href}[2]{#2}

\end{document}